\documentclass{amsart}
\usepackage[latin1]{inputenc}
\usepackage{amssymb}
\usepackage{amsmath}
\usepackage{enumerate}
\usepackage{epsfig,color,xcolor}
\usepackage{longtable}
\allowdisplaybreaks[4]
\usepackage[english]{babel}
\usepackage[all]{xy}
\usepackage{fullpage}
\usepackage{pgf,tikz}
\usepackage{mathrsfs}
\usetikzlibrary{arrows}
\usepackage{hyperref}
\usepackage{graphicx}
\allowdisplaybreaks
\newtheorem{theorem}{Theorem}[section]
\newtheorem{lemma}[theorem]{Lemma}
\newtheorem{prop}[theorem]{Proposition}
\newtheorem{proposition}[theorem]{Proposition}
\newtheorem{corollary}[theorem]{Corollary}

\newtheorem{definition}[theorem]{Definition}
\newtheorem{example}[theorem]{Example}

\newtheorem{rem}[theorem]{Remark}
\newtheorem{remark}[theorem]{Remark}

\numberwithin{equation}{section}
\newcommand{\rk}{{\rm rank}}

\newcommand{\ra}{\rightarrow}

\newcommand{\PP}{ \mathbb{P}}
\newcommand{\C }{ \mathbb{C}}

\newcommand{\Z}{\mathbb{Z}}
\newcommand{\Q}{\mathbb{Q}}

\newcommand{\N}{\mathbb{N}}

\usepackage{imakeidx}

\newtheoremstyle{dico}% name of the style to be used
{\baselineskip}   % ABOVESPACE
{\topsep}   % BELOWSPACE
{}  % BODYFONT
{0pt}       % INDENT (empty value is the same as 0pt)
{} % HEADFONT
{.}         % HEADPUNCT
{5pt plus 1pt minus 1pt} % HEADSPACE
{}          % CUSTOM-HEAD-SPEC
\theoremstyle{dico}

\makeatletter
\def\blfootnote{\xdef\@thefnmark{}\@footnotetext}

\author{Alice Garbagnati}
\address{Dipartimento di Matematica, Univ. Statale di Milano, Milan, Italy}
\email{alice.garbagnati@unimi.it}
\urladdr{ https://sites.google.com/site/alicegarbagnati/}

\title[Specializations of symplectic and van Geemen--Sarti involutions on K3 surfaces]{Specializations of symplectic and van Geemen--Sarti involutions on K3 surfaces}
\begin{document}


	\maketitle
	\begin{abstract} Given a symplectic involution $\iota$ on a K3 surface $X$, the desingularization $Y$ of $X/\iota$ is still a K3 surface, which in general has a different N\'eron--Severi group. Nevertheless, if the involution is induced by the translation by a 2-torsion section on an elliptic fibration (i.e. it is a van Geemen--Sarti involution) and the Picard number is minimal, the N\'eron--Severi groups of $X$ and $Y$ are known to be isometric.	We first determine infinitely many codimension 2 subfamilies of projective K3 surfaces with a symplectic involution (not of van Geemen--Sarti type) whose generic members satisfy $NS(X)\simeq NS(Y)$. Then, we describe the cohomological action of a van Geemen--Sarti involution and we characterize specializations of K3 surfaces with a van Geemen--Sarti involution for which it is still true that $NS(X)\simeq NS(Y)$. There is a 5-dimensional family of K3 surfaces with van Geemen--Sarti involution for which $X\simeq Y$. The K3 surfaces in such a family admit complex multiplication, and we describe its cohomological action. 
	
	We briefly discuss similar problems for order 3 symplectic automorphisms induced by a translation by a 3-torsion section on an elliptic fibration.\end{abstract}
	
	\section{Introduction}

	In the study of K3 surfaces, the symplectic automorphisms of finite order play a very important role, since they create a relation between different K3 surfaces. Indeed, if $\alpha$ is a finite order symplectic automorphism on a K3 surface $X$, then the minimal resolution of $X/\alpha$ is another K3 surface, always denoted in the following as $Y$. A priori $X$ and $Y$ are different K3 surfaces and generically also their N\'eron--Severi groups (and transcendental lattices) are not isometric. Nevertheless, there are special choices for $X$ and $\alpha$ which produce K3 surfaces $Y$ for which $NS(X)\simeq NS(Y)$, and even more specific choices for which $X$ and $Y$ are isomorphic. The first phenomenon appears, for example, for generic K3 surfaces admitting an elliptic fibration with an order $n$ torsion section: the translation by the torsion section is a symplectic automorphism of the same order, and if $X$ is generic among the K3 surfaces admitting the required torsion section, then $NS(X)\simeq NS(Y)$ (see \cite[Proposition 4.2]{vGS} for $n=2$ and \cite[Proposition 4.3]{G} for $n>2$). The second appears for certain specific K3 surfaces which admit a 2-torsion section, see \cite[Section 6.1]{vGSc} and Proposition \ref{prop: subfamily isomorphic quotients}. The property that $NS(X)\simeq NS(Y)$ is preserved if $X$ is the general member of some special subfamilies of the K3 surfaces admitting an elliptic fibration with an $n$-torsion section, but it is false for most of the subfamilies.
	
	We now focus on the case of the involutions (even if in the last section we briefly consider the order 3 case): the involution given by the translation by a 2-torsion section on an elliptic fibration is classically known as van Geemen--Sarti involution. The K3 surfaces admitting such an involution are studied in many papers, exactly in view of the isometry existing between the N\'eron--Severi groups of the K3 surface admitting a van Geemen--Sarti involution and the one of the quotient by such an involution, see e.g. \cite{vGS} (the original paper by B. van Geemen and A. Sarti), \cite{CD} (where the name ``van Geemen--Sarti involution" is introduced), \cite{CD2}, \cite{CM1}, \cite{CM2}, \cite{CM3}, \cite{CG}, \cite{vGSc}. Notice that in many of these papers the authors are interested in subfamilies of the family of K3 surfaces admitting a van Geemen--Sarti involution, such that the property $NS(X)\simeq NS(Y)$ is preserved for the generic members of the subfamilies.

	The aim of this paper is twofold: on one hand we provide countable many different families of K3 surfaces with a symplectic involution for which $NS(X)\simeq NS(Y)$ on the generic member, showing that this property is not exclusive of the generic K3 surface admitting a van Geemen--Sarti involution; on the other hand, we are interested in the specializations of the K3 surfaces which admit a van Geemen--Sarti involution for which the property $NS(X)\simeq NS(Y)$ remains true on generic members. We will see that the property is surely not preserved on the  codimension 1 subfamilies, but there exist infinitely many higher codimensional subfamilies for which it remains true. \\
	
	The K3 surfaces admitting a symplectic involution are generically non projective, but the countably many 11-dimensional components of the family of the projective ones are well known and described in \cite[Propositions 2.2 and 2.3]{vGS}. The same holds true for the family of the projective K3 surfaces which are obtained as desingularization of the quotient of another K3 surface by a symplectic involution, see \cite[Corollary 2.3]{GS}. In particular, let $X$ be a K3 surface admitting a symplectic involution $\iota$ and, as above, $Y$ the desingularization of the quotient $X/\iota$, then (see Corollary \ref{rem: NS(X)neq NS(Y) in the generic cases}):
	\begin{itemize} \item $\rho(X)=\rho(Y)\geq 8$ and if $\rho(X)=\rho(Y)=8$ then $NS(X)\not\simeq NS(Y)$;
		\item if $X$ (or equivalently $Y$) is projective, then $\rho(X)=\rho(Y)\geq 9$ and if $\rho(X)=\rho(Y)=9$ then $NS(X)\not\simeq NS(Y)$.
		\end{itemize}
	Therefore, if $X$ is a projective K3 surface with a symplectic involution such that $NS(X)\simeq NS(Y)$, then $\rho(X)=\rho(Y)\geq 10$.  The family of the K3 surfaces admitting a van Geemen--Sarti involution provides an example of K3 surfaces $X$ such that $NS(X)\simeq NS(Y)$ with minimal Picard number, i.e. with  $\rho(X)=\rho(Y)= 10$ (recall that the minimality of the Picard number is equivalent to the maximality of the dimension of the family). In Section \ref{sec: K3 surfaces with involution and their specializations}, we will show that this is not a sporadic example, indeed in Theorem \ref{theo: NS(X)=NS(Y) rank 10}, we prove the following. 	
	\begin{theorem} There are infinitely many 12-dimensional families of projective K3 surfaces whose generic member, $X$, admits a symplectic involution $\iota$ and $NS(Y)\simeq NS(X)$, where $Y$ is the minimal resolution  of $X/\iota$.
	\end{theorem}
	The proof of this result is purely lattice theoretic and we do no have geometric description of the involution which makes clear the reason for which $X$ and $Y$ have the same N\'eron--Severi group. A geometric example of the surface $X$ and of its quotient for a specific family (the one with a polarization of the lowest possible degree) is given in Subsection \ref{subsection: geometry d=1}.\\

	After this general result, we concentrate on the family $\mathcal{N}$ of the K3 surfaces which admit a van Geemen--Sarti involution, i.e. on the family of the K3 surfaces with an elliptic fibration $\mathcal{E}$ admitting a 2-torsion section $T$. As already observed, the general member of this family satisfies the condition $NS(X)\simeq NS(Y)$ and this can be proved showing that $\mathcal{E}$ induces an elliptic fibration on the quotient surface $Y$ with properties analogous to the ones of $\mathcal{E}$. 
	
	Many special members of the family $\mathcal{N}$, do not satisfy the condition $NS(X)\simeq NS(Y)$, but it always remains true that if $X\in\mathcal{N}$ then $Y\in\mathcal{N}$, even if their N\'eron--Severi group are not the same.	It is natural to ask what is the relation between the N\'eron--Severi groups of $X$ and $Y$ if $X\in\mathcal{N}$ and in particular to ask for conditions such that they are isometric.
	
	The advantage of considering these questions specifically for the family $\mathcal{N}$ (instead of the other families mentioned in the previous theorem), is that both the cohomological and geometric action of a van Geemen--Sarti involution can be explicitly described, as done in Section \ref{sec: family N and vGS inv}, and hence it is relatively easy to study their specializations by considering specializations of the elliptic fibration $\mathcal{E}$.
	
	In Section \ref{sec: K3 with vGS rank 11} we classify all the codimension 1 subfamilies of $\mathcal{N}$, describing their N\'eron--Severi groups (see Theorem \ref{theorem: X rho 11}) and the geometric properties of the elliptic fibration on $X$ and of the elliptic fibration induced on $Y$ by the quotient structure. So, we prove the following (where $MW$ is the Mordell--Weil group of an elliptic fibration and $MWL$ is the Mordell--Weil lattice).
	\begin{theorem}{\rm (see Propositions \ref{prop: specialization elliptic fibrations, 11} and \ref{proposition E11 and F11})}
	Let $X_{11}$ be a K3 surface admitting a van Geemen--Sarti involution $\sigma$ and with $\rho(X_{11})=11$. Let $\mathcal{E}_{11}$ be the elliptic fibration on $X_{11}$ which admits the 2-torsion section inducing $\sigma$, $Y_{11}$ the minimal resolution of $X_{11}/\sigma$ and $\mathcal{F}_{11}$ the elliptic fibration induced by $\mathcal{E}_{11}$ on $Y_{11}$. Then $\mathcal{E}_{11}$ and $\mathcal{F}_{11}$ satisfy exactly one of the following
		\begin{enumerate}
				\item The singular fibers of $\mathcal{E}_{11}$ are $9I_2+6I_1$, $MW(\mathcal{E}_{11})=\Z/2\Z$ and 
				the singular fibers of $\mathcal{F}_{11}$ are $I_4+6I_2+8I_1$, $MW(\mathcal{F}_{11})=\Z/2\Z$;
				 
				\item The singular fibers of $\mathcal{E}_{11}$ are $I_4+6I_2+8I_1$, $MW(\mathcal{E}_{11})=\Z/2\Z$ and
				the singular fibers of $\mathcal{F}_{11}$ are $9I_2+6I_1$, $MW(\mathcal{F}_{11})=\Z/2\Z$;
				
				\item The singular fibers of $\mathcal{E}_{11}$ are $8I_2+8I_1$, $MW(\mathcal{E}_{11})=\Z\times \Z/2\Z$, $MWL(\mathcal{E}_{11})=[2d]$ for $d\in\mathbb{N}_{\geq 0}$ and
				the singular fibers of $\mathcal{F}_{11}$ are $8I_2+8I_1$, $MW(\mathcal{F}_{11})=\Z\times \Z/2\Z$, $MWL(\mathcal{F}_{11})=[d/2]$, $d\in\mathbb{N}_{\geq 0}$.
				\item The singular fibers of $\mathcal{E}_{11}$ are $8I_2+8I_1$, $MW(\mathcal{E}_{11})=\Z\times \Z/2\Z$, $MWL(\mathcal{E}_{11})=[d/2]$ for $d\in\mathbb{N}_{\geq 0}$ and
				the singular fibers of $\mathcal{F}_{11}$ are $8I_2+8I_1$, $MW(\mathcal{F}_{11})=\Z\times\Z/2\Z$, $MWL(\mathcal{F}_{11})=[2d]$, $d\in\mathbb{N}_{\geq 0}$.
		\end{enumerate}
	\end{theorem}
	
	One observes that if $X_{11}$ is as in the previous theorem, its N\'eron--Severi group is never isometric to the one of $Y_{11}$.
	
	In Section \ref{sec: K3 with vGS rank 12} we consider more restrictive conditions on $X$, indeed we consider codimension 2 subfamilies of $\mathcal{N}$ and in particular we prove that there are countably many subfamilies which satisfy the condition $NS(X)\simeq NS(Y)$. 
	We also discuss the intersection of these families with the family of K3 surfaces with van Geemen--Sarti involution for which $X\simeq Y$. In particular we prove the following.
	
	\begin{theorem}\label{thm intro: rank 12 vGS}{\rm (See Theorem \ref{theorem: rank 12, equal NS}) }There are infinitely many subfamilies of $\mathcal{N}$ with codimension 2 whose generic member $X_{12}$, admits a van Geemen--Sarti involution $\sigma$ such that $NS(Y_{12})\simeq NS(X_{12})$, where $Y_{12}$ is the minimal desingularization of $X_{12}/\sigma$.\end{theorem}
	
	In Section \ref{sec:Following specializations of $X$ to a K3 surface of Picard number 20} we show two of our main results. They are: the description of the complex multiplication $\nu$ on the transcendental lattice of the K3 surfaces lying in the 5-dimensional family of K3 surfaces with a van Geemen--Sarti involution such that $X\simeq Y$  and a lattice theoretic sufficient condition that implies that a K3 surface $X$ (possibly with $\rho(X)>10$) with a van Geemen--Sarti involution is such that $NS(X)\simeq NS(Y)$. Both these results are proven by considering a K3 surface $X$ whose N\'eron--Severi group is $U\oplus N$ such that $X\simeq Y$ and by constructing following specializations of $X$ in such a way that the property $X\simeq Y$ (and consequentially also $NS(X)\simeq NS(Y)$) is preserved by each specialization.  Then the (rational) quotient map $\pi:X\dashrightarrow  Y$ induces a self map of $(T_X)_{\Q}$ which restricts to $\gamma,$ a specific abstract action a lattice $\Gamma$, where $\Gamma$ is a prescribed lattice embedded with a finite index in  $U\oplus U\oplus N\simeq T_X$. 
	
	The lattice theoretic criterion that we prove is then the following:
	
	\begin{theorem}\label{theorem intro order 2} A K3 surface $Z$ with a van Geemen--Sarti involution $\sigma$ is such that $NS(Z)\simeq NS(W)$, where $W$ is the minimal resolution of $\widetilde{Z/\sigma}$, if $NS(Z)$ (or equivalently $T_Z$) is preserved by the map $\gamma$ described above. %or, equivalently, if $T_Z$ is preserved by $\gamma$.
	\end{theorem}
	 
    The maps $\nu$ and $\gamma$ are described in Theorem \ref{theorem: action selfmap} and the criterion is given in Corollary \ref{cor: NSX=NSY condition}.
    
	The description of $\gamma$ is due to a careful analysis of the action of the quotient map by a van Geemen--Sarti involution on specific specializations of elliptic fibrations, which are obtained by ``gluing" two fibers of type $I_h$ (for certain $h$) to a fiber of type $I_{2h}$ and by analysing the impact of this gluing on the quotient map. The specializations considered are summarized in Proposition \ref{prop: the five specializations} and shown in Figure \ref{Figure speicalization}.
		
	In Section \ref{sec: order 3}, we consider the analogous problems for K3 surfaces admitting elliptic fibrations with a 3-torsion section and we prove the analogue of Theorems \ref{thm intro: rank 12 vGS} and \ref{theorem intro order 2} for these surfaces, see Theorems \ref{thm order 3 rank 16} and \ref{thm: order 3 "complex multiplication"} and Corollary \ref{cor: order 3}.\\
	
	{\bf Acknowledgements. }{\it\ I warmly thank Bert van Geemen for many useful discussions, for all his suggestions and for reading a preliminary version of this paper.
	
	The author is member of INDAM-GNSAGA.	}
	\section{K3 surfaces with symplectic a involution and their specializations}\label{sec: K3 surfaces with involution and their specializations}
	
The second cohomology group of any K3 surface is isometric to a standard lattice, denoted $\Lambda_{K3}$, which is even, unimodular and of signature $(3,19)$. It is isometric to $U^{\oplus 3}\oplus E_8^{\oplus 2}$.

We fix the following notation: $\Lambda_{K3}$ is generated by $u_i^{(j)}$, and $e_{k}^{(h)}$ with $i=1,2$, $j=1,2,3$, $k=1,\ldots 8$ $h=1,2$ and where $u_1^{(j)}, u_2^{(j)}$ spans the $j$-th copy of $U$ and $e_k^{(h)}$ the $h$-th copy of $E_8$, with $e_k^{(h)}e_{k+1}^{(h)}=1$, $k=1,\ldots 6$ and $e_3^{(h)}e_8^{(h)}=1$.

Let $X$ be a K3 surface admitting a symplectic involution $\iota$, that is an involution which acts trivially on $H^{2,0}(X)$.

 In \cite{Nik}, it is proved that the action of $\iota^*$ on the second cohomology of $X$, i.e. on $H^2(X,\Z)\simeq \Lambda_{K3}$, is essentially unique, which means that its realization on $\Lambda_{K3}$ does not depend on $X$ or on $\iota$, and indeed in \cite{Mo} it is proved that there exists an isometry $\gamma:H^2(X,\Z)\ra \Lambda_{K3}$ such that $ \gamma\circ\iota^*\circ\gamma^{-1}:\Lambda_{K3}\ra\Lambda_{K3}$ switches the  two copies of $E_8\subset\Lambda_{K3}$ and acts as the identity on $U^{\oplus 3}\subset\Lambda_{K3}$, i.e. $\iota^*(e_k^{(1)})=e_k^{(2)}$ and $\iota^*(u_i^{(j)})=u_i^{(j)}$. This implies that the lattice which is invariant by $\iota^*$ is  isometric to $U^{\oplus 3}\oplus E_8(2)$ and its orthogonal complement is isometric to $E_8(2)$. As a consequence one can prove 
that a K3 surface $X$ admits a symplectic involution $\iota$ if and only if the lattice $E_8(2)$ is primitively embedded in $NS(X)$, see \cite{Nik}, \cite{vGS}. 

The quotient surface $X/\iota$ is singular in 8 points and its desingularization is a K3 surface, denoted by $Y$ in the following. In the N\'eron--Severi group of $Y$ there are the 8 classes of the disjoint rational curves which resolve the singularities of $X/\iota$, which are denoted $N_i$, $i=1,\ldots, 8$. We call Nikulin lattice the lattice $N$ generated by $N_i$, $i=1,\ldots, 8$ and by $\hat{N}:=(\sum_{i=1}^8N_i)/2$, where $N_i^2=-2$ and $N_iN_j=0$ if $i\neq j$. It is the minimal primitive sublattice of $NS(Y)$ containing the classes of the curves arising from the resolution of $X/\iota$. It can be proved that a projective K3 surface $Y$ is obtained as a desingularization of the quotient of a K3 surface by a symplectic involution if and only if the lattice $N$ is primitively embedded in $NS(Y)$, see \cite{Nik} and \cite{GS}. 

The quotient map $\pi:X\ra X/\iota$ induces a $2:1$ rational map  $\pi:X\ra Y$ (both the maps are denoted with the same latter by an abuse of notation). Since both $X$ and $Y$ are K3 surfaces, $\pi$ induces the maps $\pi^*:H^2(Y,\Z)\simeq \Lambda_{K3}\ra H^2(X,\Z)\simeq \Lambda_{K3}$ and $\pi_*:H^2(X,\Z)\simeq \Lambda_{K3}\ra H^2(Y,\Z)\simeq \Lambda_{K3}$. 

To work only with smooth surfaces, one can blow up $X$ in the fixed locus of $\iota$, obtaining a non minimal surface $\widetilde{X}$ on which $\iota$ induces an involution $\widetilde{\iota}$. Then $Y\simeq \widetilde{X}/\widetilde{\iota}$. Since the fixed locus of $\iota$ on $X$ consists of 8 points, the second cohomology group of  $\widetilde{X}$ is isometric to $H^2(X,\Z)\oplus \langle -1\rangle^{\oplus 8}\simeq \Lambda_{K3}\oplus \langle -1\rangle^{\oplus 8}$. Again with an abuse of notation we denote $\pi$ also the quotient map $\widetilde{X}\ra\widetilde{X}/\widetilde{\iota}\simeq Y$ and hence we consider $\pi_*$ as map from $ \Lambda_{K3}\oplus \langle -1\rangle^{\oplus 8}$ to $\Lambda_{K3}$. 

\begin{proposition}{\rm(See \cite[Proposition 1.8]{vGS}) }\label{prop: pi*}
Let $X$, $\widetilde{X}$ and $Y$ as above. The map $\pi_*$, defined on $H^2(\widetilde{X},\Z)$, is the following
$$\begin{array}{ccccccccccc}\pi_*:H^2(\widetilde{X},\Z)\simeq &U^{\oplus 3}&\oplus E_8&\oplus E_8&\oplus \langle -1\rangle^{\oplus 8}&\ra &U(2)^{\oplus 3}&\oplus E_8&\oplus N&\hookrightarrow H^2(Y,\Z)\simeq \Lambda_{K3}\\
&(u,&x,&y,&z)&\mapsto& (u,&x+y,&z) .\end{array}$$
In particular, $\pi_*(H^2(\widetilde{X}, \Z))\simeq U(2)^{\oplus 3}\oplus E_8\oplus N$ is a sublattice of index $2^6$ of $H^2(Y,\Z)$ and the map $\pi^*$ defined on this sublattice acts as follow:
$$\begin{array}{ccccccccccc}\pi^*: &U(2)^{\oplus 3}&\oplus E_8&\oplus N&\ra &U^{\oplus 3}&\oplus E_8&\oplus E_8&\simeq H^2(X,\Z)\simeq \Lambda_{K3}\\
&(u,&x,&z)&\mapsto& (2u,&x,&x). \end{array}$$
The map $\pi^*$ on $H^2(Y,\Z)$ is the $\Q$-linear extension of this map. 
\end{proposition}

In \cite{vGS} the families of projective K3 surfaces admitting a symplectic involution were studied: they are countably many 11-dimensional families of  K3 surfaces which can be described as families of lattice polarized K3 surfaces. To state the precise results about these families we first have to recall some results and definitions in lattice theory, in the following lemma.
\begin{lemma}
\begin{enumerate}
\item
For any $d\in\mathbb{N}_{>0}$ there exists a unique (up to isometries) primitive embedding of $\langle 2d\rangle \oplus E_8(2)$ in $\Lambda_{K3}$.

\item For any $d\equiv 0\mod 2$, there exists a unique, up to isometries, even finite index overlattice $(\langle 2d\rangle\oplus E_8(2))'$ of $\langle 2d\rangle\oplus E_8(2)$ in which both the direct summands are primitive. The index is necessarily 2.

\item For any  $d\equiv 0\mod 2$, there exists a unique, up to isometries, primitive embedding of $(\langle 2d\rangle\oplus E_8(2))'$ in $\Lambda_{K3}$.

\item For any $d\in\mathbb{N}_{>0}$ there exists a unique (up to isometries) primitive embedding of $\langle 2d\rangle \oplus N$ in $\Lambda_{K3}$.

\item For any $d\equiv 0\mod 2$, there exists a unique, up to isometries, even finite index overlattice $(\langle 2d\rangle\oplus N)'$ of $\langle 2d\rangle\oplus N$ in which both the direct summands are primitive. The index is necessarily 2.

\item For any  $d\equiv 0\mod 2$, there exists a unique, up to isometries, primitive embedding of $(\langle 2d\rangle\oplus N)'$ in $\Lambda_{K3}$.
\end{enumerate}
\end{lemma} 
\proof Recall that the lattices $E_8(2)$ and $N$ have discriminant forms isometric to $u(2)^4$ and $u(2)^3$ respectively, as can be explicitly computed. Hence (1) and (4), follows by \cite[Theorem 1.14.4]{NikInt}. Point (2) is proved in \cite[Proposition 2.2]{vGS}, where a basis of the overlattice is computed. From this one computes the discriminant form, which is $\Z_{2d}(\frac{1}{2d})\oplus u(2)^3$ (and in particular it coincides with the one of the lattice $\langle 2d\rangle\oplus N$ considered in the point (4), see also \cite[Proposition 3.5]{CG}). Then, one can apply again \cite[Theorem 1.14.4]{NikInt} to the lattice $(\langle 2d\rangle\oplus E_8(2))'$, proving (3). Point (5) is proved in \cite[Proposition 2.2 and Corollary 2.1]{GS}. Again one computes explicitly  the discriminant form of $(\langle 2d\rangle\oplus N)'$ (which is $\Z_{2d}(\frac{1}{2d})\oplus u(2)^2$) and applies \cite[Theorem 1.14.4]{NikInt} to obtain (6).\endproof
\begin{rem}\label{rem: embedding NS rank 9}{\rm Since we are choosing the basis of $\Lambda_{K3}$ in such a way that the involution $\iota^*$ switches two pairs of $E_8$ in $\Lambda_{K3}$, we are assuming that the embedding of the antiinvariant lattice $E_8(2)$ is generated by the classes $\{e_k^{(1)}-e_k^{(2)}\}_{k=1,\ldots 8}$ in $\Lambda_{K3}$. Then the embedding of the lattices $\langle 2d\rangle \oplus E_8(2)$ and $(\langle 2d\rangle \oplus E_8(2))'$ is determined by the embedding of the sublattice $\langle 2d\rangle\subset NS(X)$ in the orthogonal complement to $E_{8}(2)=\langle e_k^{(1)}-e_k^{(2)},\ k=1,\ldots 8\rangle$ in $\Lambda_{k3}$. For the lattice $\langle 2d\rangle\oplus E_8(2)$ one can assume that the summand $\langle 2d\rangle$ is generated by  $u_1^{(1)}+du_2^{(1)}$; for the lattice $(\langle 2d\rangle\oplus E_8(2))'$, the summand $\langle 2d\rangle$  is generated by $2u_1^{(1)}+\left(\frac{d}{2}+1+\epsilon\right)u_2^{(1)}+(e_1^{(1)}+e_1^{(2)})+\epsilon(e_3^{(1)}+e_3^{(2)})$ where $\epsilon=0$ if $d\equiv 0 \mod 4$ and $\epsilon=1$ if $d\equiv 2 \mod 4$.}\end{rem}

\begin{proposition}{\rm(See \cite{vGS}, \cite{GS})}\label{prop: K3 surfaces with sympl rank 9}
Let $X$ be a projective K3 surface with Picard number 9.

The surface $X$ admits a symplectic involution if and only if there exists $d\in\mathbb{N}_{>0}$ such that $$\mbox{either }NS(X)\simeq \langle 2d\rangle\oplus E_8(2)\mbox{ or }NS(X) \simeq (\langle 2d\rangle\oplus E_8(2))'.$$
Moreover $$NS(X)\simeq  \langle 2d\rangle\oplus E_8(2)\mbox{ if and only if }NS(Y)\simeq  \langle (4d\rangle\oplus N)'$$ and $$NS(X)\simeq ( \langle 2d\rangle\oplus E_8(2))'\mbox{  if and only if  }NS(Y)\simeq  \langle d\rangle\oplus N.$$
\end{proposition}

The next corollary immediately follows from the previous proposition.

\begin{corollary}\label{rem: NS(X)neq NS(Y) in the generic cases}\label{cor: NS(X)neq NS(Y) in the generic cases}
	Let $X$ be a generic K3 surface admitting a symplectic involution $\iota$ and $Y$ the minimal resolution of its quotient, then $E_8(2)\simeq NS(X)\not\simeq NS(Y)\simeq N$.
	
	Let $X$ be a generic projective K3 surface admitting a symplectic involution $\iota$ and $Y$ the minimal resolution of its quotient, then $NS(X)\not\simeq NS(Y)$.
	
\end{corollary}

So if either $\rho(X)=8$ or $\rho(X)=9$ and $X$ is projective, then $NS(X)\not\simeq NS(Y)$, which of course implies that $X\not\simeq Y$. Nevertheless, there exist K3 surfaces $X$ (with Picard number higher than 9) admitting symplectic involution $\iota$ such that $NS(X)\simeq NS(Y)$ and the most famous example is given by a 10-dimensional family of elliptic K3 surfaces which admit an elliptic fibration with a 2-torsion section (already presented in the introduction and that will be discussed in details in  Section \ref{sec: family N and vGS inv}). For the K3 surfaces in this family, it holds $NS(X)\simeq U\oplus N\simeq NS(Y)$.

The aim of the next theorem is to show that the lattice $U\oplus N$ is not the unique lattice of rank 10, which can appear simultaneously as N\'eron--Severi group of a K3 surface $X$ admitting a symplectic involution $\iota$ and of the minimal resolution of $X/\iota$.

\begin{theorem}\label{theo: NS(X)=NS(Y) rank 10}
Let $X$ be a K3 surface	with $NS(X)=\left(\langle 4d\rangle\oplus E_8(2)\right)'\oplus \langle -2d\rangle$. Then $X$ admits a symplectic involution $\iota$ such that $NS(Y)\simeq NS(X)$, where $Y$ is the minimal resolution of $X/\iota$.\end{theorem}
\proof Let us assume $d=2k-1$ (the case $d=2k$ is similar). 
There exists a unique embedding of $NS(X)$ in $\Lambda_{K3}$ and hence we can fix it to be the following:  $$2u_1^{(1)}+2ku_2^{(1)}+(e_1^{(1)}+e_1^{(2)}),\ \ \{e_i^{(1)}-e_i^{(2)}\}_{i=1,\ldots,8},\ \ -u_1^{(2)}+(2k-1)u_2^{(2)}$$
where the first vector generates $\langle 4d\rangle$, the following 8 vectors generate $E_8(2)$, the last one generates $\langle -2d\rangle$.

Applying the map $\pi_*$ (see Proposition \ref{prop: pi*}) one finds that the N\'eron--Severi group of $Y$ is generated by $$L:=\pi_*(u_1^{(1)})+k\pi_*(u_2^{(1)})+\pi_*(e_1),\ \ D:=-\pi_*(u_1^{(2)})+(2k-1)\pi_*(u_2^{(2)})$$ and by the classes of the Nikulin lattice $N$. Since:  $\pi_*(u_1^{(j)})$, $\pi_*(u_2^{(j)})$ generate the lattice $U(2)$ and  $\pi_*(e_1)\in E_8$, one obtains that $L^2=4k-2=2d$ and $D^2=-4(2k-1)=-4d$. So $NS(Y)\simeq \langle 2d\rangle\oplus N\oplus \langle -4d\rangle .$

Let $\{b_1,b_2\}$ a basis of a rank 2 lattice on which the form is $\langle 4d\rangle\oplus \langle-2d\rangle$. Then the bilinear form on the basis $\langle b_1+b_2, b_1+2b_2\rangle$ is $\langle 2d\rangle\oplus \langle-4d\rangle$, which shows that $\langle 4d\rangle\oplus \langle-2d\rangle\simeq \langle -4d\rangle\oplus \langle2d\rangle$. %So $NS(Y)\simeq \langle 4d\rangle\oplus N\oplus \langle -2d\rangle$.

In \cite[Theorem 3.9]{CG}, it is observed that $\langle 4e\rangle\oplus N\simeq (\langle 4e\rangle\oplus E_8(2))'$, therefore  $$NS(Y)\simeq \langle 2d\rangle\oplus N\oplus \langle -4d\rangle \simeq \langle 4d\rangle\oplus N\oplus \langle -2d\rangle\simeq (\langle4d\rangle\oplus E_8(2))'\oplus \langle -2d\rangle\simeq NS(X).$$\endproof

 The previous theorem implies that there are infinitely many families of projective K3 surfaces $X$ such that $X$ admits a symplectic involution $\iota$, $\rho(X)=10$ and $NS(X)\simeq NS(Y)$. By Corollary \ref{rem: NS(X)neq NS(Y) in the generic cases}, the minimal Picard number for which this phenomenon can appear is 10. So we obtained maximal components of the family of K3 surfaces $X$ admitting a symplectic involution such that $NS(X)\simeq NS(Y)$. 
 \begin{rem}\label{rem: Tx=Ty, only even rank}{\rm If $X$ is a K3 surface with a symplectic involution $\iota$ such that $NS(X)\simeq NS(Y)$, then $\rho(X)$ is an even number. Indeed, as observed in \cite[Section 2.4]{vGS}, the quotient by $\iota$ induces an isomorphism of Hodge structure between $\left(T_{X}\right)_{\Q}$ and $\left(T_{Y}\right)_{\Q}$. If $NS(X)\simeq NS(Y)$, then $\left(T_X\right)_{\Q}\simeq \left(T_{Y}\right)_{\Q}$ and so the Nikulin involution induces an Hodge isometry of $\left(T_{X}\right)_{\Q}$ in itself. By \cite[Proposition 2.5]{vGS} this implies that the rank of $T_X$ has to be even.}\end{rem}
 
The previous theorem does not imply that there are K3 surfaces $X$ polarized with the lattice $\left(\langle 4d\rangle\oplus E_8(2)\right)'\oplus \langle -2d\rangle$ which are necessarily isomorphic  their quotient $Y$. In the case of the $(U\oplus N)$-polarized K3 surfaces there exist subfamilies where $X$ and $Y$ are isomorphic (see \cite{vGSc} and Proposition \ref{prop: subfamily isomorphic quotients}) and this allows one to define complex (or real) multiplication on members of such subfamilies. We don't know if the analogous phenomenon appears in each of the families described in Theorem \ref{theo: NS(X)=NS(Y) rank 10}.

 \begin{rem}{\rm In the proof of Theorem \ref{theo: NS(X)=NS(Y) rank 10}, we observed that $\langle 2d\rangle\oplus \langle -4d\rangle\simeq \langle 4d\rangle\oplus \langle -2d\rangle$. In the case $d=1$, this has a very well known geometric interpretation. Let $\langle 4\rangle\oplus \langle -2\rangle$ be generated by the nef class $A$ (with $A^4=4$)  and by the class of a rational curve $R$. The map $\varphi_{|A|}$ gives a model as a quartic in $\mathbb{P}^3$ singular in one point (the contraction of $R$). The projection form this point is a double cover of $\mathbb{P}^2$, associated to the complete linear systme of $H=A-R$. The image of the rational curve $R$ is a conic everywhere tangent to the branch locus (whose class is obviously $R$). Posing $V=A-2R$ one observes that $\langle H,V\rangle\simeq \langle 2\rangle\oplus \langle -4\rangle$. So the change of basis $\{A,R\}\mapsto \{H,V\}$ described in a lattice theoretic context in the proof of the theorem, corresponds here to project a quartic from its node. The construction remains analogous if the quartic surface contains other 8 nodes (which form a divisible sets) represented by the Nikulin lattice, and this gives the isometry $\langle 2d\rangle\oplus N\oplus\langle -4d\rangle\simeq \langle 4d\rangle\oplus N\oplus \langle -2d\rangle$.}
 \end{rem}
 
 \begin{rem}{\rm The last step in the proof of Theorem \ref{theo: NS(X)=NS(Y) rank 10} is based on the fact that 
 	$\langle 4d\rangle\oplus N\oplus \langle -2d\rangle\simeq \langle (4d\rangle\oplus E_8(2))'\oplus \langle -2d\rangle$. If one restricts again to the case $d=1$, one finds an explicit description of the isometry  	$\langle 4\rangle\oplus N\simeq (\langle 4\rangle\oplus E_8(2))'$ in \cite[Section 3.4]{CG} and this induces the required one by acting as the identity on the last summand.}\end{rem}
 \subsection{Geometric description of the case $d=1$}\label{subsection: geometry d=1}
 Let us consider a K3 surface $S$ such that $NS(S)\simeq \left(\langle4\rangle\oplus E_8(2)\right)'\oplus \langle -2\rangle$, the lattice given in Theorem \ref{theo: NS(X)=NS(Y) rank 10} if $d=1$. We now describe the geometry of such a surface, of the associated involution and of the quotient. First, we observe that $S$ is a special member of the family of the $\left(\langle4\rangle\oplus E_8(2)\right)'$-polarized K3 surface, studied in \cite[Section 3.5]{vGS}. In particular, this implies that $S$ admits a model as double cover of $\mathbb{P}^1\times \mathbb{P}^1$ branched on a curve $B$ of bidegree $(4,4)$, which is invariant for the following involution  of $\mathbb{P}^1\times \mathbb{P}^1$: $$\alpha:((x_0:x_1),(y_0:y_1))\mapsto ((y_0:y_1),(x_0:x_1)).$$ 
 The involution $\alpha$ lifts to two different involutions on $S$, one of them is symplectic and will be denoted by $\iota$. The involution $\alpha$ on $\mathbb{P}^1\times \mathbb{P}^1$ fixes the diagonal $\Delta$ and hence $\iota$ fixes on $S$ the inverse image of the eight points $\Delta\cap B$. The quotient surface $S/\iota$ is a double cover of $(\mathbb{P}^1\times \mathbb{P}^1)/\alpha\simeq \mathbb{P}^2$ branched on a conic $C_0$ and a quartic curve $B_4$ which are the image under the quotient map $\mathbb{P}^1\times \mathbb{P}^1\ra(\mathbb{P}^1\times \mathbb{P}^1)/\alpha\simeq \mathbb{P}^2$ of the curves $\Delta$ and $B$ respectively.
 
The surface $S$ is a special member of the family since in its N\'eron--Severi group there is also an extra class, with self intersection $-2$ and on which $\iota$ acts as the identity. This implies that there is a smooth rational curve on $S$ (which is not present on generic members of the family) which is contracted by the map $S\ra \mathbb{P}^1\times \mathbb{P}^1$ and which is invariant for the action of $\iota$. The presence of such a curve implies that the branch curve $B$ is singular in a point, which is necessarily contained in $B\cap \Delta$ (otherwise the curve $B$ would have two singularities, switched by $\alpha$ and hence $S$ would have two rational curves switched by $\iota$, in particular it would have a bigger N\'eron--Severi group). 

To construct a smooth model of $S$, one has to consider the blow up of $\mathbb{P}^1\times \mathbb{P}^1$ in the singular point of $B$. Let us denote $\widetilde{\mathbb{P}^1\times \mathbb{P}^1}$ the blow up and observe that $\alpha$ lifts to an involution $\widetilde{\alpha}$ of $\widetilde{\mathbb{P}^1\times \mathbb{P}^1}$ which preserves the exceptional divisor. In particular, it acts as an involution of the exceptional divisor and fixes two points on it: one is the intersection with the strict transform of $\Delta$, the other is an isolated fixed point. So $\widetilde{\mathbb{P}^1\times \mathbb{P}^1}/\widetilde{\alpha}$ is no longer a smooth surface isomorphic to $\mathbb{P}^2$, but it is singular in a point (which is contained in the image of the strict transform neither of $\Delta$ nor of $B$). 

As in the general case $S/\iota$ is birational to a double cover of $\widetilde{\mathbb{P}^1\times \mathbb{P}^1}/\widetilde{\alpha}$ branched on the image of the curve $B\cup \Delta$, but now the described double cover is singular in ten points: two points are switched by the cover involution and are the cover of the singular point of $\widetilde{\mathbb{P}^1\times \mathbb{P}^1}/\widetilde{\alpha}$, the others are due to the singularities of the branch locus, i.e. are the points covering $B\cap \Delta$. The latter correspond to the curves in the Nikulin lattice in $\widetilde{S/\iota}$. To resolve the two singularities which are not related with the Nikulin lattice, one introduces two $(-2)$-curves in $\widetilde{S/\iota}$ switched by the cover involution. This implies that a model of $\widetilde{S/\iota}$ is a double cover of $\mathbb{P}^2$ which is a specialization of the general case(in particular it is branched on a quartic and a conic) for which there also  exists a rational curve (indeed a conic) whose intersection points with the branch locus have even multiplicity; therefore this rational curve in $\mathbb{P}^2$ splits into two rational curves on $\widetilde{S/\iota}$ which are permuted by the cover involution.

One is able to identify this special curve also from a lattice theoretic point of view: the N\'eron--Severi group of $\widetilde{S/\iota}$
is $\langle 2\rangle\oplus N\oplus \langle -4\rangle$, where the class $V$ spanning  $\langle -4\rangle$ is the one due to the fact that $S$ is not general in the family of the $(\langle 4\rangle\oplus E_8(2))'$ and so $\widetilde{S/\iota}$ is non general in the family of the $\langle 2\rangle\oplus N$-polarized K3 surfaces. So the class $V$ is the one which shows that there is a special curve on $\widetilde{S/\iota}$. Indeed, if one considers just the lattice $\langle 2\rangle\oplus \langle -4\rangle$, generated by $H$ and $V$, where $H$ is the pullback on $\widetilde{S/\iota}$ of the class of a line in $\mathbb{P}^2$, one realizes that $H+V$ and $H-V$ are the classes of two rational curves switched by the cover involution of the cover $\varphi_{|H}:\widetilde{S/\iota}\ra \mathbb{P}^2$ and both of them are mapped on the same conic in $\mathbb{P}^2$ (which is necessarily a conic tangent to the branch locus in each intersection point, since $H(H-V)=H(H+V)=2$). 

\section{The $(U\oplus N)$-polarized family and the van Geemen Sarti involution}\label{sec: family N and vGS inv}

From now on, we focus on the family $\mathcal{N}$ of the $(U\oplus N)$-polarized K3 surfaces, which is the family of K3 surfaces $X$ admitting a van Geemen--Sarti involution.  First we recall standard results on such a family in Section \ref{subsec: preliminarie on N}, then we write the cohomological action of a van Geemen--Sarti involution $\sigma$ as a standard one, finding explicitly the two copies of $E_8$ switched by $\sigma^*$ in terms of the generators of the N\'eron--Severi group and of the transcendental lattice of $X$, see Section \ref{subsec: cohomological action}. 

\subsection{Preliminaries on $\mathcal{N}$}\label{subsec: preliminarie on N}
The main result on the $(U\oplus N)$-polarized K3 surfaces, which makes this family so interesting and studied in the last years, are summarized in the following proposition. We refer to \cite{ScSh} for basic definition and properties of elliptic fibrations.  
\begin{proposition}\label{prop: the family UNpolarized}{\rm (See \cite[Section 4.1]{vGS})}
	Let $X$ be a K3 surface admitting an elliptic fibration $\mathcal{E}:X\ra \mathbb{P}^1_{(t:s)}$ with a 2-torsion section $T$. Let $\sigma$ be the van Geemen--Sarti involution which is the translation by $T$.
	Then
	\begin{itemize} \item the Weierstrass equation of $\mathcal{E}$ is $$y^2=x(x^2+a(t:s)x+b(t:s)),\ \ \deg(a(t:s))=4,\ \deg(b(t:s))=8,\ \ T:(t:s)\mapsto (0,0)$$
		and generically, the singular fibers of  $\mathcal{E}$ are $8I_2+8I_1$ and the Mordell--Weil group is $\Z/2\Z$;
		\item the van Geemen--Sarti involution switches the two components and the two singular points of the $I_2$-fibers and acts on the fibers of type $I_1$ preserving the singular point;
		\item the surface $Y$, minimal resolution of $X/\sigma$, admits an elliptic fibration $\mathcal{F}$ with Weierstrass equation  $$y^2=x(x^2+-2a(t:s)x+(a^2(t:s)-4b(t:s)));$$
		generically, the singular fibers of $\mathcal{F}$ are $8I_2+8I_1$ and the Mordell--Weil group is $\Z/2\Z$;
		\item $NS(X)$ primitively contains $U\oplus N$ and generically $NS(X)\simeq U\oplus N$;
		\item $NS(Y)$ primitively contains $U\oplus N$ and generically $NS(Y)\simeq U\oplus N\simeq NS(X)$.
	\end{itemize}
\end{proposition}

By the previous result, if $NS(X)\simeq U\oplus N$, then $NS(X)\simeq NS(Y)$, which implies that $X$ and $Y$ are in the same family of polarized K3 surfaces, but of course they are not necessarily isomorphic. 
\begin{remark}\label{rem: quotient of quotient}{\rm  Since $Y$ admits a van Geemen--Sarti involution, denoted $\sigma_Y$, one can consider the minimal resolution of $Y/\sigma_Y$. This is a K3 surface which is isomorphic to $X$, indeed $X$ can be viewed as an elliptic curve $\mathcal{E}$ over $k(\mathbb{P}^1)$ and the minimal resolution of $Y/\sigma_Y$ is the quotient of the elliptic curve $\mathcal{E}$ by its full 2-torsion subgroup, i.e. $\mathcal{E}/\mathcal{E}[2]$.}\end{remark}

More surprising, there exists a subfamily of the family of the $(U\oplus N)-$polarized K3 surfaces, such that $X$ and $Y$ are isomorphic, as shown in the following result by van Geemen and Sch\"utt.
\begin{proposition}{\rm (\cite[Proposition 6.2 and its proof]{vGSc})} \label{prop: subfamily isomorphic quotients}
	Let $\alpha,\beta$ be two polynomials in one variable with complex coefficients, such that $\deg(\alpha)\leq 2$ and $\deg(\beta)\leq3$.
	The elliptic K3 surfaces with Weierstrass equation \begin{equation}\label{eq: Weierstrass siom quotient}y^2=x(x^2+2\alpha(t^2)x+\frac{1}{2}\alpha^2(t^2)+t\beta(t^2))\end{equation}
	form a 5-dimensional family $\mathcal{L}$ of K3 surfaces. Each K3 surface $X\in\mathcal{L}$ admits a van Geemen--Sarti involution such that $X\simeq Y$. Each member of the family has Complex Multiplication by $\mathbb{Q}(\sqrt{-2})$.
	
	Let $X$ be a very general element of $\mathcal{L}$, then $NS(X)\simeq U\oplus N$ and the singular fibers of the elliptic fibration are $8I_2+8I_1$, so $\mathcal{L}\subset \mathcal{N}$.
\end{proposition} 
\begin{rem}\label{rem: real multiplication}{\rm In the same Proposition of \cite{vGSc}, a different family, say $\mathcal{L}'$, of elliptic K3 surfaces with a van Geeemen--Sarti involution such that $X\simeq Y$, is described. It is given by the equation $y^2=x(x^2+2t\gamma(t^2)x+\frac{1}{2}t^2\gamma^2(t^2)+t\delta(t^2))$ where $\deg(\gamma(t))\leq 1$ and $\deg\delta(t)\leq 3$ and admits Real Multiplication by $\Q(\sqrt{2})$. The general member of this family has still $U\oplus N$ as N\'eron--Severi group, but the main difference is that the  singular fibers are $2III+6I_1+6I_2$.
In the following we will focus mainly on the family described in Proposition \ref{prop: subfamily isomorphic quotients} since the fibers are the same as the ones of a generic member of $\mathcal{N}$.}\end{rem}

By Proposition \ref{prop: the family UNpolarized}, there is a selfmap $\eta:\mathcal{N}\ra \mathcal{N}$ such that for each $X\in\mathcal{N}$, $\eta(X)=Y$. This map is an involution, since  $\eta(\eta(X))=X$ by Remark \ref{rem: quotient of quotient}.
The family $\mathcal{L}$ described in Proposition \ref{prop: subfamily isomorphic quotients} is contained in the fixed locus of $\eta$. There are also other subfamilies contained in the fixed locus (e.g. the family with real multiplication described in Remark \ref{rem: real multiplication}). 

A priori one can restrict $\eta$ to subfamilies $\mathcal{N}'$ of $\mathcal{N}$, but in general one does not obtain an involution of these subfamilies, since in general $\eta(\mathcal{N}')\neq \mathcal{N}'$. The condition that for a generic $X\in\mathcal{N}'$ it holds $NS(X)\simeq NS(Y)$, is equivalent to the condition that the involution $\eta$ restricts to an involution of $\mathcal{N}'$, i.e. that $\eta(\mathcal{N}')=\mathcal{N}'$. By Remark \ref{rem: Tx=Ty, only even rank} the subfamilies $\mathcal{N}'$ such that $\eta(\mathcal{N}')=\mathcal{N}'$ have an even codimension in $\mathcal{N}$.

\subsection{Cohomological action of the van Geemen--Sarti involution}\label{subsec: cohomological action}
Let $X$ be a K3 surface admitting an elliptic fibration with a 2-torsion section as in Proposition \ref{prop: the family UNpolarized} so that $NS(X)\simeq U\oplus N$ and $T_X\simeq U\oplus U\oplus N$. Let $F$, $F+S$, $N_i$, $i=1,\ldots,7$, $\hat{N}:=\sum_{i=1}^8N_i/2$ be a basis of $NS(X)$, whose Gram matrix is $U\oplus N$. From the point of view of the elliptic fibration, $F$ is the class of the fiber; $S$ of the zero section; $N_i$, $i=1,\ldots ,8$ are the irreducible components of the reducible fibers which do not intersect the section $S$; the class $T=2F+S-\hat{N}$ is the 2-torsion section. We often call  the component of a reducible fiber which intersects the zero section the ``trivial component", and the others ``non trivial components". We denote by $\sigma$ the symplectic involution which is the translation by $T$. It acts as follows on the basis of $NS(X)$:
\begin{equation}\label{eq: action sigma}\sigma^*(F)=F,\ \sigma^*(S)=T,\ \ \sigma^*(N_i)=F-N_i.\end{equation}
	
Up to isometries there exists a unique embedding of $NS(X)\simeq U\oplus N$ in $\Lambda_{K3}=E_8\oplus E_8\oplus U\oplus U\oplus U$
and we now describe an embedding $\varphi:U\oplus N\hookrightarrow \Lambda_{K3}$ such that the van Geemen--Sarti involution switches the two copies of $E_8$. This is equivalent to require that the van Geemen--Sarti involution acts as $-id$ on $E_8(2)=\langle e_k^{(1)}-e_k^{(2)}\rangle$ and as $id$ on its orthogonal complement $\langle u_1^{(j)},u_2^{(j)},
 e_k^{(1)}+e_k^{(2)}\rangle$.  Then $\varphi$ is given by
 {\Small	\begin{longtable}{l}\\$\varphi(F)=\sum_{j=1}^2\left(-4e_1^{(j)}  -7e_2^{(j)} -10e_3^{(j)}  -8e_4^{(j)}  -6e_5^{(j)}  -4e_6^{(j)}  -2e_7^{(j)}  -5e_8^{(j)}\right)+2u_1^{(3)}+2u_2^{(3)}$\\ 
 			$\varphi(S)=   e_1^{(2)}$\\
 			$\varphi(N_1)= -2e_1^{(1)}  -3e_2^{(1)}  -4e_3^{(1)}  -3e_4^{(1)}  -2e_5^{(1)}  -e_6^{(1)}   -2e_8^{(1)}  -2e_1^{(2)}  -4e_2^{(2)}  -6e_3^{(2)}  -5e_4^{(2)}  -4e_5^{(2)}  -3e_6^{(2)}  -2e_7^{(2)}  -3e_8^{(2)} +u_1^{(3)}+   u_2^{(3)}$\\
 			$\varphi(N_2)=-2e_1^{(1)}  -3e_2^{(1)}  -4e_3^{(1)}  -3e_4^{(1)}  -2e_5^{(1)}  -e_6^{(1)}  -e_7^{(1)}  -2e_8^{(1)}  -2e_1^{(2)}  -4e_2^{(2)}  -6e_3^{(2)}  -5e_4^{(2)}  -4e_5^{(2)}  -3e_6^{(2)}  -e_7^{(2)}  -3e_8^{(2)}+u_1^{(3)}+u_2^{(3)}$\\
 			$\varphi(N_3)=-2e_1^{(1)}  -3e_2^{(1)}  -4e_3^{(1)}  -3e_4^{(1)}  -2e_5^{(1)}  -2e_6^{(1)}  -e_7^{(1)}  -2e_8^{(1)}  -2e_1^{(2)}  -4e_2^{(2)}  -6e_3^{(2)}  -5e_4^{(2)}  -4e_5^{(2)}  -2e_6^{(2)}  -e_7^{(2)}  -3e_8^{(2)}+u_1^{(3)}+u_2^{(3)}$\\
 			$\varphi(N_4)=-2e_1^{(1)}  -3e_2^{(1)}  -4e_3^{(1)}  -3e_4^{(1)}  -3e_5^{(1)}  -2e_6^{(1)}  -e_7^{(1)}  -2e_8^{(1)}  -2e_1^{(2)}  -4e_2^{(2)}  -6e_3^{(2)}  -5e_4^{(2)}  -3e_5^{(2)}  -2e_6^{(2)}  -e_7^{(2)}  -3e_8^{(2)}+u_1^{(3)}+u_2^{(3)}$
 			\\$\varphi(N_5)=-2e_1^{(1)}  -3e_2^{(1)}  -4e_3^{(1)}  -4e_4^{(1)}  -3e_5^{(1)}  -2e_6^{(1)}  -e_7^{(1)}  -2e_8^{(1)}  -2e_1^{(2)}  -4e_2^{(2)}  -6e_3^{(2)}  -4e_4^{(2)}  -3e_5^{(2)}  -2e_6^{(2)}  -e_7^{(2)}  -3e_8^{(2)}+u_1^{(3)}+u_2^{(3)}$\\
 			$\varphi(N_6)=-2e_1^{(1)}  -3e_2^{(1)}  -5e_3^{(1)}  -4e_4^{(1)}  -3e_5^{(1)}  -2e_6^{(1)}  -e_7^{(1)}  -2e_8^{(1)}  -2e_1^{(2)}  -4e_2^{(2)}  -5e_3^{(2)}  -4e_4^{(2)}  -3e_5^{(2)}  -2e_6^{(2)}  -e_7^{(2)}  -3e_8^{(2)}+u_1^{(3)}+u_2^{(3)}$\\
 			$\varphi(N_7)=-2e_1^{(1)}  -3e_2^{(1)}  -5e_3^{(1)}  -4e_4^{(1)}  -3e_5^{(1)}  -2e_6^{(1)}  -e_7^{(1)}  -3e_8^{(1)}  -2e_1^{(2)}  -4e_2^{(2)}  -5e_3^{(2)}  -4e_4^{(2)}  -3e_5^{(2)}  -2e_6^{(2)}  -e_7^{(2)}  -2e_8^{(2)}+u_1^{(3)}+u_2^{(3)}$\\
 			$\varphi(\hat{N})=-9e_1^{(1)} -14e_2^{(1)} -20e_3^{(1)} -16e_4^{(1)} -12e_5^{(1)}  -8e_6^{(1)}  -4e_7^{(1)} -10e_8^{(1)}  -7e_1^{(2)} -14e_2^{(2)} -20e_3^{(2)} -16e_4^{(2)} -12e_5^{(2)}  -8e_6^{(2)}  -4e_7^{(2)} -10e_8^{(2)}$\\
 			$+4u_1^{(3)}+4u_2^{(3)}$\\ \caption{The embedding $\varphi$ of $U\oplus N$ in $\Lambda_{K3}$\label{eq: U+Nin Lambda}}
 	\end{longtable}
 }
 
	The transcendental lattice of $X$ is $T_X=N\oplus U\oplus U$ with basis $t_i$, $i=1\ldots, 12$, where $t_i^2=-2$ if $i=1,\ldots, 7$, $t_it_j=0$ if $i,j=1,\ldots 7$ and $i\neq j$ and $t_8^2=-4$, $t_8t_i=-1$, $i=1,\ldots 7$, $\{t_9,t_{10}\}$ and $\{t_{11},t_{12}\}$ are standard basis of a copy of $U$ in $T_X$. We fix an embedding $\phi:T_X\hookrightarrow\Lambda_{K3}$:
	\begin{align}\label{eq: transcendental in Lambda}\begin{array}{c}
	\phi(t_1)=\sum_{j=1}^2(2e_1^{(j)} +  4e_2^{(j)}+   6e_3^{(j)}+   5e_4^{(j)}+   4e_5^{(j)}+   3e_6^{(j)}+   2e_7^{(j)}+   3e_8^{(j)})-u_1^{(3)}-u_2^{(3)}\\	
	\phi(t_2)=\sum_{j=1}^2(2e_1^{(j)}+   4e_2^{(j)}+   6e_3^{(j)}+   5e_4^{(j)}+   4e_5^{(j)}+   3e_6^{(j)}+   e_7^{(j)}+   3e_8^{(j)})-u_1^{(3)}-u_2^{(3)}\\
	\phi(t_3)=\sum_{j=1}^2(2e_1^{(j)}+   4e_2^{(j)}+   6e_3^{(j)}+   5e_4^{(j)}+   4e_5^{(j)}+   2e_6^{(j)}+   e_7^{(j)}+   3e_8^{(j)})-u_1^{(3)}-u_2^{(3)}\\
	\phi(t_4)=\sum_{j=1}^2(2e_1^{(j)}+   4e_2^{(j)}+   6e_3^{(j)}+   5e_4^{(j)}+   3e_5^{(j)}+   2e_6^{(j)}+   e_7^{(j)}+   3e_8^{(j)})-u_1^{(3)}-u_2^{(3)}\\
	\phi(t_5)=\sum_{j=1}^2(2e_1^{(j)}+   4e_2^{(j)}+   6e_3^{(j)}+   4e_4^{(j)}+   3e_5^{(j)}+   2e_6^{(j)}+   e_7^{(j)}+   3e_8^{(j)})-u_1^{(3)}-u_2^{(3)}\\
	\phi(t_6)=\sum_{j=1}^2(2e_1^{(j)}+   4e_2^{(j)}+   5e_3^{(j)}+   4e_4^{(j)}+   3e_5^{(j)}+   2e_6^{(j)}+   e_7^{(j)}+   3e_8^{(j)}) -u_1^{(3)}-u_2^{(3)}\\
	\phi(t_7)=\sum_{j=1}^2(2e_1^{(j)}+   4e_2^{(j)}+   5e_3^{(j)}+   4e_4^{(j)}+   3e_5^{(j)}+   2e_6^{(j)}+   e_7^{(j)}+   2e_8^{(j)})-u_1^{(3)}-u_2^{(3)}\\
	\phi(t_8)=\sum_{j=1}^2(7e_1^{(j)}+  14e_2^{(j)}+  20e_3^{(j)}+  16e_4^{(j)}+  12e_5^{(j)}+   8e_6^{(j)}+   4e_7^{(j)}+  10e_8^{(j)}) -3u_1^{(3)}-4u_2^{(3)}\\
	\phi(t_9)=u_1^{(1)},\ \ 
	\phi(t_{10})=u_2^{(1)},\ \ 
	\phi(t_{11})=u_1^{(2)},\ \ 
	\phi(t_{12})=u_2^{(2)}\end{array}\end{align}
	
	To obtain $H^2(X,\Z)$ as overlattices of $NS(X)\oplus T_X$ where $NS(X)=\langle N_1,\ldots, N_7,\hat{N}, F, S\rangle$ and $T_X=\langle t_1,\ldots t_{12}\rangle$, one adds the classes
	$$(N_1 + N_2 + t_1 + t_2 )/2,\ (N_2 + N_7 +t_2 + t_7)/2,$$ 
	$$(N_3 + N_4+t_3 + t_4)/2,\ 
	(N_1+N_2+N_4 + N_7+t_1+t_2+t_4+t_7)/2,$$ $$(N_5 + N_6+ t_5+t_6)/2,\ (N_6+N_8+2t_8-t_1-t_2-t_3-t_4-t_5-t_7)/2.$$
	
	The lattice $(\Lambda_{K3}^{\sigma^*})^{\perp}\simeq E_8(2)\subset NS(X)$ is generated as follows:
	\begin{align}\label{eq: diagE8(2)} \xymatrix{2F-\hat{N}\ar@{-}[r]&-F+N_6+N_7\ar@{-}[r]&N_5-N_6\ar@{-}[r]\ar@{-}[d]&N_4-N_5\ar@{-}[r]&N_3-N_4\ar@{-}[r]&N_2-N_3\ar@{-}[r]&N_1-N_2\\&&N_6-N_7} 	\end{align}
	which is indeed the antiinvariant lattice on $NS(X)$ for the action described in \eqref{eq: action sigma}. The embedding $\varphi$ (as in Table \ref{eq: U+Nin Lambda}) is obtained observing that the lattice isometric to $E_8(2)$ as in \eqref{eq: diagE8(2)} has to coincide with the classes $e_i^{(1)}-e_i^{(2)}$ and that $\left(e_i^{(1)}-e_i^{(2)}\right)+\left(e_i^{(1)}+e_i^{(2)}\right)=2e_i^{(1)}$ is 2-divisible in $H^2(X,\Z)$. So, for example, $\varphi(N_5-N_6)=e_3^{(1)}-e_3^{(2)}$ and since $(N_5-N_6)+(t_5-t_6)$ is 2-divisible, one obtains $e_3^{(1)}=\left(\varphi(N_5-N_6)+\phi(t_5-t_6)\right)/2$ and $e_3^{(2)}=\left(-\varphi(N_5-N_6)+\phi(t_5-t_6)\right)/2$. This allows to write the basis $\{u_1^{(j)}, u_2^{(j)},e_k^{(h)}\}$ in terms of $\{F,S,N_i, t_j\}$ and by the inverse change of basis one obtains $\varphi$ and $\phi$. 

\begin{rem}\label{rem: varphi}{\rm
The choice of the embeddings $\varphi:NS(X)\hookrightarrow \Lambda_{K3}$ and $\phi:T_X\hookrightarrow \Lambda{K_3}$ are induced by the requirement that $(NS(X)^{\sigma^*})^{\perp}\simeq E_8(2)$ has to be the one described in \eqref{eq: diagE8(2)} and simultaneously generated by $e_j^{(1)}-e_j^{(2)}$ in the basis of $\Lambda_{K3}$. So we identify the K3 surfaces $X$ with a van Geemen--Sarti involution, with the marked K3 surfaces $(X,\Phi)$, where a $\Phi$ is an appropriate $\Q$-linear extension of $(\varphi,\phi):NS(X)\oplus T_X \hookrightarrow \Lambda_{K3}$. This guarantees the existence of a symplectic involution $\sigma$ on $X$ such that $\Phi\circ\sigma^*\circ\Phi^{-1}$ switches two pairs of $E_8$ in $\Lambda_{K3}$ and is the identity on $U^3$ and which acts as the van Geemen--Sarti involution (i.e. as in \eqref{eq: action sigma}) on the elliptic fibration whose class of the fiber is the first generator of $U$ in $U\oplus N\simeq NS(X)$.
}\end{rem}
In the following we often identify the classes in $NS(X)$ (resp. $T_X$, $H^2(X,\Z)$) with the corresponding ones in $\Lambda_{K3}$, omitting the embedding $\varphi$ (resp. $\phi$, $\Phi$).

We apply the map $\pi_*$ described in Proposition \ref{prop: pi*} to find the image $\pi_*(T_X)$ of $T_X$ in $H^2(Y,\Z)$. We observe that $T_Y$ is necessarily an overlattice of finite index of $\pi_*(T_X)$, since $\pi_*(T_X)_{\Q}\simeq (T_Y)_{\Q}$. Now, $H^2(X,\Z)^{\sigma^*}\simeq E_8(2)\oplus U\oplus U\oplus U$, generated by $e_i^{(1)}+e_i^{(2)}$, $i=1,\ldots, 8$, $u_j^{(h)}$, $j=1,2$, $h=1,2,3$.
	By Proposition \ref{prop: pi*}, one obtains that $\pi_*(H^2(X,\Z))\simeq E_8\oplus U(2)\oplus U(2)\oplus U(2)$ is generated by $E_i:=\pi_*(e_i^{(1)})=\pi_*(e_i^{(2)})$, $U_{j}^{(h)}:=\pi_*(u_j^{(h)}),\ j=1,2,\ h=1,2,3$. The invariant classes $e_i^{(1)}+e_i^{(2)}$ are mapped to $2E_i$ and the invariant classes $u_j^{(h)}$ are mapped to $U_j^{(h)}$.\\    
In particular one obtains that $\pi_*(T_X)$ is generated by 
$$\begin{array}{c}\pi_*(t_1)=2(2E_1 +  4E_2+   6E_3+   5E_4+   4E_5+   3E_6+   2E_7+   3E_8)-U_1^{(3)}-U_2^{(3)}\\
\pi_*(t_2)=2(2E_1+   4E_2+   6E_3+   5E_4+   4E_5+   3E_6+   E_7+   3E_8)-U_1^{(3)}-U_2^{(3)}\\

\pi_*(t_3)=2(2E_1+   4E_2+   6E_3+   5E_4+   4E_5+   2E_6+   E_7+   3E_8)-U_1^{(3)}-U_2^{(3)}\\

\pi_*(t_4)=2(2E_1+   4E_2+   6E_3+   5E_4+   3E_5+   2E_6+   E_7+   3E_8)-U_1^{(3)}-U_2^{(3)}\\

\pi_*(t_5)=2(2E_1+   4E_2+   6E_3+   4E_4+   3E_5+   2E_6+   E_7+   3E_8)-U_1^{(3)}-U_2^{(3)}\\

\pi_*(t_6)=2(2E_1+   4E_2+   5E_3+   4E_4+   3E_5+   2E_6+   E_7+   3E_8) -U_1^{(3)}-U_2^{(3)}\\

\pi_*(t_7)=2(2E_1+   4E_2+   5E_3+   4E_4+   3E_5+   2E_6+   E_7+   2E_8)-U_1^{(3)}-U_2^{(3)}\\

\pi_*(t_8)=2(7E_1+  14E_2+  20E_3+  16E_4+  12E_5+   8E_6+   4E_7+  10E_8) -3U_1^{(3)}-4U_2^{(3)}\\

\pi_*(t_9)=U_1^{(1)},\ \ 
\pi_*(t_{10})=U_2^{(1)},\ \ 
\pi_*(t_{11})=U_1^{(2)},\ \ 
\pi_*(t_{12})=U_2^{(2)}\end{array}$$

The transcendental lattice $T_Y$, overlattice of finite index of $\pi_*(T_X)$, is primitively embedded in $H^2(Y,\Z)$, so the classes $\pi_*(t_i)+\pi_*(t_{i+1})$, $i=1,2,3,4,5,6$, which are non primitive, have to be divided by 2 in $H^2(Y,\Z)$. In this way one finds the following  $\Z$-basis of $T_Y$
$$\tau_1:=\pi_*(t_1), \ \tau_i=\left(\pi_*(t_{i-1}+t_{i})\right)/2,\ \ i=2,\ldots, 7, \tau_j=\pi_*(t_j),\ j=8,\ldots, 12.$$
The lattice generated by $\tau_i$, $i=1,\ldots, 8$ has the following intersection matrix
\begin{eqnarray}\label{eq: matrix M}M=\left[\begin{array}{cccccccccc}
-4& -2&  0&  0&  0&  0&  0& -2\\
-2& -2& -1&  0&  0&  0&  0& -2\\ 
0& -1& -2& -1&  0&  0&  0& -2\\
0&  0& -1& -2& -1&  0&  0& -2\\
0&  0&  0& -1& -2& -1&  0& -2\\
0&  0&  0&  0& -1& -2& -1& -2\\
0&  0&  0&  0&  0& -1& -2& -2\\
-2& -2& -2& -2& -2& -2& -2& -8\end{array}\right]\end{eqnarray}

The lattices $T_Y$ is then isometric to $M\oplus U(2)\oplus U(2)$ (the two copies of $U(2)$ are generated by $\tau_i$, $i=9,\ldots,12$).

The discriminant group is generated by $\tau_1/2$ and $\tau_i/2$, $i=8,\ldots , 12$. The discriminant form is the same as the one of $D_4\oplus U(2)\oplus U(2)$, which is isometric to the one of $U(2)^{\oplus 3}$, (see \cite{NikInt}). So $T_Y$ is an even lattice of signature $(2,10)$ and discriminant form $u(2)^3$, which implies that it is isometric to $U\oplus U\oplus N$.

The lattice $\pi_*(NS(X))\simeq U$ is generated by $F_Y=(\pi_*(F))/2$ and $S_Y=\pi_*(S)$, which correspond to the class of the fiber and of the section, respectively, of the elliptic fibration induced on $Y$ by the one of $X$. The surface $Y$ is the minimal resolution of $X/\sigma$ and the resolution introduces eight disjoint rational curves, spanning a Nikulin lattice. These curves will be denoted by $n_i$. Hence $H^2(Y,\Z)$ is an overlattice of finite index of $T_Y\oplus U\oplus N$, where $T_Y$ is generated by the classes $\tau_i$ defined above, the lattice $U$ is generated by $F_Y$ and $S_Y$ and the Nikulin lattice $N$ is generated by the curves $n_i$.
To obtain a $\Z$-basis of $H^2(Y,\Z)$, it suffices to add to the previous generators the classes:
$$\begin{array}{cc}(\tau_{11}+n_1+n_2+n_3+n_4)/2, & (\tau_{12}+n_1+n_2+n_3+n_5)/2\\
(\tau_{9}+n_1+n_2+n_6+n_7)/2,& (\tau_{10}+n_1+n_2+n_6+n_8)/2,\\ 
(\tau_1+n_1+n_2)/2,& (\tau_8+n_2+n_3+n_7+n_8)/2.\end{array}$$

\section{Specializations of $X$ and $Y$ to K3 surfaces with  Picard number 11}\label{sec: K3 with vGS rank 11}
We now classify all the K3 surfaces $X_{11}$ with Picard number 11 and admitting a van Geemen--Sarti involution, by describing their N\'eron--Severi groups and the related specializations of the elliptic fibration $\mathcal{E}$ given in Proposition \ref{prop: the family UNpolarized}. The K3 surfaces $X_{11}$ are contained in countably many irreducible subfamilies of the family $\mathcal{N}$, each of codimension 1. We denote $\mathcal{N}_{11}$ the union of all these components.
We denote $Y_{11}$ the minimal resolution of the quotient of $X_{11}$ by the van Geemen--Sarti involution. Since the K3 surfaces which are quotient by a van Geemen--Sarti involution are contained in $\mathcal{N}$, if $X_{11}\in\mathcal{N}_{11}$ then also $Y_{11}\in\mathcal{N}_{11}$, but in general $X_{11}$ and $Y_{11}$ are contained in different irreducible components. More precisely, we show that $NS(Y_{11})\not\simeq NS(X_{11})$ and in particular, it is not possible that $X_{11}\simeq Y_{11}$. This is already known, see Remark \ref{rem: Tx=Ty, only even rank}, but here we explicitly determine $NS(Y_{11})$ depending on $NS(X_{11})$ and viceversa, obtaining a result analogous to Proposition \ref{prop: K3 surfaces with sympl rank 9}.

 We underline that we are specializing the pair $(X,\sigma)$, where $X$ is a K3 surface generic among the ones in Proposition \ref{prop: the family UNpolarized} and $\sigma$ is its van Geemen--Sarti involution. Hence, by Remark \ref{rem: varphi}, we are considering marked K3 surface such that $\varphi(U\oplus N)$ is as in Table \ref{eq: U+Nin Lambda}, and the action of the specialized van Geemen--Sarti involution is $-1$ on the lattice in \eqref{eq: diagE8(2)} and $+1$ on its orthogonal complement. This implies that the quotient map is the one described in Section \ref{subsec: cohomological action}.
\subsection{The surfaces $X_{11}$}\label{subsec: X11}	
	\begin{theorem}\label{theorem: X rho 11}
	Let $X_{11}$ be a K3 surface admitting a van Geemen--Sarti involution such that $\rho(X_{11})=11$, then one of the following holds:
	\begin{itemize}
		\item $NS(X_{11})=U\oplus N\oplus \langle -2d\rangle$ and $T_{X_{11}}=U\oplus N\oplus \langle 2d\rangle$; 
		\item $d\equiv 0\mod 2$ and $NS(X_{11})$ is the unique, up to isometries, overlattice of index 2 of $U\oplus N\oplus \langle -2d\rangle$ in which both $U\oplus N$ and $\langle -2d\rangle$ are primitively embedded. In this case $T_{X_{11}}\simeq U\oplus D_4\oplus D_4\oplus \langle 2d\rangle$ and $NS(X_{11})\simeq U\oplus D_4\oplus D_4\oplus \langle -2d\rangle.$
	\end{itemize}	
In all these cases there exists a unique (up to isometries) embedding of $NS(X_{11})$ and of $T_{X_{11}}$ in $\Lambda_{K3}$.
\end{theorem}

\begin{proof}
If a K3 surface admits a van Geemen--Sarti involution, then $U\oplus N$ is primitively embedded in its N\'eron--Severi group, see \cite{vGS}.
If the Picard number is 11, the orthogonal complement of $U\oplus N$ in $NS(X_{11})$ is one dimensional and it is generated by a vector $V$. By the Hodge index theorem $V^2<0$, so $V^2=-2d$ for a certain $d\in \mathbb{N}_{>0}$ and $NS(X_{11})$ is a finite index overlattice of $U\oplus N\oplus \Z V\simeq U\oplus N\oplus \langle -2d\rangle$ in which both $U\oplus N$ and $\langle -2d\rangle$ are primitively embedded (since $V$ is the generator of the orthogonal complement to $U\oplus N$). To determine such an overlattice, we recall that the finite index overlattices of a direct sum of lattices $L\oplus M$ are in $1:1$ correspondence with the isotropic subgroups of $A_L\oplus A_M$. In our case $L\simeq U\oplus N$ and $M=\langle -2d\rangle$, the discriminant form of $A_L$ is $u(2)^3$ and the one of $A_M$ is $\mathbb{Z}_{2d}(-\frac{1}{2d})$. Moreover, since we are requiring that both $U\oplus N$ and $\langle -2d\rangle$ are primitive in $NS(X)$, one has to chose the isotropic subgroup $H$ in such a way that $H\cap A_L$ and $H\cap A_M$ are both non trivial.

Let us denote by $\delta_1,\delta_2,\delta_3,\delta_4,\delta_5,\delta_6,\delta_7$ a basis of the discriminant group on which the discriminant quadratic form is $u(2)^3\oplus(-\frac{1}{2d})$. Since the discriminant quadratic form on $\delta_i,\ i=1,\ldots,6$ takes only integer values, if $d\equiv 1\mod 2$, then there are no overlattices of $U\oplus N\oplus \langle -2d\rangle$ as required. Otherwise, if $d\equiv 2\mod 4$, the class $\delta_1+\delta_2+d\delta_7$ is a subgroup of order 2 which is isotropic and it determines an overlattice of index 2 of $U\oplus N\oplus \langle -2d\rangle$. Similarly,  if $d\equiv 0\mod 4$, the class $\delta_1+d\delta_7$ is a subgroup of order 2 which is isotropic and it determines an overlattice of index 2 of $U\oplus N\oplus \langle -2d\rangle$. These overlattices are unique because in $u(2)^3$ there is a unique orbit, up to isometries, of vectors with self-intersection 1 (which contains $\delta_1+\delta_2$) and a unique orbit of non trivial vectors with self-intersection 0 (which contains $\delta_1$).

The discriminant form of $U\oplus N\oplus \langle -2d\rangle$ is $u(2)^3\oplus(-\frac{1}{2d})$ and the ones of its index two overlattices (in case $d\equiv 0\mod 2$) is obtained as the orthogonal to the isotropic group $H$ in $A_{U\oplus N\oplus \langle -2d\rangle}$. Hence when $NS(X)$ is an overlattice of $U\oplus N\oplus \langle-2d\rangle$, the discriminant group and form is as follow: 
\begin{itemize}
	\item if $d\equiv 0\mod 4$, $A_{NS(X)}$ is generated by $\delta_1, \delta_2+\delta_7, \delta_3,\delta_4,\delta_5,\delta_6$. Since $\delta_1=-d(\delta_2+\delta_7)+\delta_1+d\delta_7$, a basis for the discriminant group is $\delta_2+\delta_7, \delta_3,\delta_4,\delta_5,\delta_6$ and the discriminant from is $\mathbb{Z}_{2d}(\frac{-1}{2d})\oplus u(2)^2$;
	\item if $d\equiv 2\mod 4$,  $A_{NS(X)}$ is generated by $\delta_1-\delta_2, \delta_1+\delta_7, \delta_3,\delta_4,\delta_5,\delta_6$. Since $\delta_1-\delta_2=-d(\delta_2+\delta_7)+(\delta_1+\delta_2+d\delta_7)+(d-2)\delta_2$, a basis for the discriminant group is $\delta_1+\delta_7, \delta_3,\delta_4,\delta_5,\delta_6$ and the discriminant form is  $\mathbb{Z}_{2d}(\frac{-1}{2d})\oplus u(2)^2$.
\end{itemize}
The discriminant form of the transcendental lattice is the opposite of the one of the N\'eron--Severi group and this determines the transcendental lattice uniquely.

The embedding of $NS(X_{11})$ (resp. $T_{X_{11}}$) in $\Lambda_{K3}$ is unique by \cite[Theorem 1.14.4]{NikInt} .
\end{proof}

Let $X_{11}$ be a K3 surface whose N\'eron--Severi group is as in Theorem \ref{theorem: X rho 11}. We fix explicitly the embeddings of $NS(X_{11})$ and $T_{X_{11}}$ is such a way that they are compatible with the cohomological action of the van Geemen--Sarti involution described in Section \ref{subsec: cohomological action}: a $\Q$-basis of $NS(X_{11})$ is given by $F,S, N_i,\hat{N}$ as in Table \ref{eq: U+Nin Lambda} and 
$$V:=\left\{\begin{array}{ll}t_{11}-dt_{12}&\mbox{ if NS}(X)=U\oplus N\oplus \langle -2d\rangle,\\ \\t_1+t_2+2t_{11}-2kt_{12} &\mbox{ if }d=4k+2\mbox{ and NS}(X)\mbox{ is an index 2 overlattice of }U\oplus N\oplus \langle -2d\rangle,\\
	\\ t_1+t_2+t_3+t_4+2t_{11}-2kt_{12} &\mbox{ if }d=4k+4\mbox{ and NS}(X)\mbox{ is an index 2 overlattice of }U\oplus N\oplus \langle -2d\rangle.\end{array}\right.$$
In the first case, the basis is also a $\Z$-basis. In the remaining two, if $V=t_1+t_2+2t_{11}-2kt_{12}$, then the class $(V+N_1+N_2)/2\in NS(X_{11})$ and if $V=t_1+t_2+t_3+t_4+2t_{11}-2kt_{12}$, then the class $(V+N_1+N_2+N_3+N_4)/2\in NS(X_{11})$. This provides a $\Z$-basis in all the cases.
Moreover, in case $k=0$, one can equivalently consider the class $t_1+t_2$ (resp. $t_1+t_2+t_3+t_4$) as vector $V$ of self intersection $-4$ (resp. $-8$).

To obtain the transcendental lattice of $X_{11}$ it suffices to find the orthogonal to $V$ in $N\oplus U\oplus U$.\\

The following definition fixes the notation for the three lattices introduced before: 
\begin{definition} We define the following lattices:
	\begin{itemize}	\item for any $d\in \mathbb{N}_{>0}$: $\Lambda_d=U\oplus N\oplus \langle -2d\rangle$;
	\item if $d\equiv 2\mod 4$: $\Lambda_d^{(a)}$ is the unique, up to isometries, overlattice of index 2 of  $U\oplus N\oplus \langle -2d\rangle$; 
	\item if $d\equiv 0\mod 4$: $\Lambda_d^{(b)}$ is the unique, up to isometries, overlattice of index 2 of  $U\oplus N\oplus \langle -2d\rangle$. 
\end{itemize}
\end{definition}
A direct consequence of Theorem \ref{theorem: X rho 11} is the following.
\begin{corollary} If $X_{11}$ is a K3 surface admitting a van Geemen--Sarti involution and $\rho(X_{11})=11$, then $X_{11}$ is a general member of exactly one of the following 9-dimensional families of K3 surfaces:
	$$\mathcal{L}_d=\{(\Lambda_d)\mbox{-polarized K3 surfaces}\},\ \ 
	\mathcal{L}_d^{(a)}=\{(\Lambda_d^{(a)})\mbox{-polarized K3 surfaces}\},\ \ \mathcal{L}_d^{(b)}=\{(\Lambda_d^{(b)})\mbox{-polarized K3 surfaces}\}.$$
	In particular $\mathcal{N}_{11}=\bigcup_{d\in \N_{>0}} \mathcal{L}_d\cup \bigcup_{d\equiv 2\mod 4} \mathcal{L}_d^{(a)}\cup\bigcup_{d\equiv 0\mod 4} \mathcal{L}_d^{(b)}$.
\end{corollary}

The specialization from $X$, whose N\'eron--Severi is isometric to $U\oplus N$, to $X_{11}$ preserves the elliptic fibration $\mathcal{E}$ and its torsion section, in particular it can be described as a specialization of the elliptic fibration $\mathcal{E}$. So we are either gluing some singular fibers or introducing some new sections. There are a finite number of possibilities of the first type:  the singular fibers of $X$ are $8I_2+8I_1$ so a priori one can glue either two fibers of type $I_1$ to obtain one of type $I_2$ (case (1) of the following proposition), or two fibers of type $I_2$ to obtain one of type $I_4$ (case (2) in the following proposition), or one fiber of type $I_1$ and one of type $I_2$ to obtain one of type $I_3$. In the latter case one destroys the torsion section, hence this possibility can not appear as specialization of $X$ inside the family $\mathcal{N}$ of the K3 surfaces admitting a van Geemen--Sarti involution. The specializations which leave invariant the configuration of singular fibers and introduce new sections of infinite order are all contained in the family $\mathcal{N}$ and there are infinitely many of them (cases (3) and (4) of the following proposition).

\begin{prop}\label{prop: specialization elliptic fibrations, 11}
	Let $X_{11}$ be a K3 surface as in Theorem \ref{theorem: X rho 11}. Then $X_{11}$ admits an elliptic fibration $\mathcal{E}_{11}$ and a van Geemen--Sarti involution which specializes the one of $X$.
Let $T$ denote the 2-torsion section associated to the van Geemen--Sarti involution, then:
\begin{enumerate}
	\item if $X_{11}\in\mathcal{L}_1$, the singular fibers of $\mathcal{E}_{11}$ are $9I_2+6I_1$, $MW(\mathcal{E}_{11})=\Z/2\Z$ and $T$ meets eight fibers of type $I_2$ in the non trivial component and the other one in the trivial component; 
	\item if $X_{11}\in\mathcal{L}_2^{(a)}$, the singular fibers of $\mathcal{E}_{11}$ are $I_4+6I_2+8I_1$, $MW(\mathcal{E}_{11})=\Z/2\Z$ and $T$ meets all the reducible fibers in a non trivial component (the central one for the $I_4$-fiber); 
	\item if $X_{11}\in\mathcal{L}_d$ with $d>1$, the singular fibers of $\mathcal{E}_{11}$ are $8I_2+8I_1$, $MW(\mathcal{E}_{11})=\Z\times \Z/2\Z$ and the section of infinite order generating the non torsion part of the Mordell--Weil group intersects the zero section $S$ in $d-2$ points, $T$ in $d$ points and all the reducible fibers in the trivial component. The Mordell--Weil lattice is $\left[2d\right]$; 
	\item if $X_{11}\in\mathcal{L}_d^{(a)}$ (resp. $\mathcal{L}_d^{(b)}$) with $d>2$, the singular fibers of $\mathcal{E}_{11}$ are $8I_2+8I_1$, $MW(\mathcal{E}_{11})=\Z\times \Z/2\Z$ and the section of infinite order generating the non torsion part of the Mordell Weil group intersects the zero section $S$ in $(d-6)/4$ (resp. $(d-4)/4$) points, $T$ in $(d-2)/4$  (resp. $(d-4)/4$) points and two (resp. four) of the reducible fibers in the non trivial component. The Mordell--Weil lattice is $\left[d/2\right]$ (resp. $\left[d/2\right]$).
\end{enumerate}\end{prop}
\proof The classes $F$, $S$, $N_i$, $i=1,\ldots, 8$ and $T$ in $NS(X)\hookrightarrow NS(X_{11})$ will be denoted in the same way when they are considered as classes in $NS(X_{11})$. The class $F$ is the class of the fiber of the fibration $\mathcal{E}_{11}$ which specializes $\mathcal{E}$ and similarly $S$ and $T$ remain sections.

If $d=1$, then $NS(X_{11})=U\oplus N\oplus \langle -2\rangle$, and the last summand is generated by $V$, then we can assume that $V$ is the class of a $(-2)$-curve orthogonal to the classes of the fiber and of the zero section. Hence, it is a non trivial component of a reducible fiber. The roots lattice of the orthogonal complement of $\langle F,S\rangle\simeq U$ is isometric to $A_1^{\oplus 9}$, and it is generated by $N_i$, $i=1,\ldots, 8$ and $V$. So, $V$ is a component of a ``new" reducible fiber, which is necessarily of type either $I_2$ or $III$. Generically, under specializations, one obtains semistable fibers, so the elliptic fibration has $9I_2+6I_1$ as singular fibers (one can even check this directly on the equation of the elliptic fibration $\mathcal{E}$ given in Proposition \ref{prop: the family UNpolarized}). The torsion section remains a torsion section and the $(-2)$-curve $V$ is orthogonal to it. So the torsion section meets the ``new" $I_2$-fiber in the trivial component. The rank of the Mordell--Weil group is trivial by Shioda--Tate formula. The involution $\sigma$ switches the components and the singular points of the $I_2$-fibers with components $N_i$ (as $\sigma$ do on $\mathcal{E}$) and preserves the components and the two singular points of the ``new" $I_2$-fiber.

If $d=2$ and $NS(X_{11})$ is an overlattice of index two of $U\oplus N\oplus \langle -4\rangle$, the class $(V-N_1-N_2)/2$ corresponds to a $(-2)$-curve orthogonal to $\langle F,S\rangle$ and which intersects $N_1$ and $N_2$, so $(V-N_1-N_2)/2$, $N_1$ and $N_2$ are contained in the same reducible fiber. The remaining curves $N_i$ (with $i>2$) do not intersect $N_1$, $N_2$ and $(V-N_1-N_2)/2$. So we obtain 7 reducible fibers, six of them are of type $I_2$ and their non trivial components are $N_i$, $i=3,\ldots, 8$ and the seventh is of type $I_4$, indeed the root lattice orthogonal to $\langle F, S\rangle$ is isometric to $A_3\oplus A_1^{\oplus 6}$. The 2-torsion section intersects $N_i$, $i=3,\ldots, 8$ and by the height formula it is forced to intersect the $I_4$-fiber in the component which has trivial intersection with the trivial component, i.e. in the central component of $I_4$.

If $d>1$ and $NS(X_{11})=U\oplus N\oplus \langle -2d\rangle$, there are no $(-2)$-curves orthogonal to the class of the fiber, but the $N_i$'s. So the configuration of the reducible fibers does not change. The class $G_X:=dF+S+V$ is a $(-2)$-class and it is a section of the fibration. We observe $SG_X=d-2$ and $TG_X=(2F+S-\hat{N})G_X=d$. We recall that the Mordell--Weil lattice, denoted by $MWL$, is such that $|d(NS)|=|d(MWL)d(Tr)|/(|MW_{tors}|^2)$ where $Tr$ is the trivial lattice of the fibration (in our context $U\oplus A_1^{8}$) and $|MW_{tors}|$ is the order of the torsion part of the Mordell--Weil group.
Hence $|d(MWL)|=2d$. Computing the height formula on $G_X$ one obtains $2d$, so $G_X$ is a $\Z$-generator of the free part of Mordell Weil group.

If $d>2$, $d\equiv 2\mod 4$ and  $NS(X_{11})$ is an overlattice of index 2 of $U\oplus N\oplus \langle -2d\rangle$, $(V-N_1-N_2)/2\in NS(X)$ and $((V-N_1-N_2)/2)^2=-1-d/2$.
The class $G_X:=\frac{2+d}{4}F+S+\frac{V-N_1-N_2}{2}$ is a $(-2)$-class and it is a section of the fibration. It intersects the non trivial component of two fibers of type $I_2$ and the trivial component of the remaining $I_2$-fibers. Moreover, $SG_X=(d-6)/4$ and $TG_X=(d-2)/4$. 
As above one computes $d(MWL)=\frac{d}{2}$, which is indeed the value of the height formula computed on $G_X$ and this guarantees that $G_X$ is a generator of the free part of the Mordell--Weil group.

If $d>2$ and $d\equiv 0\mod 4$, the proof is similar, posing $G_X:=\frac{4+d}{4}F+S+\frac{V-N_1-N_2-N_3-N_4}{2}$.\endproof

\begin{remark}\label{rem: gluing 2 I2}{\rm In case (2) of the previous proposition, i.e. $d=2$ and $NS(X_{11})$ overlattice of index two of $U\oplus N\oplus \langle -4\rangle$, the specialization consists of gluing two fibers of type $I_2$ to obtain a fiber of type $I_4$.
Indeed, one adds the class $(V-N_1-N_2)/2$ with $V^2=-4$, which has a non trivial intersection with the classes $N_1$ and $N_2$. These three $(-2)$-classes span the lattice  $A_3$, but a priori we don't know if they correspond to irreducible curves. 

Let us denote by $A$, $B$, $C$ and $D$ the irreducible components of the $I_4$-fiber, where we assume that $AB=BC=CD=DA=1$ and that $D$ is the zero component, i.e. $DS=1$.

Since $T$ is a 2-torsion section, the height formula implies that $TB=1$ and the translation by $T$, maps $A$ to $C$ and $B$ to $D$.
The involution $\sigma$ induced by the van Geemen--Sarti involution on $X$, maps $N_i$ to $F-N_i$, and $(V-N_1-N_2)/2$ to $(V-2F+N_1+N_1)/2$. In particular, $N_1$ and $N_2$ can not be irreducible components of the $I_4$-fiber, otherwise the action of $\sigma$ and the one induced by the translation by $T$ do not coincide. This implies that the class $N_1$ (resp. $N_2$) corresponds to a reducible curve, i.e. the specialization of the N\'eron--Severi group given by the imposing that $(V-N_1-N_2)/2$ is an algebraic class, makes the curves $N_1$ and $N_2$ reducible.
More explicitly, to require that the action of $\sigma$ and of the translation by $T$ coincide, implies that $N_1=A+B$, $N_2=B+C$, $B=(V+N_1+N_2)/2$ and $D=F-A-B-C$. So $A=(-V+N_1-N_2)/2$ maps  to $C=(-V-N_1+N_2)/2$ and $B=(V+N_1+N_2)/2$ to $(V+F-N_1+F-N_2)/2=F-A-B-C$, accordingly to the action both of $\sigma$ and of the translation by the torsion section $T$.

The expression of the torsion section $T$ on $X_{11}$ can be written in terms of the irreducible components of the reducible fibers or in terms of the classes $N_i$; this gives the following equality
$$T=2F+S-\frac{1}{2}(N_3+N_4+N_5+N_6+N_7+N_8+A+2B+C)=T=2F+S-\hat{N}$$
which shows that the formal equation of $T$ in terms of $N_i$ is the same both on $X$ and on its specialization $X_{11}$. This guarantees that the van Geemen--Sarti involution on $X_{11}$ is the one induced, via specialization, by the one on $X$.}\end{remark}

\subsection{The surface $Y_{11}$}
We now consider the K3 surface $Y_{11}$ obtained as minimal resolution of $X_{11}/\sigma$. By Proposition \ref{prop: the family UNpolarized}, we know that $Y_{11}$ admits a van Geemen--Sarti involution and by \cite{In} we known that $\rho(Y_{11})=\rho(X_{11})=11$; it follows that $Y_{11}$ is a member of one of the families $\mathcal{L}_d$, $\mathcal{L}_d^{(a)}$, $\mathcal{L}_d^{(b)}$ for a certain $d\in\mathbb{N}_{>0}$.  We relate  $X_{11}$ with $Y_{11}$ first in terms and of their N\'eron--Severi groups and then of their elliptic elliptic fibrations.
\begin{theorem}\label{theorem: Y rho 11} Let $X_{11}$ be a K3 surface endowed with a van Geemen--Sarti involution $\sigma$, $\rho(X_{11})=11$ and $Y_{11}$ the minimal resolution of $X_{11}/\sigma$. Then $$\begin{array}{ccc}X_{11}\in \mathcal{L}_d,\ \ d\equiv 1\mod 2&\mbox{ if and only if }&Y_{11}\in\mathcal{L}_{2d}^{(a)},\\
X_{11}\in \mathcal{L}_d,\ \ d\equiv 0\mod 2&\mbox{ if and only if }&Y_{11}\in \mathcal{L}_{2d}^{(b)},\\
X_{11}\in \mathcal{L}_d^{(a)},\ \ d\equiv 2\mod 4&\mbox{ if and only if }&Y_{11}\in\mathcal{L}_{d/2},\\
X_{11}\in \mathcal{L}_d^{(b)},\ \ d\equiv 0\mod 4&\mbox{ if and only if }&Y_{11}\in \mathcal{L}_{d/2}.\end{array}$$
\end{theorem}
\proof To determine the family of $Y_{11}$ once one knows the one of $X_{11}$, one computes $T_{Y_ {11}}$, which is the minimal primitive sublattice of $H^2(Y_{11},\Z)$ containing $\pi_*(T_{X_{11}})$. Alternatively, one can compute directly $NS(Y_{11})$, which is the minimal primitive sublattice of $H^2(Y_{11},\Z)$ containing $\pi_*(NS(Y_{11}))\oplus N$, where $N$ is the Nikulin lattice generated by $n_i$. We present both the computations, since they are useful in different contexts.

{\bf Cases 1) and 2): $X_{11}\in\mathcal{L}_d$.} Let us assume that $X_{11}\in\mathcal{L}_d$, i.e. $NS(X_{11})\simeq U\oplus N\oplus \langle -2d\rangle$, $V=t_{11}-dt_{12}$
 and $T_{X_{11}}\simeq \langle t_i,\ i=1,\ldots,10,\ t_{11}+dt_{12}\rangle\simeq U\oplus N\oplus \langle2d\rangle$. 
Then $\pi_*(T_{X_{11}})$ is generated by $\pi_*(t_i)$, $i=1,\ldots,10$ and  $\pi_*(t_{11}+dt_{12})=\tau_{11}+d\tau_{12}$. Hence $\tau_i$, $i=1,\ldots,10$ are contained in $T_{Y_{11}}$ and since $\langle \tau_i, i=1,\ldots, 10, \tau_{11}+d\tau_{12}\rangle$ is primitive in $H^2(Y_{11},\Z)$ it coincides with $T_{Y_{11}}$.
The intersection form of $T_{Y_{11}}$ is $M\oplus U(2)\oplus \langle 4d\rangle$, (the matrix $M$ is the one in \eqref{eq: matrix M}), so $A_{T_{Y_{11}}}\simeq (\Z/2\Z)^4\times \Z/4d\Z$ and $T_{Y_{11}}\simeq U(2)\oplus D_8\oplus \langle 4d\rangle$. Since the discriminant form of $T_{Y_{11}}$ is the opposite of the one of $NS(Y_{11})$, we deduce $NS(Y_{11})$ and we obtain that $Y_{11}\in\mathcal{L}_{2d}^{(a)}$ if $d\equiv 1\mod 2$ and  $Y_{11}\in\mathcal{L}_{2d}^{(b)}$ if $d\equiv 0\mod 2$.

Let us compute directly $NS(Y_{11})$: $\pi_*(NS(X_{11}))$ is generated by $F_Y:=(\pi_*(F))/2$, $S_Y:=\pi_*(S)$ and $\pi_*(V)=\tau_{11}-d\tau_{12}$. In particular $\tau_{11}+n_1+n_2+n_3+n_4-d(\tau_{12}+n_1+n_2+n_3+n_5)$ is a class in $NS(Y_{11})$. But $\left(\tau_{11}+n_1+n_2+n_3+n_4-d(\tau_{12}+n_1+n_2+n_3+n_5)\right)/2$ is contained in $H^2(Y_{11},\Z)$ and since $NS(Y_{11})$ is a primitive sublattice of $H^2(Y_{11},\Z)$, one obtains that it is also contained in $NS(Y_{11})$. In particular, if $d$ is even $(\tau_{11}+n_1+n_2+n_3+n_4)/2\in NS(X_{11})$; if $d$ is odd $(\tau_{11}-\tau_{12}-n_4-n_5)/2\in NS(X_{11})$.

{\bf Case 3): $X_{11}\in\mathcal{L}_d^{(a)}$, $d=4k+2$} and   $NS(X_{11})$ is an overlattice of index 2 of $U\oplus N\oplus \langle -2d\rangle$. So $V=t_1+t_2+2t_{11}-2kt_{12}$ and $T_{X_{11}}\simeq \langle t_i, i=3,\ldots,7, t_1-t_2, t_2-t_8, t_1+t_{12}, t_{11}+kt_{12}, t_9, t_{10}\rangle$. Then $\pi_*(T_X)$ is generated over $\Q$ by $\pi_*(t_i), i=3,\ldots,7, \ \pi_*(t_1-t_2), \pi_*(t_2-t_8), \pi_*(t_1+t_{12})$, $\pi_*(t_9)$, $\pi_*(t_{10})$. So $T_{Y_{11}}$ is generated by $$\tau_i,\ i=4,\ldots,10,\ \tau_1-\tau_2,\ \pi_*(t_3)=\tau_1-2\tau_2+2\tau_3,\ \pi_*(t_2-t_8)=\tau_2-\tau_8,\ \pi_*(t_1+t_{12})=\tau_1+\tau_{12},\ \pi_*(t_{11}+kt_{12})=\tau_{11}+k\tau_{12}.$$
Alternatively, one observes that $T_{Y_{11}}$ is the orthogonal to $\pi_*(t_1+t_2-2t_{11}-2kt_{12})$ in $\langle \tau_i\rangle_{i=1,\ldots, 12}$.

To compute directly $NS(Y_{11})$ one observes that $\pi_*(NS(X_{11}))$ is generated by $F_Y:=(\pi_*(F))/2$, $S_Y:=\pi_*(S)$ and $\pi_*((V-N_1-N_2)/2)=\tau_2+\tau_{11}-k\tau_{12}-F$, so $NS(Y_{11})\simeq \pi_*(NS(X_{11}))\oplus N$. 

{\bf Case 4): $X_{11}\in\mathcal{L}_d^{(b)}$, $d=4k+4$.} 
The computation are similar to the previous cases: $T_{X_{11}}\simeq \langle t_i, i=5,6,7, t_1-t_2, t_2-t_3, t_3-t_4, t_3+t_4-t_8,  t_1+t_{12}, t_{11}+kt_{12}, t_9, t_{10}\rangle$. Then $\pi_*(T_{X_{11}})$ is generated by
$$\pi_*(t_5)=2\tau_5-2\tau_4+2\tau_3-2\tau_2+\tau_1,\ \ \ \tau_6,\ \  \tau_7$$
$$\pi_*(t_{i-1}-t_i)/2=-\tau_i+2\left(\sum_{j=1}^{i-2}(-1)^{j+1}\tau_{i-j}\right)+(-1)^{i}\tau_1,\ i=2,3,4,\ \ \pi_*(t_3+t_4-t_8)=2\tau_4-\tau_8,$$ 
$$ \pi_*(t_1+t_{12})=\tau_1+\tau_{12},\ \pi_*(t_{11}+kt_{12})=\tau_{11}+k\tau_{12},\ \  \pi_*(t_j)=\tau_j,\ j=9,10.$$
Its discriminant group is $(\Z/\Z)^6\times \Z/d\Z$ and $T_{Y_{11}}\simeq U(2)\oplus D_4\oplus D_4\oplus \langle d\rangle\simeq U\oplus N\oplus \langle d\rangle$. Since $\pi_*(NS(X_{11}))$ is generated by $F_Y:=(\pi_*(F))/2$, $S_Y:=\pi_*(S)$ and $\pi_*((V-N_1-N_2-N_3-N_4)/2)=\tau_2+\tau_3+\tau_{11}-k\tau_{12}-2F$, $NS(Y_{11})\simeq \pi_*(NS(X_{11}))\oplus N$. \endproof

By the previous theorem and the description of the elliptic fibrations in each of the families $\mathcal{L}_d$, $\mathcal{L}_d^{(a)}$, $\mathcal{L}_d^{(b)}$ given in Proposition \ref{prop: specialization elliptic fibrations, 11}, we obtain the following proposition, where $\mathcal{E}_{11}$ is the elliptic fibration on $X_{11}$ (specialization of $\mathcal{E}$); $\mathcal{F}_{11}$ is the elliptic fibration induced on the quotient surface $Y_{11}$; $G_X$ (resp. $G_Y$) is the generator of the infinite part of $MW(\mathcal{E}_{11})$ (resp. $MW(\mathcal{F}_{11})$), when the Mordell--Weil group is not simply a torsion group.
\begin{proposition}\label{proposition E11 and F11}
The Mordell--Weil groups of $\mathcal{E}_{11}$ and $\mathcal{F}_{11}$ are isomorphic.
	The section $G_X$ intersects any reducible fiber in the trivial component if and only if $G_Y$ intersects some fibers (2 or 4 according to the parity of $d/2$) in the non trivial component and viceversa.

If $G_XS=d-2$, then $G_YS=\frac{d-1}{2}$ if $d$ is odd, and $G_YS=\frac{d-2}{2}$ if $d$ is even, and vice versa.

	 The elliptic fibrations with fibers $9I_2+6I_1$ and $I_4+6I_2+8I_1$ are exchanged by the quotient by the van Geemen--Sarti involution. 	 
\end{proposition}
\proof The proposition follows by Proposition \ref{prop: specialization elliptic fibrations, 11}, where the elliptic fibrations are described in terms of the family of polarized K3 surfaces, and Theorem \ref{theorem: Y rho 11}, where the relation between the N\'eron--Severi groups of $X_{11}$ and of $Y_{11}$ is provided.\endproof

Denoted $\beta:Y_{11}\ra X_{11}/\sigma$, the desingularization of $X_{11}/\sigma$, in the following examples we describe more geometrically the relations between the elliptic fibrations $\mathcal{E}_{11}$ and $\mathcal{F}_{11}$.
\begin{example}\label{example: quotient I2}{\rm
	Let us consider the case $X_{11}\in\mathcal{L}_1$, i.e. $NS(X_{11})\simeq U\oplus N\oplus \langle -2\rangle$ (see case (1), Proposition \ref{prop: specialization elliptic fibrations, 11}). In this case $V=t_{11}-t_{12}$ geometrically corresponds to the non trivial class of the ``new"  $I_2$-fiber (the fiber for which the 2-torsion section intersects the trivial component). The class $\pi_*(V)=\tau_{11}-\tau_{12}$ has self intersection $-4$. Since $V$ is a component of the $I_2$-fiber of $\mathcal{E}_{11}$ where the van Geemen--Sarti involution preserves the components, $\sigma$ fixes two points on $V$, hence $\pi_*(V)$ passes through two singular points of $X_{11}/\sigma$.
 	
	Blowing up these two points, one obtains a fiber of type $I_4$ on $Y_{11}$, whose components are: the strict transform of image of $F-V$, called $D$, the strict transform of the image of $V$, called $B$, two exceptional curves called $A$ and $C$ (these are two of the curves $n_i$ on $Y_{11}$). With this notation $AB=BC=CD=DA=1$ and $D$ is the zero component of an $I_4$-fiber.
	One obtains $\pi_*(V)=2B$, $\beta^*(\pi_*(V))=2B+A+C$, and that the strict transform on $Y_{11}$ of $\pi(V)$ is $(\beta^*(\pi_*(V))-A-C)/2$.
	
	As observed, the curves $A$ and $C$ are exceptional curves in $Y_{11}$, so they are among the curves $n_i$. To be consistent with the notation introduced in the proof of Theorem \ref{theorem: Y rho 11} we have that $\{A,C\}=\{n_4,n_5\}$ and that to generate the N\'eron--Severi group of $Y_{11}$ we need $\pi_*(NS(X))$ and $(\beta^*(\pi_*(V))-n_4-n_5)/2$.

Comparing this construction with the Remark \ref{rem: gluing 2 I2}, one observes that they are consistent, indeed the $-4$ vector added in that Remark to glue two fibers of type $I_2$ to one of type $I_4$ has the same properties (and relations with the $I_4$-fiber components) as the ones of $\pi_*(V)$ described now.}
\end{example}
\begin{example}\label{example: quotient I4}{\rm
		Let us consider the case $X_{11}\in\mathcal{L}_2^{(a)}$, i.e. $NS(X_{11})\simeq (U\oplus N\oplus \langle -4\rangle)'$ (see case (2), Proposition \ref{prop: specialization elliptic fibrations, 11}). In this case $V=t_1+t_2$, which allows one to construct a fiber $I_4$ with components $A$, $B$, $C$, $D$ as in the Remark \ref{rem: gluing 2 I2}. The van Geemen--Sarti involution acts on these curves as follow: $\sigma(A)=C$, $\sigma(B)=D$. There are no points fixed by $\sigma$ in this fiber. In the quotient, the components $B$ and $D$ (resp. $A$ and $C$) are identified, so the quotient fiber has two components and it is an $I_2$-fiber; the image of $B$ and $D$ is the trivial component $M_0$ of the quotient $I_2$-fiber, and the image of $A$ and $C$ is the non trivial component $M_1$ of the same fiber. In particular, $\beta^*\pi_*(A)=\pi_*(A)=\pi_*(C)=\beta^*\pi_*(C)=M_0$ and $\beta^*\pi_*(B)=\pi_*(B)=\pi_*(D)=\beta^*\pi_*(D)=M_1$.}
\end{example}		

\begin{example}\label{example: quotient section infinite}{\rm Let us consider the case $X_{11}\in\mathcal{L}_d$, $d>1$ i.e. $NS(X_{11})\simeq U\oplus N\oplus \langle -2d\rangle$. In this case $V=t_{11}-dt_{12}$. By Proposition \ref{prop: specialization elliptic fibrations, 11}, there is an infinite order section $G_X$, whose class is $dF+S+V$. The class $\pi_*(G_X)=2dF_Y+S_Y+\tau_{11}-d\tau_{12}$ corresponds to a section of the fibration $Y_{11}$. This section splits in the double cover $X_{11}\ra Y_{11}$ and indeed $\pi^*(\pi_*(G_X))=G_X+G_X'$, where $G_X'=dF+T+V=\sigma(G_X)$. Even if $G_X$ is a generator of the Mordell--Weil group of the elliptic fibration on $X$, $\pi_*(G_X)$ is not $G_Y$, i.e. it is not a generator of the Mordell--Weil group. Indeed, the heigh formula computed on $\pi_*(G_X)$ gives $2d-4$, but the Mordell--Weil lattice of $\mathcal{F}_{11}$ is $[d]$. The generator of the Mordell--Weil group is $G_Y:=\frac{d+1}{2}F_Y+S_Y+\frac{\tau_{11}-d\tau_{12}-n_4-n_5}{2}$.
The pull back $\pi^*(G_Y)=\frac{d+1}{2}F+S+T+t_{11}-dt_{12}$ corresponds to a rational curve on $X$ which is a bisection of the elliptic fibration $\mathcal{E}$. This bisection passes through two of the points fixed by $\sigma$ and hence in the quotient it meets two exceptional divisors, $n_4$ and $n_5$.}\end{example}

\section{Some specializations of $X$ and $Y$ to K3 surfaces with  Picard number 12}\label{sec: K3 with vGS rank 12}

We now consider subfamilies of $\mathcal{N}$ of codimension 2, so let $X_{12}$ be a K3 surface with a van Geemen--Sarti involution $\sigma$ and $\rho(X_{12})=12$ and $Y_{12}$ the desingularization of the quotient $X_{12}/\sigma$. As in the previous section we assume that the van Geemen--Sarti involution specializes the one on $X$, where $X$ is a general member of $\mathcal{N}$, i.e. we assume that $X_{12}$ is marked with $\Phi$ such that  $\varphi(U\oplus N)\subset \Phi(NS(X_{12}))$, where $\varphi$ is given in Table \ref{eq: U+Nin Lambda}.

First, we describe the general structure of the Picard group of $X_{12}$. Then, we identify some particular subfamilies which satisfy the property $NS(X_{12})\simeq NS(Y_{12})$, in Theorem \ref{theorem: rank 12, equal NS}. 

\begin{lemma}\label{lemma: lattice Omega and rank 12}
Let $\Omega$ be a rank 2 negative definite even lattice. Then, $\Omega$ admits a primitive embedding in $T_{X}\simeq U\oplus U\oplus N$ and there exists a unique primitive embedding of $U\oplus N\oplus \Omega$ in $\Lambda_{K3}$. 

For each element $\omega\in A_{\Omega}$ such that $\omega$ has order 2 and $\omega^2\in\Z$, there exists a unique overlattice of index 2, $(U\oplus N\oplus \Omega)'$, of $U\oplus N\oplus \Omega$ in which both $U\oplus N$ and $\Omega$ are primitively embedded. Moreover, there exists a unique primitive embedding of $(U\oplus N\oplus \Omega)'$ in $\Lambda_{K3}$.
\end{lemma}
\proof The lattice $\Omega(-1)$ has the same discriminant group as $\Omega$ and opposite discriminant form. It follows that $U\oplus U$ is an overlattice of index $|A_{\Omega}|$ of $\Omega\oplus \Omega(-1)$ and that $\Omega$ admits a primitive embedding in $U\oplus U$. Hence $\Omega$ admits at least one primitive embedding (possibly more than one) in $U\oplus U\oplus N\simeq T_{X}$. We observe that $\rk(U\oplus N\oplus \Omega)=12$ and $\ell(U\oplus N\oplus \Omega)=\ell(N)+\ell(\Omega)\leq 8$. So $U\oplus N\oplus \Omega$ admits a unique primitive embedding in $\Lambda_{K3}$, by \cite{NikInt}.

To construct an overlattice of finite index of $U\oplus N\oplus \Omega$ as required, one has to find an isotropic vector in $A_{U\oplus N\oplus \Omega}\simeq A_{N}\oplus A_{\Omega}$ which has non trivial components both on $A_{N}$ and $A_{\Omega}$. Since $A_{N}\simeq (\Z/2\Z)^6$ contains only vectors of order 2 and with integer self intersection, each element in $A_{N}$ can be glued only with an element $\omega\in A_{\Omega}$ with the same properties. There are just two orbits of non trivial elements in $A_N$ for the action induced by $O(U\oplus N)$ on $A_{N}$: one contains all the non trivial vectors of self intersection 0, the other contains the ones of self intersection 1. So, chosen an element $\omega\in A_{\Omega}$, up to isometries it can be glued with a unique element in $A_N$, obtaining a unique overlattice of $U\oplus N\oplus \Omega$. Since $\ell\left((U\oplus N\oplus \Omega)'\right)<\ell(U\oplus N\oplus \Omega)=\ell(N)+\ell(\Omega)\leq 8$, each lattice $(U\oplus N\oplus \Omega)'$ admits a unique primitive embedding in $\Lambda_{K3}$.\endproof

Let $\mathcal{S}^0_{\Omega}$ (resp. $\mathcal{S}^1_{\Omega}$) be the set of the elements in $A_{\Omega}$ of order 2 and self intersection equal to 0 (resp. 1), considered up to isometries of $\Omega$. By the proof of the previous lemma each element in $\mathcal{S}^0_{\Omega}$ (resp. $\mathcal{S}^1_{\Omega}$) determines an overlattice of index 2 of $U\oplus N\oplus \Omega$, denoted $(U\oplus N\oplus \Omega)'$.
\begin{proposition}\label{prop: rank 12 all the cases}
	Let $\Omega$ be a rank 2 negative definite even lattice. There exists a 10-dimensional family of K3 surfaces $X_{\Omega}$ such that $NS(X_{\Omega})\simeq U\oplus N\oplus \Omega$ and all the K3 surfaces in such a family admits a van Geemen--Sarti involution. Moreover $T_{X_{\Omega}}\simeq \Omega(-1)\oplus N$.
	
	For each element $\omega\in\mathcal{S}_{\Omega}^0$ (resp. $\mathcal{S}_{\Omega}^1$) there exists  a 10-dimensional family of K3 surfaces $X_{\Omega}'$ such that $NS(X_{\Omega}')\simeq \left(U\oplus N\oplus \Omega\right)'$ and all the K3 surfaces in such a family admit a van Geemen--Sarti involution.
\end{proposition}
\proof We already observed that there exists a primitive embedding of $\Omega$ in $U\oplus U\subset T_X$. The orthogonal complement of $\Omega$ in $T_{X}$ is isometric to $\Omega(-1)\oplus N$ and it is the transcendental lattice of a K3 surface $X_{12}$, whose Picard number is 12. Its N\'eron--Severi lattice is $U\oplus N\oplus \Omega\simeq NS(X)\oplus \Omega$, since there are no gluing relation between $NS(X)$ and the two copies of $U$ in $T_{X}$. The family has dimension 10 since the Picard number of the general member is 12 and its members admit a van Geemen--Sarti involution since the family is properly contained in the family $\mathcal{N}$ of the $(U\oplus N)$-polarized K3 surfaces.

It is even possible that $\Omega$ is embedded in $T_{X}$, but not in the direct summands $U\oplus U$ of $T_{X}$. Let us fix an embedding $\varepsilon:\Omega\hookrightarrow T_{X}$. The orthogonal complement of $\varepsilon(\Omega)$ in $T_{X}$ is the transcendental lattice of a K3 surface $X_{\Omega}'$, whose N\'eron--Severi group is by construction an overlattice of finite index (possibly 1) of $U\oplus N\oplus \Omega$. Hence  $NS(X_{\Omega}')\supseteq U\oplus N\oplus \Omega$, and if $NS(X_{\Omega}')\not\simeq U\oplus N\oplus \Omega$, then $NS(X_{\Omega}')$ is an overlattice of index 2 of $U\oplus N\oplus \Omega$ as described in Lemma \ref{lemma: lattice Omega and rank 12}. Since up to isometries there are only two non trivial classes of elements in $A_{U\oplus N}$, once one fixes an element in $\mathcal{S}_{\Omega}^0$ (resp. $\mathcal{S}_{\Omega}^1$) one obtains a unique overlattice $(U\oplus N\oplus \Omega)'$. This determines the discriminant form of both $NS(X_{\Omega}')\simeq (U\oplus N\oplus \Omega)'$ and its orthogonal and by \cite[Theorem 1.14.4]{NikInt}, one has that both are primitively embedded in $\Lambda_{K3}$, hence they determine a unique family of K3 surfaces. By construction $U\oplus N$ is primitively embedded in $NS(X'_{\Omega})$, so $X'_{\Omega}$ admits a van Geemen--Sarti involution. 
\endproof

We now consider certain specializations, the ones in which we add an orthogonal copy of $\langle-2e\rangle$ to the N\'eron--Severi group of type $\Lambda_d^{(a)}$, $\Lambda_d^{(b)}$ of the surfaces of Picard number 11 considered in the previous section.
With the notation of the previous Proposition, we are assuming $\Omega=\langle -2d\rangle\oplus \langle -2e\rangle$ and that $NS(X_{12})$ is a prescribed overlattice of index 2 of $U\oplus N\oplus \Omega$.
\begin{theorem}\label{theorem: rank 12, equal NS}
	Let $X_{12}$ be a K3 surface whose N\'eron--Severi group is either  $NS(X_{12})\simeq \Lambda_d^{(a)}\oplus \langle -2e\rangle$ or $NS(X_{12})\simeq \Lambda_d^{(b)}\oplus \langle -2e\rangle$ and so $T_{X_{12}}\simeq D_4\oplus D_4\oplus \langle 2d\rangle\oplus \langle 2e\rangle$. Then $$NS(Y_{12})\simeq\left\{\begin{array}{r} \langle -d\rangle \oplus (U\oplus N\oplus  \langle -4e\rangle)'=\Lambda_{2e}^{(a)}\oplus \langle-d\rangle\mbox{ if }e\equiv 1\mod 2\\\langle -d\rangle \oplus (U\oplus N\oplus  \langle -4e\rangle)'=\Lambda_{2e}^{(b)}\oplus \langle-d\rangle\mbox{ if }e\equiv 0\mod 2 \end{array}\right.\mbox{ and  }T_{Y_{12}}\simeq D_4\oplus D_4\oplus \langle d\rangle\oplus \langle 4e\rangle.$$ 
	In particular, if $d=2e$, then $T_{Y_{12}}\simeq T_{X_{12}}$ and $NS(Y_{12})\simeq NS(X_{12})$. 
\end{theorem}

\proof
We give a detailed proof in the case $\Lambda_d^{(a)}$, the case $\Lambda_d^{(b)}$ is analogue.
Let $W$ be the generator of the summand $\langle -2e\rangle$ in $NS(X_{12})$. Since we are considering $X_{12}$ as a specialization of $X_{11}$ such that $NS(X_{11})\simeq \Lambda_{d}^{(a)}$, $W$ is contained in $T_{X_{11}}$ and we can chose it to be $W=t_9-et_{10}$.
In this case, recalling that $d=4k+2$, the transcendental lattice of $X_{12}$ is generated by $t_i$, $i=3,\ldots, 7$, $t_1-t_2$, $t_2-t_8$, $t_1+t_{12}$, $t_{11}+kt_{12}$, $t_9+et_{10}$ and it is $D_4\oplus D_4\oplus\langle 2d\rangle\oplus\langle 2e\rangle$. Its discriminant is $2^6de=2^7(2k+1)e$.

The transcendental lattice of $Y_{12}$ is generated by $\tau_i$, $i=4,\ldots, 7$, $\tau_1-\tau_2$, $\tau_1-2\tau_2+2\tau_3$, $\tau_2-\tau_8$, $\tau_1+\tau_{12}$, $\tau_{11}+k\tau_{12}$, $\tau_9+e\tau_{10}$ and its discriminant is $2^7(2k+1)e$.
Since one has a basis of $T_{X_{12}}$ and $T_{Y_{12}}$ one can compute directly the intersection form and the discriminant form and one obtains the statement.
\endproof

\begin{example}\label{ex: I47I2}{\rm  Let us consider $X_{12}$ such that $NS(X_{12})\simeq NS(Y_{12})\simeq \Lambda_2^{(a)}\oplus \langle -2\rangle$, i.e.  $d=2e=2$. By applying twice  Proposition \ref{prop: specialization elliptic fibrations, 11}, one obtains that the elliptic fibration $\mathcal{E}_{12}$ is obtained by specializing $\mathcal{E}$ first gluing two $I_2$-fibers and then gluing two $I_1$-fibers. So the singular fibers of $\mathcal{E}_{12}$ are $I_4+7I_2+6I_1$. The torsion section meets the fiber $I_4$ and six of the fibers $I_2$ non trivially. The quotient of an $I_{2n}$ fiber by a non trivial 2-torsion translation is a fiber $I_n$. So in the quotient one has an $I_2$-fiber (quotient of the $I_4$) and six $I_1$-fibers (quotient of six fibers of type $I_2$). The quotient of the six $I_1$-fibers consists of six singular fibers, whose resolution produces six $I_2$-fibers on $Y_{12}$. On the remaining $I_2$ the 2-torsion section is trivial and so the van Geemen--Sarti involution fixes the two singular points of the fiber. Hence their image is singular in the quotient and so blown up in $Y_{12}$. This produces an $I_4$-fiber on $Y_{12}$ (two components are the exceptional divisors, the other two are the image of the components of the $I_2$-fiber on $X_{12}$, see Example \ref{example: quotient I2}). In particular, $Y_{12}$ has an elliptic fibration whose singular fibers are $7I_2+I_4+6I_1$, i.e. the same configuration of the fibration on $X_{12}$.}
\end{example}

\begin{prop}\label{prop: isom quotients rank 12}{\rm (See also \cite[Paragraph 6.6]{vGSc})}
There is a 4-dimensional subfamily of the family $\mathcal{L}$ introduced in Proposition \ref{prop: subfamily isomorphic quotients} , which corresponds to K3 surfaces $X_{12}$ with an elliptic fibration $\mathcal{E}_{12}$ such that $MW(\mathcal{E}_{12})=\Z/2\Z$ and the singular fibers of $\mathcal{E}_{12}$ are $I_4+7I_2+6I_1$. For each member $X_{12}$ of this family, $X_{12}\simeq Y_{12}$ and it is a K3 surface with complex multiplication by $\mathbb{Q}(\sqrt{-2})$.

The N\'eron--Severi group of the generic member of this family is $\Lambda_2^{(a)}\oplus \langle-2\rangle$.\end{prop}
\proof 
In Proposition \ref{prop: subfamily isomorphic quotients}, we recalled that the general member of the family $\mathcal{L}\subsetneq\mathcal{N}$ has Weierstrass equation $y^2=x(x^2+2\alpha(t^2)x+\frac{1}{2}\alpha^2(t^2)+t\beta(t^2))$ and it is isomorphic to the desingularization of its quotient by translation by the two torsion section $(x,y)=(0,0)$. This Weierstrass equation depends on the choice of the polynomials $\alpha$ and $\beta$ of degree 2 and 3 respectively and each specialization of these polynomials gives a subfamily of K3 surfaces which are isomorphic to their quotient by the van Geemen--Sarti involution.

The discriminant of this elliptic fibration is
$$\frac{1}{2}\left(\alpha^2(t^2)+2t\beta(t^2)\right)^2\left(\alpha^2(t^2)-2t\beta(t^2)\right).$$
The singular fibers of type $I_2$ are over the zeros of $\alpha^2(t^2)+2t\beta(t^2)$, the ones of type $I_1$ over the zeros of 
$\alpha^2(t^2)-2t\beta(t^2)$. Notice that if $\overline{t_0}$ is a zeros of one of these two polynomials, then $-\overline{t_0}$ is a zero of the other.
So, if we require that $\alpha^2(t^2)+2t\beta(t^2)$ has a multiple root, the same holds true for $\alpha^2(t^2)-2t\beta(t^2)$, and vice versa. 

To impose a double root to $\alpha^2(t^2)+2t\beta(t^2)$ is equivalent to glue two $I_2$-fibers to an $I_4$ and similarly to impose a double root to $\alpha^2(t^2)-2t\beta(t^2)$ is equivalent to glue two $I_1$-fibers to an $I_2$.

Therefore, if we impose the condition that $\alpha^2(t^2)+2t\beta(t^2)$ has a double root, we obtain a codimension 1 subfamily of $\mathcal{L}$, which consists of K3 surfaces with an elliptic fibration with singular fibers $I_4+7I_2+6I_1$ and which are isomorphic to their quotient by the van Geemen--Sarti involution. Since generically the surfaces with equation \eqref{eq: Weierstrass siom quotient} have no section of infinite order, the same hold true for the generic member of the subfamily we are considering. The N\'eron--Severi can be computed directly by the knowledge of the reducible fibers, see also Example \ref{ex: I47I2}.\endproof
\begin{rem}\label{rem: first specialization real multiplication}{\rm A result similar to the one of the previous proposition holds also for the family $\mathcal{L}'$ with Real Multiplication described in Remark \ref{rem: real multiplication}: the discriminant of the generic member of the family is $t^3\left(2t\gamma^2(t^2)-4\delta(t^2)\right)\left(2t\gamma^2(t^2)+4\delta(t^2)\right)^2$ and to require that the polynomial $\left(2t\gamma^2(t^2)-4\delta(t^2)\right)$ admits a multiple root is equivalent to require that the polynomial $\left(2t\gamma^2(t^2)+4\delta(t^2)\right)$ has a multiple root, so that one obtains a specialization of the K3 surfaces in this family whose singular fibers are $2III+I_4+5I_2+4I_1$.}\end{rem}

In Example \ref{ex: I47I2} and Proposition \ref{prop: isom quotients rank 12}, we considered surfaces $X_{12}$ as in Theorem \ref{theorem: rank 12, equal NS} (i.e. for which $NS(X_{12})\simeq NS(Y_{12})$) with the extra restriction that $d=2e=2$. But the result of Proposition \ref{prop: isom quotients rank 12} holds more in general (without the restriction on that the values of $d=2e$ is necessarily 2). Indeed in \cite[Paragraph 6.6]{vGSc} specializations of the family $\mathcal{L}$ which imply the presence of sections of infinite order are considered. It follows that the family of the K3 surfaces with Picard group $\Lambda_{2e}^{(a)}\oplus\langle -2e\rangle$ (resp. $\Lambda_{2e}^{(b)}\oplus \langle-2e\rangle$) intersects non trivially the aforementioned family, i.e. there are K3 surfaces $X_{12}$ such that $NS(X_{12})\simeq \Lambda_{2e}^{(a)}\oplus\langle -2e\rangle$ (resp. $\Lambda_{2e}^{(b)}\oplus \langle-2e\rangle$)  and $X_{12}\simeq Y_{12}$ for every positive integer $e$.

\subsection{Self maps of the transcendental lattice and specializations}\label{subsectio: map nu}
We use the notation $X$ to indicate a K3 surface whose N\'eron--Severi group is exactly $U\oplus N$ and hence whose transcendental lattice $T_X\simeq U\oplus U\oplus N$. We will consider the specialization of $X\in\mathcal{L}$ to $X_{12}\in\mathcal{L}$, in order to describe a specific map $\gamma:T_X \ra T_X$. 

If $X\in\mathcal{L}$, then $Y\simeq X$. So, the quotient map $\pi:X\ra Y\simeq X$ induces the maps $\pi^*:(T_Y)_{\Q}\simeq (T_X)_{\Q}\ra (T_X)_{\Q}$ and  $\pi_*:(T_X)_{\Q}\ra (T_Y)_{\Q}\simeq (T_X)_{\Q}$, i.e. two self maps of $(T_X)_{\Q}$. More precisely, if $t\in T_X$, then $\pi_*(t)\in T_Y\simeq T_X$ and hence we can describe it not just as a linear combination of the basis $\tau_i$ of $\pi_*(H^2(X,\Z))$, but also as a  linear combination of the basis $t_i$ of $H^2(X,\Z)$ (see Subsection \ref{subsec: cohomological action} for the notation). We denote $\gamma:T_X\ra T_X$ the map $\gamma(t):=\pi_*(t)$, where $t\in T_X$ and $\pi_*(t)$ is considered as an element in $T_X$. 
To describe explicitly this map, one can consider specializations of $X$. Indeed, let us consider, for example, a K3 surface $X_{12}$ as in Proposition \ref{prop: isom quotients rank 12}. It is endowed with a van Geemen--Sarti which specializes the one of $X$ and its  N\'eron--Severi group is obtained by considering two classes, say $a_1$ and $a_2$ in $T_X$ and requiring that these two classes become algebraic (see proof of Theorems \ref{theorem: X rho 11}, \ref{theorem: rank 12, equal NS}). Since the action of the van Geemen--Sarti involution on $X_{12}$ is induced by the one on $X$, $\sigma^*$ acts in the same way on $H^2(X,\Z)$ and on $H^2(X_{12},\Z)$ and it determines the action of the map $\pi_*$. 
The images $\gamma(a_i)$, $i=1,2$, of the classes $a_i$ under $\gamma$ in $(T_X)_{\Q}\simeq (T_{X_{12}}\oplus \langle a_1,a_2\rangle)_{\Q}\subset H^2(X,\Z)_{\Q}$ does not depend on the fact that we are considering $\sigma$ as an involution on $X$ or on $X_{12}$.

If one considers $\gamma$ on $H^2(X_{12},\Z)$, both $a_i$ and $\gamma(a_i)$ are algebraic classes (since the quotient map preserves the Hodge decomposition). Hence it is possible to describe the action of $\gamma$ on the classes $a_i$ as a geometric action of $\sigma$ on certain classes in the N\'eron--Severi group of $X_{12}$, i.e. on certain curves appearing in the elliptic fibration on $X_{12}$. 

We now investigate this phenomenon for the specialization to $X_{12}$ considered in Proposition \ref{prop: isom quotients rank 12}, and, in the next section, we reconsider the problem in a more general setting specializing $X$ to obtain a K3 surface with Picard rank 20.

\begin{example}\label{rem: first specialization and action of complex multiplication}\label{rem: map nu}{\rm
Let us consider the map $\gamma$ for the surface described in the Example \ref{ex: I47I2}, i.e. in the case $NS(X_{12})\simeq \Lambda_d^{(a)}\oplus \langle-2e\rangle$ with $d=2e=2$. The map $\gamma$ acts on $(T_X)_{\Q}$ and restricts to an action on $\left(T_{X_{12}}\right)_{\Q}$.  Since $NS(X_{12})_{\Q}\simeq (NS(X_{10})\oplus \langle -4\rangle\oplus \langle -2\rangle)_{\Q}$, where $\langle -4\rangle\oplus \langle -2\rangle\subset T_{X}$ and $T_{X_{12}}$ is the orthogonal complement to $\langle -4\rangle\oplus \langle -2\rangle$ in $T_{X}$, $\gamma$ restricts to an action on $\langle -4\rangle\oplus \langle -2\rangle$. Denoted $W_1$, $V_1$ the generators of $\langle -4\rangle\oplus \langle -2\rangle$, the action of $\gamma$ is $V_1\ra W_1$, $W_1\ra 2V_1$. Indeed, this is the geometric action of the quotient by the van Geemen--Sarti involution: we recall that $V_1$ is the non trivial component of the $I_2$-fiber where the torsion section is trivial. By Example \ref{example: quotient I2}, in the quotient it is mapped to $\beta^*\pi_*(V_1)=2B+A+C$ (where, with the same notation of the example, $A$, $B$, $C$ are the non trivial components of the fiber of type $I_4$, with the assumption that $B$ does not meet the trivial component). Since we are considering the map onto $NS(X)^{\perp_{NS(X_{12})}}=\langle -4\rangle \oplus \langle -2\rangle$, we consider the projection of $\beta^*\pi_*(V_1)=2B+A+C$ to the subspace of $NS(X_{12})$ which is orthogonal to $NS(X)\simeq \langle F,S,N_i\rangle$: we obtain $A+C$. By Remark \ref{rem: gluing 2 I2}, $-A-C$ is the $-4$ class orthogonal to $NS(X)$ added to obtain a fiber of type $I_4$ (i.e. the class called $V$ in the Remark \ref{rem: gluing 2 I2}). Assuming $W_1$ to be the opposite of the $(-4)$-class considered in Remark \ref{rem: gluing 2 I2}, we obtain that the class $V_1$ is mapped to $W_1$ by $\gamma$.
By Example \ref{example: quotient I4}, the class $W_1=A+C$ is mapped to $2V_1$ since $\pi_*(A)=\pi_*(C)=M_1$ is the non trivial component of the $I_2$-fiber and $V_1$ is exactly the non trivial component $M_1$ of the $I_2$-fiber on which the torsion section is trivial.  }\end{example}
	
\section{The main results and a specialization of $X$ and $Y$ to a K3 surface with  Picard number  20}\label{sec:Following specializations of $X$ to a K3 surface of Picard number 20}
		
		In this section we first discuss K3 surfaces $X\in\mathcal{L}$ contained in the family of K3 surfaces admitting complex multiplication by $\sqrt{-2}$ described in Proposition \ref{prop: subfamily isomorphic quotients}. We describe the action of the quotient map and of the complex multiplication on their transcendental lattice, see Theorem \ref{theorem: action selfmap}.
Then we deduce more general results on K3 surfaces $X\in \mathcal{N}$ contained in the family of K3 surfaces admitting a van Geemen--Sarti involution. More specifically, 
		in order to prove Theorem \ref{theorem: action selfmap} we describe further specializations of K3 surfaces $X\in\mathcal{L}$, by considering their explicit equations. These specializations are obtained by gluing fibers on the elliptic fibration of K3 surfaces in $\mathcal{L}$.  The analogue gluing of fibers on the elliptic fibration of the K3 surfaces $X$ also exists  if $X$ is a generic member of the family $\mathcal{N}\supsetneq\mathcal{L}$; its description from a lattice theoretic point of view is the same if one considers the surface as a member of $\mathcal{L}$ or as a member of $\mathcal{N}$. In Proposition \ref{prop: the five specializations} we state the existence of these specializations for $X\in\mathcal{N}$
		and this leads to Corollary \ref{cor: NSX=NSY condition}, which provides a condition under which the N\'eron--Severi groups of a specialization $Z_{2b}$ (of $X$ in $\mathcal{N}$) and of the minimal resolution of its quotient by a van Geemen--Sarti involution, are isometric.

	\begin{theorem}\label{theorem: action CM}\label{theorem: action selfmap}\label{theorem: action CM}
		Let $X\in\mathcal{L}$ with $\rho(X)=10$ and $\Gamma$ the lattice $\left(\langle -4\rangle\oplus \langle -2\rangle\right)^{\oplus 4}\bigoplus\left(\langle -8\rangle\oplus \langle-4\rangle\right)\bigoplus \left(\langle 8\rangle\oplus \langle4\rangle\right).$
		\begin{enumerate}\item  There exists a basis of $(T_X)_\Q$ on which the bilinear form is isometric to $\Gamma$ and the map $\gamma:\Gamma\ra \Gamma$ induced by $\pi_*$, where $\pi$ is the rational quotient map $X\ra Y=\widetilde{X/\sigma}\simeq X$, acts as  $\left[\begin{array}{rr}0&1\\2&0\end{array}\right]^{\oplus 5}\oplus \gamma_+$ with $\gamma_+:\langle 8\rangle\oplus \langle4\rangle\ra \langle 8\rangle\oplus \langle4\rangle$.
			\item 	The endomorphism $\nu:(T_X)_\Q\ra (T_X)_\Q$ giving the complex multiplication acts on $(\Gamma)_{\Q}\simeq (T_X)_\Q$ 
 via the block matrix $\left[\begin{array}{rr}0&1\\-2&0\end{array}\right]^{\oplus 6}$.
		\end{enumerate}
		Both the results in (1) and in (2) remain true even if $\rho(X)>10$, by substituting for $T_X$ the lattice which is the orthogonal complement to $\varphi(U\oplus N)$ in $H^2(X,\Z)\simeq \Lambda_{K3}$ where $\varphi:U\oplus N\hookrightarrow NS(X)$ is the embedding determined in Table \ref{eq: U+Nin Lambda}.		
	\end{theorem}

	Recall that if $X\in \mathcal{L}$ and $\rho(X)=10$, then $T_X$ is isometric to the transcendental lattice of the very general members $S$ of the family $\mathcal{N}$. In particular $(T_S)_{\Q}\simeq (T_X)_{\Q}\simeq (\Gamma)_{\Q}$. 
	The $\Q$-linear extension of the map $\gamma$ acts on $(\Gamma)_{\Q}$, and via the isometry $(\Gamma)_{\Q}\simeq (T_S)_{\mathbb{Q}}$, we obtain that $\gamma$ acts on $(T_S)_{\mathbb{Q}}$. By the previous theorem one knowns that $\gamma$ is induced by the quotient map $\pi$ and we explicitly knows how it acts on $(\Gamma)_{\Q}\simeq (T_S)_{\Q}$. This allows to identify proper subfamilies of $\mathcal{N}$ whose generic member is in the same family of the resolution of its quotient by a van Geemen--Sarti involution: these are the families characterized by N\'eron--Severi groups and transcendental lattices invariant for $\gamma$.

	\begin{corollary}\label{cor: NSX=NSY condition}
		Let $Z_{2b}$ be a K3 surface with $\rho(Z_{2b})=2b$ and $Z_{2b}\in\mathcal{N}$. So it admits a van Geemen--Sarti involution $\sigma_Z$ which specializes the van Geemen--Sarti involution defined on very general members $X\in\mathcal{N}$. If $T_{Z_{2d}}$ is invariant for (the $\Q$-linear extension of) $\gamma$, then $NS(Z_{2d})\simeq NS(\widetilde{Z_{2d}/\sigma_Z})$ and $T_{Z_{2d}}\simeq T_{\widetilde{Z_{2d}/\sigma_Z}}$. \end{corollary}
	
	To prove Theorem \ref{theorem: action CM} and Corollary \ref{cor: NSX=NSY condition} we consider a sequence of specializations, which are also summarized in the following proposition where we denote (as in the previous section) by $X_{n}$ a K3 surface with Picard number $n$ which is a specialization of $X$ and which  admits a van Geemen--Sarti involution $\sigma$, specialization of the one defined on $X$. In particular on $X_n$ there is an elliptic fibration $\mathcal{E}_n$, with a 2-torsion section, which specializes $\mathcal{E}$. We also denote by $Y_n$ the minimal model of $X_{n}/\sigma$.
	\begin{proposition}\label{prop: the five specializations}
		There exist the following specializations of $X$:
		\begin{itemize}\item for $i=1,2,3,4$, there is a $(10-2i)$-dimensional family of K3 surfaces whose generic members $X_{10+2i}$  have $\rho(X_{10+2i})=10+2i$ and admit a van Geemen--Sarti involution $\sigma$ such that $NS(X_{10+2i})\simeq NS(Y_{10+2i})$; the singular fibers of $\mathcal{E}_{10+2i}$ are $iI_4+(8-i)I_2+(8-2i)I_1$  and $MW(\mathcal{E}_{10+2i})$ is a torsion group. There is a $(5-i)$-dimensional subfamily for which moreover $X_{10+2i}\simeq Y_{10+2i}$.
			\item $X_{20}$ is a K3 surface with  $\rho(X_{20})=20$ which admits a van Geemen--Sarti involution $\sigma$ such that $X_{20}\simeq Y_{20}$; the singular fibers of $\mathcal{E}_{20}$ are $I_8+3I_4$ and  $MW(\mathcal{E}_{20})=\Z/4\Z\times \Z/2\Z$. The transcendental lattice of $X_{20}$ is $\langle 8\rangle\oplus\langle 4\rangle$.
		\end{itemize}
	\end{proposition}

	\subsection{Proof of Theorem \ref{theorem: action selfmap}}\label{subset: proof of theorem}
	\subsubsection{Strategy of the proof}
Let $X\in\mathcal{L}$ be a K3 surface with N\'eron--Severi group isometric to $U\oplus N$.  Then $\gamma$ and $\nu$ act on $(T_X)_{\Q}$. 

Let $X_{2j}$ be a certain specialization of $X$ inside the family $\mathcal{L}$ such that $\rho(X_{2j})=2j\geq 10$. To be consistent with the previous notation we denote $X$ the surface $X_{10}$ and we observe that $X_{2j}$ is a non trivial specialization of $X(=X_{10})$ if  $2j>10$. 
Then $X_{2j}$ admits a van Geemen--Sarti involution which is a specialization of the one on $X$. If $2j>10$, the specialization is obtained by considering a subspace $\Theta_{2j}\subset T_X$ with $rk(\Theta_{2j})=2j-10$ and assuming that $T_{X_{2j}}$ is the orthogonal complement to $\Theta_{2j}$ in $T_X$, while $NS(X_{2j})$ is an overlattice of finite index (possibly $1$) of $NS(X)\oplus \Theta_{2j}$. 

Since the specializations are considered in $\mathcal{L}$ and the involution is obtained specializing the involution $\sigma$ on $X$, $\gamma$ and $\nu$ preserve $(T_{X_{2j}})_{\Q}$, i.e. $\gamma(T_{X_{2j}})=T_{X_{2j}}$ and $\nu(T_{X_{2j}})=T_{X_{2j}}$ (where $T_{X_{2j}}$ is considered as subspace of $T_X$). This equivalently implies that $\gamma(\Theta_{2j})=\Theta_{2j}$ and $\nu(\Theta_{2j})=\Theta_{2j}$.

To completely describe the action of $\gamma$ and $\nu$ on $T_X$ we use the following strategy: we consider specializations $X_{2j}$ of $X$ in $\mathcal{L}$ for which the action of the quotient by van Geemen--Sarti involution can be geometrically described (as in Example \ref{rem: first specialization and action of complex multiplication}). Therefore, we are able to describe the action induced by $\gamma$ on $\Theta_{2j}$ by geometric consideration (Section \ref{subsec: the action of gamma}) whereas $\nu$ is determined by properties of the bilinear forms on it (Section \ref{subsec: the action of nu}). This allows to describe abstractly the action of $\gamma$ and $\nu$ on $\Theta_{2j}\subset T_X$. 
Hence, we look for specializations of $X$ to K3 surfaces $X_{2j}$ such that: $X_{2j}\simeq Y_{2j}\in\mathcal{L}$; $2j$ is as big as possible; we are able to describe the action of $\gamma$ (resp. $\nu$) on the subspace $\Theta_{2j}\subset T_X$. The assumption that $2j$ is as big as possible is equivalent to the condition that the Picard number $2j$ of the K3 surface $X_{2j}$ is 20, i.e. $j=10$. So, we want to specialize $X$ to a surface $X_{20}$. 

Since $\rho(X_{20})=20$, the transcendental lattice $T_{X_{20}}$ has rank 2, and it is positive definite. In this way we determine completely the action of $\gamma$ on the negative part of $(T_{X})_{\Q}$. 
Moreover, $\gamma$ preserves $T_{X_{20}}$, so the positive definite subspace $T_{X_{20}}\subset T_{X}$ is preserved by $\gamma$.

Under the previous assumptions, we are also able to describe the action of $\nu$ on the full transcendental lattice, by apply a technical result on bilinear forms, see Lemmas \ref{lemma: CM} and \ref{lemma CM2}.

In Proposition \ref{prop: isom quotients rank 12} and Example \ref{rem: first specialization and action of complex multiplication} we showed that one can glue two fibers of type $I_2$ and two fibers of type $I_1$ obtaining a K3 surfaces $X_{12}$ in a subfamily of $\mathcal{L}$. In particular $X_{12}$ is such that $X_{12}\simeq Y_{12}$ and $(NS(X_{12}))_{\Q}\simeq NS(X)\oplus \left(\langle -4\rangle\oplus \langle -2\rangle\right)_{\Q}$. We want to specialize $X$ as much as possible, by iterating the previous gluing of fibers.

\subsubsection{Construction of the basis of $(T_X)_{\mathbb{Q}}$}\label{subsubsec: basis}

In Proposition \ref{prop: isom quotients rank 12} we showed how to specializes $X\in\mathcal{L}$ in order to a construct a subfamily of $\mathcal{L}$, whose generic element has Picard number 12: it suffices to require that the discriminant of the elliptic fibration on $X$ given in \eqref{eq: Weierstrass siom quotient} acquires multiple roots. By \eqref{eq: Weierstrass siom quotient}, the discriminant locus of the fibration $\mathcal{E}$ on $X$ is
$$\frac{1}{2}\left(\alpha^2(t^2)+2t\beta(t^2)\right)^2\left(\alpha^2(t^2)-2t\beta(t^2)\right)$$ and  
we saw in proof of Proposition \ref{prop: isom quotients rank 12} that imposing a double root either to $\alpha^2(t^2)+2t\beta(t^2)$ or to $\alpha^2(t^2)-2t\beta(t^2)$, is equivalent to require that two fibers of type $I_2$ glue to a fiber of type $I_4$ and two fibers of type $I_1$ glue a fiber of type $I_2$.    
By Remark \ref{rem: gluing 2 I2} and Example \ref{ex: I47I2}, to require that two fibers of type $I_2$ glues to one of type $I_4$ and two fibers of type $I_1$ glues to one of type $I_2$, is equivalent to add to $NS(X)$ two classes, which span $\langle-4\rangle\oplus \langle -2\rangle$, i.e with the notation on the previous subsection, to chose $\Theta_{12}$ to be $\langle-4\rangle\oplus \langle -2\rangle$. Iterating this process one obtains the following.
\begin{itemize}
	\item If $\left(\alpha^2(t^2)+2t\beta(t^2)\right)$ has one double root, then the associated K3 surface $X_{12}$ has Picard number $2j=12$, the singular fibers of the fibration are $I_4+7I_2+6I_1$ and $\Theta_{12}=\langle -4\rangle\oplus \langle -2\rangle$;
	\item If $\left(\alpha^2(t^2)+2t\beta(t^2)\right)$ has two double roots, then the associated K3 surface $X_{14}$ has Picard number $2j=14$, the singular fibers of the fibration are $2I_4+6I_2+4I_1$  and $\Theta_{14}=\left(\langle -4\rangle\oplus \langle -2\rangle\right)^{\oplus 2}$;
	\item If $\left(\alpha^2(t^2)+2t\beta(t^2)\right)$ has three double roots, then the associated K3 surface $X_{16}$ has Picard number $2j=16$, the singular fibers of the fibration are $3I_4+5I_2+2I_1$  and $\Theta_{16}=\left(\langle -4\rangle\oplus \langle -2\rangle\right)^{\oplus 3}$;
	\item If $\left(\alpha^2(t^2)+2t\beta(t^2)\right)$ has four double roots, then the associated K3 surface $X_{18}$ has Picard number $2j=18$, the singular fibers of the fibration are $4I_4+4I_2$  and $\Theta_{18}=\left(\langle -4\rangle\oplus \langle -2\rangle\right)^{\oplus 4}$;
	\item If $\left(\alpha^2(t^2)+2t\beta(t^2)\right)$ has two double roots and a root with multiplicity 4, then the associated K3 surface $X_{20}$ has Picard number $2j=20$ and the singular fibers of the fibration are $I_8+3I_4+2I_2$. The fiber of type $I_8$ is obtained by gluing two fibers of type $I_4$ and one of the fiber of type $I_4$ is obtained by gluing two fibers of type $I_2$. By Lemma \ref{lemma: gluing An-1}, this implies that $\Theta_{20}=\left(\langle -4\rangle\oplus \langle -2\rangle\right)^{\oplus 4}\oplus \langle -8\rangle\oplus \langle -4\rangle$.
\end{itemize} 
The dimension of the subfamilies of $\mathcal{L}$ corresponding to the previous specializations decreases by one at each multiple root, indeed it is $5-j$, i.e. 4,3,2,1,0 respectively. Notice that $\Theta_{12}\subsetneq \Theta_{14} \subsetneq \Theta_{16}\subsetneq \Theta_{18}\subsetneq \Theta_{20}$.

After all the previous specializations, one obtains a rigid K3 surface $X_{20}$ with an elliptic fibration $\mathcal{E}_{20}$ with singular fibers $I_8+3I_4+2I_2$. It has Mordell--Weil group $MW(\mathcal{E}_{20})=\Z/4\Z\times \Z/2\Z$ and  transcendental lattice  $T_{X_{20}}\simeq \langle 8\rangle\oplus \langle 4\rangle$, see e.g. \cite{SZ}. By construction $(T_X)_{\Q}\simeq(\Theta_{20}\oplus T_{X_{20}})_{\Q}\simeq (\Gamma)_{\Q}$.
\subsubsection{The action of $\gamma$}\label{subsec: the action of gamma}

Let $X\in\mathcal{L}$ be a K3 surface with N\'eron--Severi group isometric to $U\oplus N$. Then $\gamma$ acts on $(T_X)_{\Q}\simeq (\Gamma)_{\Q}$. 

Let $X_{2j}$ be the specializations considered in subsection \ref{subsubsec: basis}. By construction $X_{2j}$ admits a van Geemen--Sarti involution which is a specialization of the one on $X$ and we saw in subsection \ref{subsubsec: basis} that  $\gamma(\Theta_{2j})=\Theta_{2j}$ and $\Theta_{2j}\subset NS(X_{2j})$.
Therefore, we are able to describe the action induced by $\gamma$ on $\Theta_{2j}$ by geometric considerations. 

In the first specialization, the Picard number is 12, the lattice $\Theta_{12}\simeq \langle -4\rangle\oplus \langle -2\rangle$ and the singular fibers are $7I_2+6I_1$. The class $\langle -4\rangle$ is due to the gluing of two fibers of type $I_2$ and the class $\langle -2\rangle$ is due to the gluing of two fibers of type $I_1$. The quotient by $\sigma$ (i.e. the action of $\gamma$) switches these two ``new" fibers as explained in Example \ref{ex: I47I2}, and so the map $\gamma$ acts as $\left[\begin{array}{cc}0&1\\2&0\end{array}\right]$ on $\Theta_{12}\simeq \langle -4\rangle\oplus \langle -2\rangle$. Iterating three more times the same process, one obtains $X_{18}\in\mathcal{L}$, whose Picard number is 18, for which the lattice $\Theta_{18}$ is $\left(\langle -4\rangle\oplus \langle -2\rangle\right)^{\oplus 4}$ and the singular fibers are $4I_4+4I_2$. The action of $\gamma$ is $\left[\begin{array}{cc}0&1\\2&0\end{array}\right]^{\oplus 4}$. The last step is similar, but one glues two fibers of type $I_4$ to one of type $I_8$ and two fibers of type $I_2$ to one of type $I_4$, which corresponds to introducing two vectors with self intersection $-8$ and $-4$ respectively. Reasoning as in the previous case one shows that also in this case $\gamma$ acts as $\left[\begin{array}{cc}0&1\\2&0\end{array}\right]$. So we proved that $\gamma$ acts as $\left[\begin{array}{cc}0&1\\2&0\end{array}\right]^{\oplus 5}$ on the negative part of $(T_X)_{\Q}$, i.e. of the lattice $\Theta_{20}=\left(\langle -4\rangle\oplus \langle -2\rangle\right)^{\oplus 4}\oplus\langle -8\rangle\oplus \langle -4\rangle$.

The proof is based on the action of $\gamma$ on $\Theta_{20}$ in the case $\Theta_{20}$ is completely contained in the N\'eron--Severi group. But the action of $\gamma$ on $\Theta_{20}$ is purely lattice theoretic and does not depend on the rank of the intersection of $\Theta_{20}$ with the N\'eron--Severi group or with the transcendental lattice. So we described the action of $\gamma$ even in the case $\Theta_{20}$ is completely contained in the transcendental lattice, i.e. in the generic case in which $\rho(X)=10$. For this reason the result remains true for every $X\in\mathcal{L}$ such that  $U\oplus N$ is embedded in $NS(X)$ coherently with Table  \ref{eq: U+Nin Lambda}, by observing that $\left(\Theta_{20}\oplus \langle 8\rangle\oplus \langle 4\rangle\right)_{\Q}$ is orthogonal to $\varphi({U\oplus N})$ in $\Lambda_{K3}$.

\subsubsection{The action of $\nu$}\label{subsec: the action of nu}
We are now interested in determining the action of the complex multiplication $\nu$. This action will be determined step by step, considering before the action on $T_{X_{20}}$, when the rank of the transcendental lattice is as small as possible, and then extending the results when the rank of the transcendental lattice increases.  First we consider easy results on the action of certain endomorphism on 2-dimensional $\mathbb{Q}$-vector space, that will be useful in the following.
	\begin{lemma}\label{lemma: CM} Let $V_x^k$ be a 2-dimensional $\mathbb{Q}$-vector space endowed with a bilinear form $q$. Let $\{v_1,v_2\}$ be a basis of $V_x^k$ such that $q$ is represented by $\left[\begin{array}{cc}kx&0\\0&x\end{array}\right]$ with $x\in \mathbb{Z}$, $x\neq 0$ and $k\in\mathbb{N}_{>0}$. Let $A$ be an endomorphism of $V_x$ such that: \begin{itemize}\item $A^2=-kId$\item the isotropic vectors of $V_x\otimes \mathbb{C}$ are eigenvectors for the $\mathbb{C}$ linear extension of $A$.\end{itemize}
		Then the eigenvalues of $A$ are $\pm\sqrt{-k}$ and $A$ is represented with respect to the basis $\{v_1,v_2\}$ by $\pm M_k$ where $M_k=\left[\begin{array}{rr}0&1\\-k&0\end{array}\right]$.\end{lemma}
	\proof The endomorphism $A$ is represented by a matrix with $\Q$ coefficients and so $A^2=-2Id$ implies that $Tr(A)=0$ and $\det(A)=k$, i.e. the matrix which represents $A$ is of the form $\left[\begin{array}{rr}a&b\\c&-a\end{array}\right]$ with $a^2+bc=-k$. The eigenvalues are hence forced to be $\pm\sqrt{-k}$. The isotropic vectors of $V_x^k\otimes \mathbb{C}$ for the bilinear form $q$ are $\alpha(v_1\pm\sqrt{-k}v_2)$, $\alpha\in \mathbb{C}-\{0\}$. Denoted $w_1=v_1+\sqrt{-k}v_2$, $w_2=v_1-\sqrt{-k}v_2$, one obtains $v_1=(w_1+w_2)/2$ and $v_2=\sqrt{-k}(w_2-w_1)/2k$ and there are two possibilities either  $A(w_1)=\sqrt{-k}w_1$ and $A(w_2)=-\sqrt{-k}w_2$ or $A(w_1)=-\sqrt{-k}w_1$ and $A(w_2)=\sqrt{-k}w_2$. In the first case one obtains $A(v_1)=-kv_2$ and $A(v_2)=v_1$, in the latter $A(v_1)=kv_2$ and $A(v_2)=-v_1$.  \endproof
	
	\begin{lemma}\label{lemma CM2}
		Let $W$ be a $\Q$-vector space with a bilinear form $b$, $A$ an endomorphism of $W$. Let $b_{\mathbb{C}}$ and $A_{\mathbb{C}}$ be the $\mathbb{C}$-linear extension of $A$ and $b$ to $W\otimes \mathbb{C}$. If  $A^2=-kId$ with $k\in\mathbb{N}_{>0}$ and  $b_{\mathbb{C}}(A_{\mathbb{C}}v,w)=b_{\mathbb{C}}(v,\overline{A_{\mathbb{C}}}w)$ for every $v,w\in W\otimes \mathbb{C}$, then the eigenvectors for $A_{\mathbb{C}}$ are isotropic vectors with respect to $b_{\mathbb{C}}$. 
	\end{lemma}
	\proof Since $A^2=-k Id$ the eigenvalues of $A_{\mathbb{C}}$ are $\pm\sqrt{-k}$. Let $v\in W\otimes \mathbb{C}$ be an eigenvector of the eigenvalue $\pm \sqrt{-k}$, then  $\overline{A_{\mathbb{C}}}A_{\mathbb{C}}(v)=kv$ and so $$b_{\mathbb{C}}(A_{\mathbb{C}}v,A_{\mathbb{C}})=b_{\mathbb{C}}(v,\overline{A_{\mathbb{C}}}A_{\mathbb{C}})=b_{\mathbb{C}}(v,kv)=kb_{\mathbb{C}}(v,v).$$ On the other hand $$b_{\mathbb{C}}(A_{\mathbb{C}},A_{\mathbb{C}})=b_{\mathbb{C}}(\pm\sqrt{-k}v,\pm\sqrt{-k}v)=-kb_{\mathbb{C}}(v,v).$$
	This implies $b_{\mathbb{C}}(v,v)=0$, since $k\neq 0$. \endproof
	\begin{rem}\label{rem: the lemma CM2 applies to CM}{\rm Let us assume that $W=T_X$ is the transcendental lattice of a K3 surface with complex multiplication and $A\in End(T_X,\Q)$ an endomorphism of its Hodge structure. Then $A$ satisfies the hypothesis of the previous lemma, see e.g. \cite[Section 2.1]{vGSc}.}\end{rem}
	
	{\it Complex multiplication on $X_{20}\in\mathcal{L}$.} By Subsection \ref{subsubsec: basis} there exists a K3 surface $X_{20}\in\mathcal{L}$ whose transcendental lattice is $\langle 8\rangle\oplus \langle 4\rangle$. So $(T_{20})_{\Q}$ is a 2-dimensional vector space endowed with a bilinear form $q$ represented by $\left[\begin{array}{cc}kx&0\\0&x\end{array}\right]$  with $k=2$ and $x=4$. Since $X_{20}\in\mathcal{L}$, $X_{20}$ admits complex multiplication by $\sqrt{-2}$. The complex multiplication acts on the period $\omega_{20}$ of $X_{20}$ as the multiplication by $\pm {\sqrt{-2}}$ and on its conjugate $\overline{\omega_{20}}$ as the multiplication by $\mp {\sqrt{-2}}$. We can assume that $\omega_{20}$ is the eigenvectors relative the eigenvalues $\sqrt{-2}$. Observe that $\{\omega_{20},\overline{\omega_{20}}\}$ is a basis of $T_{20}\otimes \mathbb{C}$, given by isotropic vectors. So the complex multiplication acts as the endomorphism $A$ in Lemma \ref{lemma: CM}, once one has identified $V_4^2$ with $(T_{20})_{\Q}$. Therefore $A$ (and so the complex multiplication $\nu$) is represented by $\pm M_2$ on the basis of $T_{X_{20}}$. 
	
	{\it Complex multiplication on $X_{18}\in\mathcal{L}$.}  By Subsection \ref{subsubsec: basis}, there exists a one dimensional family of K3 surfaces $X_{18}\in\mathcal{L}$ whose transcendental lattice $T_{18}$ has a basis over $\mathbb{Q}$ on which the bilinear form is $\langle 8\rangle\oplus \langle 4\rangle\oplus \langle -8\rangle\oplus \langle -4\rangle\simeq T_{20}\oplus \langle -8\rangle\oplus \langle -4\rangle $. So $(T_{18})_{\Q}$ is the direct sum  $\left((T_{20})_{\Q}\right)\oplus P$ where both $(T_{20})_{\Q}$ and $P$ are a 2-dimensional $\mathbb{Q}$-vector space and $P$ is endowed with a bilinear form represented on a chosen basis by the matrix $\left[\begin{array}{cc}kx&0\\0&x\end{array}\right]$ with $k=2$ and $x=-4$, i.e. $P$ can be identified with a copy of $V_{-4}^2$ with the notation of Lemma \ref{lemma: CM}. 
	
	The K3 surface $X_{18}\in\mathcal{L}$ admits complex multiplication which specializes to the one described before on $X_{20}$. In particular, choosing a basis of $(T_{18})_{\Q}$ such that the bilinear form is $(T_{20})_{\Q}\oplus P\simeq V_4^2\oplus V_{-4}^2$, the complex multiplication is represented, with respect to this basis, by a $4\times 4$ block matrix $N$ of the form $N:=\left[\begin{array}{c|c}M&0\\ \hline 0&L\end{array}\right]$, where $M$ is the matrix determined above (which describes the complex multiplication on $X_{20}$) and $L$ is a $2\times 2$ matrix which will be determined now.  
	
	Observe that $N^2=-2Id$, since it represents the complex multiplication by $\sqrt{-2}$ and hence $L^2=-2Id$.
	
	The space $\left(T_{20}\oplus P\right)\otimes \C$ can be splitted into the sum of its eigenspaces $$\left(\left(T_{20}\oplus P\right)\otimes \C\right)_{\sqrt{-2}}\oplus \left(\left(T_{20}\oplus P\right)\otimes \C\right)_{-\sqrt{-2}},$$
	where $\left(\left(T_{20}\oplus P\right)\otimes \C\right)_{\pm\sqrt{-2}}=\left(\left(T_{20}\right)\otimes \C\right)_{\pm\sqrt{-2}}\oplus \left(P\otimes \C\right)_{\pm\sqrt{-2}}$.
	The period $\omega_{18}$ is neither totally contained in $\left(T_{20}\otimes \C\right)_{\sqrt{-2}}$ nor totally contained in $\left(P\otimes \C\right)_{\sqrt{-2}}$. Indeed,  if it were  contained in $\left(T_{20}\otimes \C\right)_{\sqrt{-2}}$, $\left(P\otimes \C\right)_{\sqrt{-2}}$ would be contained in $H^{1,1}$. So there exists $\alpha_1\in \left(T_{20}\otimes \C\right)_{\sqrt{-2}}$ and $\alpha_2\in (P\otimes \C)_{\sqrt{-2}}$ such that $\omega_{18}=\alpha_1+\alpha_2$ and $N$ restricts to $L$ on $P\otimes \C$. By Remark \ref{rem: the lemma CM2 applies to CM}, we can apply Lemma \ref{lemma CM2} to $\alpha_2$, which is an eigenvalue of $L$. Therefore $\alpha_2$ is isotropic and we can apply Lemma \ref{lemma: CM} to the endomorphism of $P$ represented by the matrix $L$, proving that $L=\left[\begin{array}{rr}0&1\\-2&0\end{array}\right]$.

	{\it Complex multiplication on $X_{2j}\in\mathcal{L}$, $2j\geq 10$.} We iterate the previous procedure: by the subsection \ref{subsubsec: basis} there exists $(5-j)$-dimensional family of K3 surfaces  $X_{2j}\in\mathcal{L}$ on which one can identify $\left(T_{2(j-1)}\right)_{\Q}$ with $\left(T_{2j}\right)_{\Q}\oplus V_{x}^2$ for a certain $x$ (which is either $-4$ or $-2$ depending on $j$). The complex multiplication on $\left(T_{2(j-1)}\right)_{\Q}$ specializes to $\left(T_{2j}\right)_{\Q}$ and in particular the action of the complex multiplication on  $\left(T_{2(j-1)}\right)_{\Q}$ is represented by a block matrix whose first block coincides with the matrix describing the complex multiplication on $\left(T_{2j}\right)_{\Q}$ and the last one represents an endomorphism which satisfy the assumption on $A$ in Lemma \ref{lemma: CM}. Hence, the complex multiplication on $X_{2j}\in\mathcal{L}$ acts on $\left(T_{2j}\right)_{\Q}$ as $\left[\begin{array}{rr}0&1\\-2&0\end{array}\right]^{\oplus j-5}$.

\subsection{Proof of Corollary \ref{cor: NSX=NSY condition}}\label{subsect: proof of the corollary assuming the theorem}
By Theorem \ref{theorem: action selfmap}, if $X\simeq Y$ the map $\pi_*:(T_X)_{\Q}\simeq (\Gamma)_{\Q}\ra (T_Y)_{\Q}\simeq (\Gamma)_{\Q}$ is $\gamma:(\Gamma)_{\Q}\ra (\Gamma)_{\Q}$. For every $X\in\mathcal{N}$ such that $\rho(X)=10$, $T_X\simeq T_Y$ (even if $X\not\simeq Y$) and then $(T_X)_{\Q}\simeq(\Gamma)_{\Q} \simeq (T_Y)_{\Q}$. So, identifying both $(T_X)_{\Q}$ and $(T_Y)_{\Q}$ with $(\Gamma)_{\Q}$, one obtains that $\pi_*$ coincides with $\gamma$. 
Let us consider a specialization of $X$. As in the proof of the previous theorem, the specialization $Z_\Theta$ of $X$ depends on a negative definite sublattice $\Theta$ of $T_X$ such that the orthogonal complement of $\Theta$ in $T_X$ is isometric to $T_{Z_{\Theta}}$ and the N\'eron--Severi of $Z_{\Theta}$ is a finite index overlattice of $NS(X)\oplus \Theta$. Notice that $\Theta$ is negative definite and its rank is at most 10 and that $T_{Z_{\Theta}}=\Theta^{\perp_{T_X}}$. Observe that $T_{Z_{\Theta}}\subset T_{X}$ and hence $\gamma$ maps  $T_{Z_{\Theta}}$ to a subset of $T_X$.

Now suppose that $\gamma(T_{Z_{\Theta}})=T_{Z_{\Theta}}$. Since $\gamma(T_X)=T_X$, this is equivalent to require that $\gamma(\Theta)=\Theta$. Since $\gamma$ acts on the negative part of an overlattice of $T_X$ as a matrix with five blocks of the form $\left[\begin{array}{rr}0&1\\2&0\end{array}\right]$, if $\gamma(\Theta)=\Theta$, then the rank of $\Theta$ is even, so that we can assume it has rank $2b-10$, $b=6,\ldots, 10$, which implies that $\rho(Z_\Theta)=2b$ (and $Z_\Theta$ can be identified with $Z_{2b}$ in the statement).
 
To compute the transcendental lattice of $\widetilde{Z_\Theta/\sigma_Z}$ (desingularization of the quotient of $X_\Theta$ for the van Geemen--Sarti involution induced by the one of $X$), we consider $\pi_*(T_{Z_\Theta})$, which is equivalent to apply the linear extension of $\gamma$ to $\Gamma$, so that  
$T_{Z_{\Theta}}=\gamma(T_{Z_{\Theta}})=T_{\widetilde{Z_\Theta/\sigma_Z}}$. The lattice $T_{Z_{\Theta}}$ has rank $12-\rk(\Theta)\geq 10$ and so it admits a unique primitive embedding in  $\Lambda_{K3}$. Moreover, its orthogonal is uniquely determined by the discriminant form of $T_{Z_\Theta}$. It follows that $NS(Z_{\Theta})\simeq NS(\widetilde{Z_\Theta/\sigma})$.\endproof

\subsection{Specializations of $X$ which glue fibers and preserve the van Geemen--Sarti involution}
It remains to construct lattice theoretically the described specializations, showing that they exist even if $X\in\mathcal{N}\setminus\mathcal{L}$ and that also in this case they can be obtained in a way which is compatible with the action of the van Geemen--Sarti involution. To do that we first consider a lattice theoretic lemma.
\begin{lemma}\label{lemma: gluing An-1}
The lattice $A_{2n-1}$ is an overlattice of index $n$ of $A_{n-1}\oplus A_{n-1}\oplus \langle -2n\rangle$.\end{lemma}
\proof Let us denote $a_i^{(j)}$ a basis of the $j$-th copy of $A_{n-1}$ such that $a_i^{(j)}a_h^{(j)}=1$ if and only if $|i-h|=1$. Then $w^{(j)}:=\left(\sum_{i=1}^{n-1}ia_i^{(j)}\right)/n$ is a generator of the discriminant group $A_{A_{n-1}}$ and $\left(w^{(j)}\right)^2=-(n-1)/n$. Let $v$ be the generator of the orthogonal complement of $A_{n-1}\oplus A_{n-1}$ in $A_{n-1}\oplus A_{n-1}\oplus \langle -2n\rangle$. The subspace $\langle w^{(1)}+w^{(2)}+v/n\rangle$ is an isotropic subspace of the discriminant group of $A_{n-1}\oplus A_{n-1}\oplus \langle -2n\rangle$ and hence it corresponds to an overlattice of index $n$ of $A_{n-1}\oplus A_{n-1}\oplus \langle -2n\rangle$, which in particular is a lattice with discriminant $-2n$. The generators of this overlattice are all the generators of $A_{n-1}\oplus A_{n-1}\oplus \langle -2n\rangle$ and $\varepsilon:=w^{(1)}+w^{(2)}+v/n$. Since $\varepsilon^2=-2$, one obtains a ``new" root in the overlattice and in particular the overlattice is a root lattice with discriminant group $\Z/2n\Z$ and discriminant form $\left[\frac{-(2n-1)}{2n}\right]$, hence it is $A_{2n-1}$.\endproof

\begin{rem}\label{rem: gluing An-1 opposite construction}{\rm One can reverse the previous construction: let $b_i$ be a basis of $A_{2n-1}$ such that $b_ib_j=1$ if $|i-j|=1$; then $v=\sum_{i=1}^nb_{2i-1}$ is such that $v^2=-2n$ and the orthogonal complement of $v$ is isometric to $A_{n-1}\oplus A_{n-1}$. The basis of the first copy of $A_{n-1}$ is $a_i^{(1)}=b_{2i-1}+b_{2i}$, $i=1,\ldots, n-1$, the one of the second is $a_i^{(2)}=b_{2i}+b_{2i+1}$, $i=1,\ldots, n-1$. Of course reversing (over $\mathbb{Q}$) the change of basis one can obtain the basis $b_i$ in terms of $a_i^{(j)}$ and $v$.}\end{rem}

\begin{rem}{\rm  If one compares the construction in the proof Lemma \ref{lemma: gluing An-1} for $n=2$ with the one of the Remark \ref{rem: gluing 2 I2}, one realizes that the geometric gluing of two $I_2$-fibers corresponds to construct $A_3$ as overlattice of $A_1\oplus A_1\oplus \langle -4\rangle$.}\end{rem}

As observed in the previous Remark the lattice theoretic result in Lemma \ref{lemma: gluing An-1} corresponds to specializations of elliptic fibrations obtained by gluing singular fibers. For our purpose we require that the gluing of the fibers is coherent with the action of the van Geemen--Sarti involution and in Remark \ref{rem: gluing 2 I2} we already observed that the gluing of two fibers of type $I_2$ is compatible with the van Geemen--Sarti involution, because the formal equation of the order 2 section remains the same in terms of the generators of the Nikulin lattice.

In the next example we discuss the gluing of two fibers of type $I_4$ with non trivial intersection with a 2-torsion section, since this is the case of interest in the following.
\begin{example}{\rm  Let us consider an elliptic fibration with two fibers of type $I_4$ and let us denote $C_i^{j}$, $i=1,2,3$, $j=1,2$ the non trivial components on the $j$-th fiber of type $I_4$ (with the assumption that $C_i^{(j)}C_h^{(j)}=1$ if and only if $|i-h|=1$). Let $T$ be a 2-torsion section meeting the component $C_2^{(j)}$ of each of the $I_4$-fibers. Then the equation of $T$ is $2F+S-\frac{1}{2}\left(\sum_{j=1}^2\left(C_1^{(j)}+2C_2^{(j)}+C_3^{(j)}\right)+Z\right)$ where $Z$ is a linear integer combination of components of other reducible fibers. The non trivial components of the two $I_4$-fibers generate $A_3\oplus A_3$. In order to glue the $I_4$-fibers to an $I_8$-fiber, one adds a class $W$ with self intersection $-8$. Accordingly to Remark \ref{rem: gluing An-1 opposite construction}, the non trivial components $D_i$ of the $I_8$ fiber can be chosen to be 
\begin{eqnarray}\label{eq: I8 components D and 2I4}
\begin{array}{l}D_1=(3C_1^{(1)}+2C_2^{(1)}+C_3^{(1)}-3C_1^{(2)}-2C_2^{(2)}-C_3^{(2)}+W)/4\\
D_2=(C_1^{(1)}-2C_2^{(1)}-C_3^{(1)}+3C_1^{(2)}+2C_2^{(2)}+C_3^{(2)}-W)/4\\
D_3=(-C_1^{(1)}+2C_2^{(1)}+C_3^{(1)}+C_1^{(2)}-2C_2^{(2)}-C_3^{(2)}+W)/4\\
D_4=(C_1^{(1)}+2C_2^{(1)}-3C_3^{(1)}-C_1^{(2)}+2C_2^{(2)}+C_3^{(2)}-W)/4\\
D_5=(-C_1^{(1)}-2C_2^{(1)}+C_3^{(1)}+C_1^{(2)}+2C_2^{(2)}-C_3^{(2)}+W)/4\\
D_6=(C_1^{(1)}+2C_2^{(1)}+3C_3^{(1)}-C_1^{(2)}-2C_2^{(2)}+C_3^{(2)}-W)/4\\
D_7=(-C_1^{(1)}-2C_2^{(1)}-3C_3^{(1)}+C_1^{(2)}+2C_2^{(2)}+3C_3^{(2)}+W)/4.\end{array}\end{eqnarray}

The formal equation of the 2-torsion section meeting non trivially a fiber of type $I_8$ is $T=2F+S-\frac{1}{2}(D_1+2D_2+3D_3+4D_4+3D_5+2D_6+D_7+Z)$. Substituting in this equation the expression of the $D_i$'s in \eqref{eq: I8 components D and 2I4} one obtains the same expression as before, i.e. $2F+S-\frac{1}{2}\left(\sum_{j=1}^2\left(C_1^{(j)}+2C_2^{(j)}+C_3^{(j)}\right)+Z\right)$ (assuming that the specialization modifies only the two $I_4$-fibers and not the curves in $Z$). This guarantees that the action of the van Geemen--Sarti involution on the specialized surface is cohomologically equal to the one of the original surface.}\end{example}
\subsection{The five specializations: proof of Proposition \ref{prop: the five specializations}}\label{subsec: the five spec}
To prove Proposition \ref{prop: the five specializations} one explicitly exhibits the five specializations, giving the classes of the transcendental lattice of $X$ which becomes algebraic at each step.

As already said, in the first four steps one adds a class $W_i$ of self intersection $-4$ and a class $V_i$ of self intersection $-2$ to the N\'eron--Severi group of the previous surface. The class $W_i$ can be glued with other classes in the N\'eron--everi group and this allows to construct an $I_4$-fiber starting from two $I_2$-fibers. 
In the fourth specialization a new phenomenon appears: the torsion part of the Mordell--Weil group changes since the elliptic fibration acquires a 4-torsion section (whose square is the original 2-torsion section associated to the van Geemen--Sarti involution) and an independent 2-torsion section.
The last specialization consists in adding a class $W_5$ with self intersection $-8$ and a class $V_5$ with self intersection $-4$. Both these classes have gluing relation with their orthogonal complement in the N\'eron--Severi group, and this allows to construct a fiber of type $I_8$ (by gluing two $I_4$'s) and one of type $I_4$ (by gluing two $I_2$'s). To conclude the proof it suffices to give explicitly the classes $V_i$ and $W_i$. A schematic synthesis of the construction is shown in Figure \ref{Figure speicalization}. 

\begin{figure}%[ht!] %!t
\includegraphics{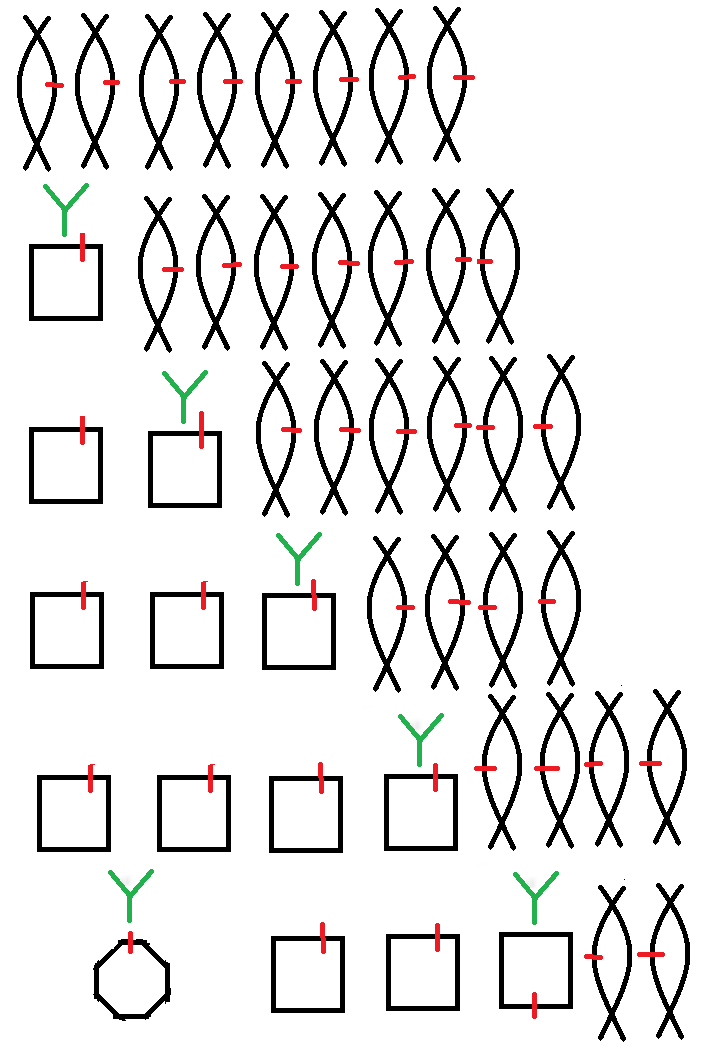}
\caption{Specializations: the 2-torsion section $T$ is red, and we assume that the trivial component of the $I_4$ and $I_8$ fibers is the horizontal lower component and the trivial component of the $I_2$-fibers is the left component.}
\label{Figure speicalization}
\end{figure}
The notation will be the following: the classes of irreducible $I_2$-fibers will be denoted $B_i^{(j)}$, $i=1,0$, and $j$ distinguishes among the different fibers of the same type.
Similarly, the classes of irreducible components of $I_4$-fibers will be denoted $C_i^{(j)}$, $i=0,1,2,3$ and the ones of $I_8$-fibers will be denoted $D_i^{(j)}$, $i=0,\ldots, 7$.

On the original surface $X$, we have $B_1^{(j)}=N_j$, where $N_j$ are the classes of the Nikulin lattice, and $B_0^{(j)}=F-N_j$.

The torsions section $T$ is 
$$T=2F+S-\hat{N}=F+2S-\frac{1}{2}\left(\sum_{i=1}^8 N_i\right).$$

The transcendental lattice is isometric to $N\oplus U\oplus U$ and generated by $t_i$, $i=1,\ldots ,12$ with the notation used in Section \ref{subsec: cohomological action}. 
\subsubsection{The first specialization: $X_{12}$}
We add to $NS(X)$ the following two classes:
$$V_1:=-t_1-t_2-t_3-t_4-t_5-t_6-t_7+2t_8-t_9\mbox{ and }W_1:=-2t_1-2t_3-2t_4-t_5-t_6+2t_8-2t_9-2t_{10}.$$
The class $(W_1+N_5+N_6)/2$ is contained in $H^2(X,\Z)$ and a multiple of it is contained in $NS(X_{12})$, then $(W_1+N_5+N_6)/2\in NS(X_{12})$. The transcendental lattice $T_{X_{12}}$ is $\langle V_1, W_1\rangle^{\perp_{T_{X}}}$ and it is generated by $\langle t_1-t_3, t_2+t_3,t_3-t_4,t_4+t_7,t_5, t_6, t_7-t_9, -t_8-t_9+t_{10}, t_{11}, t_{12}\rangle$. 

Now the singular fibers are  $7I_2+I_4+6I_1$ and the components of the reducible ones are:
$$B_1^{(j)}=N_{j},\mbox{ for }j=1,\ldots,4,\ B_1^{(j)}=N_{j+2},\mbox{ for } j=5,6,\ B_1^{(7)}=V_1,$$ 
$$C_1^{(1)}=(-W_1+N_5-N_6)/2,\ C_2^{(1)}=(W_1+N_5+N_6)/2,\ C_3^{(1)}=(-W_1-N_5+N_6)/2.$$ This implies that the classes $N_5$ and $N_6$ splits, indeed they are: $N_5=C_1^{(1)}+C_2^{(1)}$, $N_6=C_2^{(1)}+C_3^{(1)}$.
The 2-torsion section is trivial on the seventh fiber of type $I_2$ (i.e. on $B_1^{(7)}=V_1$), and non trivial on all the other reducible fibers. 

The formal equation of the class of the torsion section remains the same in terms of the classes $N_i$'s, i.e.
$$T=2F_X+S_X-\hat{N}=2F_X+S_X-\frac{1}{2}\left(\sum_{j=1}^6(B_1^{(j)})+C_1^{(1)}+2C_2^{(1)}+C_3^{(1)}\right).$$
This is the situation described in the Example \ref{ex: I47I2}.
\subsubsection{The second specialization: $X_{14}$}
We add to $NS(X_{12})$ the following two classes:
$$V_2:=t_5\mbox{ and } W_2:=t_1+t_2+t_3+t_4+t_5+t_6-2t_8+2t_{10}.$$

The specialization is similar to the previous one: $(W_2+N_7+N_8)/2\in NS(X_{14})$ and the transcendental lattice $T_{X_{14}}$ is $\langle t_1-t_3, t_2+t_3,t_3-t_4, t_6, t_7-t_9, -4t_4-t_5-4t_7+2t_8+2t_9-2t_{10}, t_{11}, t_{12}\rangle$.

\subsubsection{The third specialization: $X_{16}$} We add to $NS(X_{14})$ the following two classes:
$$V_3:=t_{11}-t_{12}\mbox{ and } W_3:=t_1-t_3.$$

The situation is similar to the previous specialization: $(W_3+N_1+N_3)/2\in NS(X_{16})$ and the transcendental lattice $T_{X_{16}}$ is 
$\langle t_1+t_3-2t_4, t_2+t_4, t_6, t_7-t_9, -4t_4-t_5-4t_7+2t_8+2t_9-2t_{10},-t_8-t_9+t_{10}, t_{11}+t_{12}\rangle$. 
\subsubsection{The fourth specialization}

We add to $NS(X_{16})$ the following two classes:
$$V_4:=-t_6+t_7-t_9-t_{11}-t_{12}\mbox{ and } W_4:=t_2+t_4.$$

Similar to the previous cases, $NS(X_{18})$ is still an overlattice of $NS_{X_{16}}\oplus \langle V_4\rangle\oplus \langle W_4\rangle$, but in this case the index of the overlattice is different: it is $2^2\cdot 4$. Indeed the following three classes are contained in $NS(X_{18})$: 
$$(W_4+N_2+N_4)/2;\ (W_3+W_4+V_1+V_2+V_3+V_4)/2;\  (-W_1+2N_6-W_2+2N_7-W_3+2N_3-W_4+2N_2+2V_3+2V_4)/4.$$ The second and the third divisible classes are due to ``new" 2-torsion and 4-torsion sections, denoted $P$ and $Q$ respectively. 
The transcendental lattice $T_{X_{18}}$ is 
$\langle t_1+t_2+t_3-t_4, t_2-3t_4-t_5-4t_7+2t_8+2t_9-2t_{10}+5t_{11}+5t_{12},t_6+t_7-t_9,t_7-t_9-t_{11}-t_{12}\rangle$.

The K3 surface $X_{18}$ admits an elliptic fibration with $4I_2+4I_4$ as singular fibers and whose Mordell--Weil group is $\Z/4\Z\times\Z/2\Z$. The components of the reducible fibers are$$B_1^{(i)}=V_i, i=1,2,3,4$$ 
$$C_1^{(1)}=(-W_1+N_5-N_6)/2,\ C_2^{(1)}=(W_1+N_5+N_6)/2,\ C_3^{(1)}=(-W_1-N_5+N_6)/2,$$ 
$$C_1^{(2)}=(-W_2+N_7-N_8)/2,\ C_2^{(2)}=(W_2+N_7+N_8)/2,\ C_3^{(2)}=(-W_2-N_7+N_8)/2,$$
$$C_1^{(3)}=(-W_3+N_1-N_3)/2,\ C_2^{(3)}=(W_3+N_1+N_3)/2,\ C_3^{(3)}=(-W_3-N_1+N_3)/2,$$
$$C_1^{(4)}=(-W_4+N_2-N_4)/2,\ C_2^{(3)}=(W_4+N_2+N_4)/2,\ C_3^{(4)}=(-W_4-N_2+N_4)/2.$$
and the expression of the classes of the ``new" torsion sections are
$$P=2F+S-\frac{1}{2}\left(\sum_{i=3}^4\left(C_1^{(i)}+2C_2^{(i)}+C_3^{(i)}\right)+\sum_{j=1}^4B_1^{(j)}\right)=2F+S-\left(W_3+W_4+V_1+V_2+V_3+V_4\right)/2;$$
$$\begin{array}{ll}Q&=2F+S-\frac{1}{4}\left(C_1^{(1)}+2C_2^{(1)}+3C_3^{(1)}+\sum_{i=2}^4\left(3C_1^{(i)}+2C_2^{(i)}+C_3^{(i)}\right)+2\sum_{j=3}^4B_1^{j}\right)=\\&=2F+S-\left(-W_1+2N_6-W_2+2N_7-W_3+2N_3-W_4+2N_2+2V_3+2V_4\right)/4.\end{array}$$

Denoted $\boxplus$ the operation in the Mordell--Weil group, $Q\boxplus Q$ is a 2-torsion section: it is the 2-torsion section specializing the unique 2-torsion section on the surface $X_{16}$ and so we denote it $T$, to be consisten with the previous notation. It is trivial on all the $I_2$-fibers, and non trivial on all $I_4$-fibers. Its formal equation remains the same in terms of the classes $N_i$'s.

\subsubsection{The fifth specialization}
We add to $NS(X_{18})$ the following two classes:
$$V_5:=t_6+t_7-t_9\mbox{ and } W_5:=t_1+t_2+t_3-t_4.$$

The self intersection of these classes is $V_5^2=-4$ and $W_5^2=-8$. Similarly to the previous cases, $NS(X_{20})$ is an overlattice of $NS_{X_{18}}\oplus \langle V_5\rangle\oplus \langle W_5\rangle$ but the index is now $2\cdot 4$, obtained adding the two classes 
$$\frac{V_3+V_4+V_5}{2}\mbox{ and }\frac{N_1+N_2}{2}-\frac{W_3+W_4+W_5}{4}.$$ The first divisible class corresponds to the gluing of two fibers of type $I_2$ to a fiber of type $I_4$, the second one to the gluing of two fibers of type $I_4$ to a fiber of type $I_8$.
 
The transcendental lattice $T_{X_{20}}$ is 
$\langle3t_1+t_2+3t_3+  3t_4+  2t_5  +2t_6 -4t_7 + 4t_8  +4t_9,
t_6,-t_7+t_9+2t_{11}+2t_{12}\rangle$.

The specializations of the 4-torsion and 2-torsion sections of $X_{18}$ give the two sections $Q$ and $P$ on $X_{20}$. The section $Q\boxplus Q$ is the 2-torsion section $T$, whose formal equation remains the same in terms of the classes $N_i$'s.
In particular, the section $T$ is non trivial on the $I_8$-fiber and on the first and second $I_4$-fibers. It is trivial on the third $I_4$-fiber and on the $I_2$-fibers.

The quotient by the translation by the 2-torsion section $T$ exchanges the $I_8$-fiber with the third $I_4$-fiber, and the first and second $I_4$-fibers with the $I_2$-fibers. The K3 surface $X_{20}$ is such that $X_{20}\simeq Y_{20}$ and, as already explained in Section \ref{subset: proof of theorem}, the quotient by the symplectic involution maps $V_i$ to $ W_i$ and $W_i$ to $2V_i$, for $i=1,2,3,4,5$.

\begin{rem}{\rm As already observed, the fifth specializations considered, can be chosen in such a way that $X_{n}\in\mathcal{L}$. But, as in Remark \ref{rem: first specialization real multiplication}, one can also consider a K3 surface $X'$ with N\'eron--Severi group $U\oplus N$ which is contained in the family $\mathcal{L}'$, described in Remark \ref{rem: real multiplication}, and specializes this K3 surface. The first specialization can be done in such a way that $X_{12}'$ is still contained in $\mathcal{L}'$, as shown in Remark \ref{rem: first specialization real multiplication}. Similarly, one can consider the second and the third specializations, obtaining (after these three specializations) a K3 surface $X_{16}'\in\mathcal{L}'$ with an elliptic fibration with singular fibers $2III+3I_4+3I_2$. This is possible since gluing together two fibers of type $I_2$ corresponds to gluing two fibers of type $I_1$ in the family $\mathcal{L}'$ (as explained in Remark \ref{rem: first specialization real multiplication}). The fourth specialization described in the previous section, cannot be applied here, since it would imply that one is gluing 2 fibers of type $I_1$ and one has no more fibers of this type on $X_{16}'$. Nevertheless, one can apply the fifth specialization, gluing two fibers of type $I_4$ to one of type $I_8$ and two fibers of type $I_2$ to one of type $I_4$. So one can perform 4 specializations, obtaining a K3 surface in $\mathcal{L}'$ such that its Picard number is 18, and the specialized elliptic fibration has $2III+I_8+2I_4+I_2$ as singular fibers.}
\end{rem}
\subsection{Back to Theorem \ref{theorem: action selfmap}}
Now we can give explicitly a basis of $T_{X}$ as the one mentioned in the theorem: according to the specializations considered above it is
$$W_1,V_1,W_2,V_2,W_3,V_3,W_4,V_4,W_5,V_5, 3t_1+t_2+3t_3+  3t_4+  2t_5  +2t_6 -4t_7 + 4t_8  +4t_9,
t_6,-t_7+t_9+2t_{11}+2t_{12}$$ where the last two classes are the generators of $T_{X_{20}}$. This gives the specialization described in Proposition \ref{prop: the five specializations}, which can be viewed as specializations both in the family $\mathcal{N}$ and in the family $\mathcal{L}$.

\section{Order 3}\label{sec: order 3}
The order 3 symplectic automorphisms are intensively studied in the last decades, see e.g. \cite{Nik}, \cite{GS1}, \cite{GP}.
In particular, as the symplectic involutions act in a standard way on $\Lambda_{K3}$ permuting two copies of $E_8$, the order 3 symplectic automorphisms permute three copies of $E_6$, see \cite[Theorem A]{GP}. If $\sigma$ is an order 3 symplectic automorphism on a K3 surface $X$, then $(NS(X)^{\sigma^*})^{\perp}$ is a lattice isometric to $K_{12}$ (it is the analogue of $E_8(2)$ in the case of the involution). By \cite[Proposition 3.1]{GP}, this lattice can be described as overlattice of index 3 of a rank 12 negative definite lattice whose intersection form is $\widetilde{K_{12}}=\left[\begin{array}{c|c}E_6(2)&E_6\\\hline E_6&E_6(2)\end{array}\right].$

Let $X$ be a K3 surface admitting an order 3 symplectic automorphism $\sigma$, then $X/\sigma$ has 6 singularities of type $A_2$. Let $Y$ be the minimal resolution of $X/\sigma$, $M_i^{(j)}$, $i=1,2$, $j=1,\ldots, 6$ the curves resolving its six singularities, $\hat{M}:=\sum_{j=1}^6\left(M_1^{(j)}+2M_2^{(j)}\right)/3$ and  $\hat{M'}:=\sum_{j=1}^6\left(2M_1^{(j)}+M_2^{(j)}\right)/3.$ The lattice $M$, generated by $\{M_i^{(j)}, \hat{M}\}$, is the minimal primitive sublattice of $NS(Y)$ which contains the curves $M_i^{(j)}$. It is the analogue of the Nikulin lattice defined in the case of the involution and characterizes the K3 surfaces which are quotient of K3 surfaces by an order 3 automorphism, i.e. a K3 surface is the desingularization of the quotient of a  K3 surface by an order 3 automorphism if and only if $M$ is primitively embedded in its N\'eron--Severi group.

In analogy with the case of order 2, if a K3 surface admits an elliptic fibration with a 3-torsion section, the translation by this section is a symplectic automorphism of order 3 and the following holds.
\begin{proposition}\label{prop: the family UMpolarized}{\rm (See \cite[Sections 3.1.1 and 4.1]{GS1})}
	Let $X$ be a K3 surface admitting an elliptic fibration $\mathcal{E}:X\ra \mathbb{P}^1_{(t:s)}$ with a 3-torsion section. Let $\sigma$ be the translation by the 3-torsion section
	Then
	\begin{itemize} \item the Weierstrass equation of $\mathcal{E}$ is 
		\begin{eqnarray*}%\label{equation 3 torsion}
			y^2=x^3+A(\tau)x+B(\tau),~~\tau\in\PP_1,\mbox{ with }
			A(\tau)=\frac{\textstyle 6d(\tau)c(\tau)+d(\tau)^4}{\textstyle
				3},~ B(\tau)=\frac{\textstyle 27c(\tau)^2-d(\tau)^6}{\textstyle
				3^3},
		\end{eqnarray*}
		$\deg d(\tau)=2$ and $\deg c(\tau)=6$; a section of order
		three is
		$\tau\mapsto \left(\frac{d(\tau)^2}{3},\frac{d(\tau)^3}{3}+c(\tau)\right)$ and, generically, the singular fibers of  $\mathcal{E}$ are $6I_3+6I_1$ and $MW(\mathcal{E})=\Z/3\Z$;
		\item the automorphism $\sigma$ permutes the three components and the three singular points of the fibers $I_3$ and acts on the fiber of type $I_1$ preserving the  singular point;
		\item the surface $Y$, minimal resolution of $X/\sigma$, admits an elliptic fibration $\mathcal{F}$ which generically has  $6I_3+6I_1$ as singular fibers and $MW(\mathcal{F})\simeq \Z/3\Z$;
		\item $NS(X)$ primitively contains $U\oplus M$ and generically $NS(X)\simeq U\oplus M$;
		\item $NS(Y)$ primitively contains $U\oplus M$ and generically $NS(Y)\simeq U\oplus M\simeq NS(X)$.
	\end{itemize}
\end{proposition}
We will denote $\mathcal{M}$ the family of the $(U\oplus M)$-polarized K3 surfaces.
In the following we denote $X$ a generic member of $\mathcal{M}$; $F$ the class of the fiber of its elliptic fibration $\mathcal{E}$; $S$ the zero section; $M_i^{(j)}$, $i=0,1,2$ the irreducible components of the $j$-th fiber of type $I_3$; $T_1$ the 3-torsion section meeting the components $M_1^{(j)}$ and $T_2$ the 3-torsion section meeting the components of $M_2^{(j)}$.
We observe that $T_1=2F+S-\hat{M'}$, $T_2=2F+S-\hat{M}$ and that $S$ meets the component $M_0^{(j)}$ of the $j$-th $I_3$-fiber.

The action of $\sigma$ is the following: $\sigma(F)=F$, $$M_1^{(j)}\ra M_2^{(j)}\ra M_0^{(j)}=F-M_1^{(j)}-M_2^{(j)},\ \ \ S\ra T_1\ra T_2.$$

In the order 2 case, the identification of $E_8(2)$ in terms of classes of curves related with the elliptic fibration allows one to describe the cohomological action of the van Geemen--Sarti involution (see Section \ref{subsec: cohomological action} and \eqref{eq: diagE8(2)}). Analogously, in the order 3 case, to describe the cohomological action of the translation by a 3-torsion section, one needs to identify the lattice $K_{12}$. Since it is a specific overlattice of index 3 of the lattice $\widetilde{K_{12}}$, it suffices to describe the two copies of $E_6(2)$ which spans over $\Q$ the lattice $K_{12}$, which is also the orthogonal complement to the invariant sublattice $NS(X)^{\sigma}$ in $NS(X)$. Since $NS(X)^{\sigma}\simeq \langle F,S+T_1+T_2\rangle$, the two copies of $E_6(2)$  are:
{\small
\begin{align}\label{eq: diagE6(2)} \xymatrix{2F-\hat{M}\ar@{-}[r]&-F+M_1^{(1)}+M_2^{(1)}+M_2^{(2)}\ar@{-}[r]&M_2^{(3)}-M_2^{(2)}\ar@{-}[r]\ar@{-}[d]&M_2^{(4)}-M_2^{(3)}\ar@{-}[r]&M_2^{(5)}-M_2^{(4)}\\&&F-M_1^{(1)}-M_2^{(1)}+M_2^{(2)}\\ 	
-\hat{M}+\hat{M'}\ar@{-}[r]&M_2^{(1)}-M_1^{(2)}\ar@{-}[r]&-M_1^{(3)}+M_1^{(2)}\ar@{-}[r]\ar@{-}[d]&-M_1^{(4)}+M_1^{(3)}\ar@{-}[r]&-M_1^{(5)}+M_1^{(4)}\\&&2F-M_2^{(1)}-M_1^{(2)}} 	\end{align}}

The discriminant group of $NS(X)$ is generated by the following classes 
$$\eta_1:=\frac{\sum_{j=1}^3\left(M_1^{(j)}+2M_2^{(j)}\right)}{3},\ \ \eta_2:=\frac{\sum_{j=2}^4\left(M_1^{(j)}+2M_2^{(j)}\right)}{3}$$
$$\eta_3:=\frac{M_1^{(2)}+2M_2^{(2)}-M_1^{(3)}-2M_2^{(3)}}{3},\ \ \eta_4:=\frac{\sum_{j=1}^4\left(M_1^{(j)}+2M_2^{(j)}\right)+2M_1^{(5)}+M_2^{(5)}}{3}$$
The non trivial intersections among them are $\eta_1\eta_2=\frac{1}{3}$ and $\eta_3^2=\eta_4^2=\frac{2}{3}$.

The transcendental lattice of $X$ is isometric to $U\oplus U(3)\oplus A_2\oplus A_2$ and we denote $v_1,v_2, u_1,u_2, a_1,a_2,b_1,b_2$ the basis on which the intersection form is $U\oplus U(3)\oplus A_2\oplus A_2$.
As a consequence, the four classes $\eta_1+u_1/3$, $\eta_2+u_2/3$, $\eta_3+(a_1+2a_2)/3$, $\eta_4+(b_1+2b_2)/3$ are contained in $H^2(X,\Z)\simeq \Lambda_{K3}$.

The description of the lattice $\widetilde{K_{12}}$ provides an embedding $\varphi:U\oplus M\ra H^2(X,\Z)\simeq\Lambda_{K3}$, similarly to the one constructed in the order 2 case, in Table \ref{eq: U+Nin Lambda}. For example, the class $\left(e_3^{(1)}-e_3^{(2)}\right)+\left(e_3^{(1)}-e_3^{(3)}\right)+\left(e_3^{(1)}+e_3^{(2)}+e_3^{(3)}\right)=3e_i^{(1)}$ is: 3-divisible; it is the sum of two terms in $\widetilde{K_{12}}$; it is a class which is invariant for $\sigma$; it is contained in $T_X$. By \eqref{eq: diagE6(2)}, $\left(e_3^{(1)}-e_3^{(2)}\right)+\left(e_3^{(1)}-e_3^{(3)}\right)=M_2^{(3)}-M_2^{(2)}-M_1^{(3)}+M_1^{(2)}$, and by the description of the classes in $H^2(X,\Z)$ obtained gluing classes in $NS(X)$ with classes in $T_X$, we observe that $M_1^{(2)}-M_2^{(2)}-M_1^{(3)}+M_2^{(3)}+a_1+2a_2$ is 3-divisible, i.e. $e_3^{(1)}=\left(M_1^{(2)}-M_2^{(2)}-M_1^{(3)}+M_2^{(3)}+a_1+2a_2\right)/3$. By arguing in a similar way for all the $e_i^{j}$, one constructs the embedding $\varphi:U\oplus M\rightarrow \Lambda_{K3}$ such that the action of automorphism $\sigma$ induced by the 3-torsion section coincides with the ``standard" action of an order 3 symplectic automorphism (i.e. the permutation of three orthogonal copies of $E_6$ as described in \cite{GP}).

The following theorem describes specializations of $X\in\mathcal{M}$ to codimension 1 subfamilies. It is the analogue to theorems \ref{theorem: X rho 11} and \ref{theorem: Y rho 11}. It can be proved with the same strategy, which now has to be applied to the lattice $M$ instead of $N$ and to the maps $\pi^*$ and $\pi_*$ induced by the quotient map $\pi:X\ra X/\sigma$, where $\sigma$ is the translation by the 3-torsion section. These maps are explicitly described in \cite[Sections 3.3 and 3.6]{GP}.
 \begin{theorem}
	Let $X_{15}\in\mathcal{M}$ with Picard number 15.
	Then $NS(X_{15})$ is one of the following:
	\begin{itemize}
		\item $U\oplus M\oplus \langle -2d\rangle$;
		\item $\left(U\oplus M\oplus \langle -2d\rangle\right)'$ if $d\equiv 0\mod 3$, where $\left(U\oplus M\oplus \langle -2d\rangle\right)'$ is the unique overlattice of index 3 of $U\oplus M\oplus \langle -2d\rangle$ in which all the direct summands are primitively embedded.
	\end{itemize}
	Moreover, denoted $Y_{15}$ the minimal resolution of $X_{15}/\sigma$, where $\sigma$ is the automorphism of order 3 induced by  the 3-torsion section,
	\begin{itemize}
		\item $NS(X_{15})\simeq U\oplus M\oplus \langle -2d\rangle$ if and only if $NY(Y_{15})\simeq \left(U\oplus M\oplus \langle -6d\rangle\right)'$;
		\item $NS(X_{15})\simeq (U\oplus M\oplus \langle -2d\rangle)'$ (with $d\equiv 0\mod 3$) if and only if $NY(Y_{15})\simeq U\oplus M\oplus \langle -2d/3\rangle$.
\end{itemize}\end{theorem}
\begin{rem}{\rm If $NS(X_{15})\simeq U\oplus M\oplus \langle -2\rangle$, the elliptic fibration acquires a ``new" reducible fiber, which is generically of type $I_2$; if $NS(X_{15})\simeq (U\oplus M\oplus \langle -6\rangle)'$ the elliptic fibration acquires a ``new" reducible fiber, which is generically of type $I_6$; in all the other cases the elliptic fibration acquires a section of infinite order.}\end{rem}

Similarly we deduce the analogue to Theorem \ref{theorem: rank 12, equal NS}.
\begin{theorem}\label{thm order 3 rank 16}
	Let $X_{16}\in\mathcal{M}$ with Picard number 16. Let $\sigma$ be the automorphism induced by the 3-torsion section and  $Y_{16}$ the minimal resolution of $X_{16}/\sigma$.
	Then  $$NS(X_{16})\simeq \left(U\oplus M\oplus \langle -6d\rangle\right)'\oplus \langle -2e\rangle\mbox{ if and only if } NS(Y_{16})\simeq \left(U\oplus M\oplus \langle -6e\rangle\right)'\oplus \langle -2d\rangle.$$
	In particular if $d=e$ one has $NS(X_{16})\simeq NS(Y_{16})$.
\end{theorem}

In the order 2 case the family $\mathcal{N}$ of the $(U\oplus N)$-polarized K3 surfaces contains a subfamily $\mathcal{L}\subset\mathcal{N}$, such that if $X\in\mathcal{L}$, then $X$ isomorphic to the desingularization of its quotient by the van Geemen--Sarti involution. We used this result to define a map $\gamma$ acting on $(T_X)_{\Q}$, which extends to a self map of $(T_X)_{\Q}$ for every $X\in\mathcal{N}$. To construct this map we consider specializations of $X$ to a K3 surface with a van Geemen--Sarti involution with the maximal possible Picard number.

In the order 3 case, we don't know if there is a family analogous to  $\mathcal{L}$, but we can consider a special member $X_{20}$ of the family $\mathcal{M}$, which has Picard number 20 and admits an elliptic fibration $\mathcal{E}_{20}$ with a 3-torsion section, which specializes the torsion section on $\mathcal{E}$. The quotient by the translation $\sigma$ by this 3-torsion section admits a desingularization $Y_{20}=\widetilde{X_{20}/\sigma}$ which is isomorphic to $X_{20}$, i.e. we are again in the special situation in which $X_{20}\simeq Y_{20}$. In this case the quotient map $\pi_*$ induces a selfmap of $\left(T_{X_{20}}\right)_{\Q}$. This map acts also on the classes in $NS(X_{20})$ which are orthogonal to $\varphi:U\oplus M\hookrightarrow NS(X_{20})$, where $\varphi$ is the embedding constructed above. So, the map $\pi_*$ can be considered as a self map of the orthogonal complement to $\varphi(U\oplus M)$ in $H^2(X,\Z)$ and we denote it as $\gamma$ when considered as abstract map on lattices. Therefore, in analogy with Theorem \ref{theorem: action selfmap} and Corollary \ref{cor: NSX=NSY condition}, we obtain the following. 

\begin{theorem}\label{thm: order 3 "complex multiplication"}
	Let $X\in\mathcal{M}$, $\varphi:U\oplus M\hookrightarrow NS(X_{20})$ the embedding given by \eqref{eq: diagE6(2)} and $\Gamma:=\left(\langle -6\rangle\oplus \langle -2\rangle\right)^{\oplus 2}\bigoplus\left(\langle -12\rangle\oplus \langle-4\rangle\right)\bigoplus \left(\langle 12\rangle\oplus \langle4\rangle\right)$. Then there exists a basis of $(\varphi(U\oplus M)^{\perp_{H^2(X,\Z)}})_{\Q}$ on which the bilinear form is isometric to $\Gamma$
	  and the map $\gamma:\Gamma\ra \Gamma$ induced by $\pi_*$, where $\pi$ is the rational quotient map $X\ra Y=\widetilde{X/\sigma}\simeq X$, acts as  $\left[\begin{array}{rr}0&1\\3&0\end{array}\right]^{\oplus 3}\oplus \gamma_+$ with $\gamma_+:\langle 12\rangle\oplus \langle4\rangle\ra \langle 12\rangle\oplus \langle4\rangle$.
\end{theorem}

\begin{corollary}\label{cor: order 3}
	Let $Z_{2b}$ be a K3 surface with $\rho(Z_{2b})=2b$ and $Z_{2b}\in\mathcal{M}$, so that it admits an order 3 automorphism induced by the translation by a 3-torsion section $\sigma_Z$ which specializes $\sigma$ on $X$. If $T_{Z_{2d}}$ is invariant for (the $\Q$-linear extension of) $\gamma$, then $NS(Z_{2d})\simeq NS(\widetilde{Z_{2d}/\sigma_Z})$ and $T_{Z_{2d}}\simeq T_{\widetilde{Z_{2d}/\sigma_Z}}$. \end{corollary}

The proofs are straightforward, once one has the right specializations of the surface $X$, since they are analogue to the ones of the order 2 case. Therefore, it suffices to exhibit the specializations, which depend on the choices of the elements $V_i$ and $W_i$ in $T_X$ which become algebraic at each step. The results on the specializations, which are the analogue to Proposition \ref{prop: the five specializations}, are summarized in the following Proposition, where $X_{k}$ is a K3 surface with a symplectic  automorphism $\sigma$ which specializes the one on $X$ and $Y_{k}$ is the desingularization of the quotient $X_k/\sigma$.

\begin{proposition}\label{prop: the three specializations in order 3}
	There exist the following specializations of $X$:
	\begin{itemize}\item for $i=1,2$, there is a $(6-2i)$-dimensional family of K3 surfaces whose generic members $X_{14+2i}$ is such that $\rho(X_{14+2i})=14+2i$ and $NS(X_{14+2i})\simeq NS(Y_{14+2i})$; the elliptic fibration $\mathcal{E}_{14+2i}$ on it has $iI_6+(6-2i)I_3+iI_2+(6-2i)I_1$ as singular fibers and $\Z/3\Z$ as Mordell--Weil group.
		\item $X_{20}$ is a K3 surface with  $\rho(X_{20})=20$  and $X_{20}\simeq Y_{20}$; the elliptic fibration $\mathcal{E}_{20}$ on it has $I_{12}+2I_3+I_4+2I_1$ as singular fibers and $\Z/3\Z$ as Mordell--Weil group. Its transcendental lattice is $\langle 12\rangle\oplus\langle 4\rangle$.
	\end{itemize}
\end{proposition}

{\bf The first specialization.} We require that the classes
$$V_1:=v_1-v_2,\ \ W_1=u_1-u_2$$
become algebraic and we observe that $V_1^2=-2$ and $W_1^2=-6$.

Since $\frac{(u_1-u_2)}{3}+\eta_1-\eta_2$ is contained in $H^2(X,\Z)$ and $u_1-u_2+3\eta_1-3\eta_2\in NS(X_{16})$, it follows that $\frac{(u_1-u_2)}{3}+\eta_1-\eta_2=\left(u_1-u_2+C_1^{(1)}+2C_2^{(1)}-C_1^{(4)}-2C_2^{(4)}\right)/3\in NS(X_{16})$.

The class $\lambda_1:=\frac{(u_1-u_2)}{3}+\eta_1-\eta_2$ has self intersection $-2$ and it is easy to check that $\{-C_1^{(1)}, -C_2^{(1)}, \lambda_1, C_2^{(4)}, C_1^{(4)}\}$ forms a standard basis of $A_5$.

Therefore, by adding these classes we are gluing two fibers of type $I_1$ to obtain a fiber of type $I_2$ and two fibers of type $I_3$ to obtain a fiber of type $I_6$ (as in Lemma \ref{lemma: gluing An-1}). The torsion section $T_1$ is trivial on the ``new" $I_2$-fiber and non trivial on the $I_6$. Hence, the quotient map switches these two ``new" fibers and $\gamma(V_1)=W_1$, $\gamma(W_1)=3V_1$ (the construction is similar to the one of Example \ref{rem: first specialization and action of complex multiplication})

So we constructed $X_{16}$ endowed with an elliptic fibration (which specializes $\mathcal{E}$) with $I_6+4I_3+I_2+4I_1$ as singular fibers and with Mordell--Weil group equals to $\Z/3\Z$.

{\bf The second specialization.}  
Let $X_{18}$ be the K3 surface obtained by $X_{16}$ by requiring that the classes 
$$V_2:=b_1,\ \ W_2:=a_1+2a_2$$
become algebraic classes, where $V_2^2=-2$ and $W_2^2=-6$.
Since $\frac{(a_1+2a_2)}{3}+\eta_3$ is contained in $H^2(X,\Z)$ it follows that $\frac{(a_1+2a_2)}{3}+\eta_3=\left(a_1+2a_2+C_1^{(2)}+2C_2^{(2)}-C_1^{(3)}-2C_2^{(3)}\right)/3\in NS(X_{18})$.
As in the previous case we glue two fibers of type $I_1$ to a fiber of type $I_2$  and two fibers of type $I_3$ to one of type $I_6$.

As a consequence $X_{18}$ admits an elliptic fibration with $2I_6+2I_3+2I_2+2I_1$ as singular fibers and with Mordell--Weil group $\Z/3\Z$.
The action of $\gamma$ is as above.

{\bf The third specialization} is obtained gluing two fibers of type $I_2$ to a fiber of type $I_4$ and two fibers of type $I_6$ to a fiber of type $I_{12}$. Since $T_1$ is trivial on the first ``new" fiber and not on the second, the quotient by $\sigma$ switches these two fibers.

Let $X_{20}$ be the K3 surface obtained by $X_{18}$ by requiring that the classes 
$$V_3:=v_1+v_2+b_1+2b_2,\ \ W_3:=u_1+u_2+3a_1$$
become algebraic classes.
Since $V_3+V_1+V_2=v_1+v_2+b_1+2b_2+v_1-v_2+b_1=2v_1+2b_1+2b_2$, $\mu_1:=(V_1+V_2+V_3)/2=v_1+b_1+b_2\in NS(X_{20})$ and the classes $\{-V_1,\mu_1,-V_2\}$ form a standard basis of $A_3$. 

Similarly the class
$$\mu_2:=\left(-C_1^{(2)}-2C_2^{(2)}+3\lambda_1+4C_2^{(4)}+5C_1^{(4)}+W_3-C_1^{(2)}-2C_2^{(2)}+3\lambda_2+4C_2^{(3)}+5C_1^{(3)}\right)/6\in NS(X_{20})$$ and the classes 
$\{C_1^{(2)}, C_2^{(2)},-\lambda_1, -C_2^{(4)}, -C_1^{(4)},\mu_2, -C_1^{(3)},C_2^{(3)},-\lambda_2, C_1^{(1)}, C_2^{(1)}\}$ form a standard basis of $A_{11}$. 

The transcendental lattice of $X_{20}$ is the orthogonal to $\langle V_i, W_1\rangle_{i=1,2,3}$ in $T_{X}$ and it is generated by $3v_1+3v_2+b_1+2b_2$ and $u_1+u_2+b_1$ so its intersection form is $\langle 12\rangle\oplus \langle 4\rangle.$
This implies that $|d(T_{X_{20}})|=|d(NS(X_{20}))|=2^4\cdot3$  and so the Mordell--Weil group is $\Z/3\Z$.

\begin{remark}
	{\rm One can consider also different specializations, which produce different decomposition in direct summands similar to the one in Theorem \ref{thm: order 3 "complex multiplication"}. One of them, very natural, consists in requiring (instead of the third specialization described above) that the classes 
	$$V_3':=a_1,\ \ W_3':=b_1+2b_2$$
	become algebraic classes. Since $V_3'^2=-2$ and $W_3'^2=-6$, as in the first and second specializations, one is gluing two fibers of type $I_1$ to a fiber of type $I_2$  and two fibers of type $I_3$ to one of type $I_6$, so that one obtains a fibration with $3I_6+3I_2$ as singular fiber. Also in this case one obtains that the K3 surface with this elliptic fibration is isomorphic to the ones which is the desingularization of its quotient. The unique difference with respect to the previous specialization is that in this case the Mordell--Weil group changes a lot, indeed it becomes $\Z/6\Z\times \Z/2\Z$, see e.g. \cite{SZ}.}
\end{remark}

\end{document}